\documentclass[11pt]{article}
 
\usepackage[margin=1in]{geometry} 
\usepackage{amsmath,amsthm,amssymb, mathtools}

\usepackage[T1]{fontenc} 
\usepackage[utf8]{inputenc} 
\usepackage{lmodern} 
\usepackage{graphicx}
\usepackage{float}
\usepackage{caption}
\usepackage{subcaption}
\graphicspath{ {images/} }
\usepackage[colorlinks=true, linkcolor=blue, citecolor=blue]{hyperref} 
\usepackage[shortlabels]{enumitem}
\usepackage{tcolorbox}
\usepackage{verbatim}
\usepackage{array}
\usepackage{xcolor}
\usepackage[backend=biber, giveninits=true, style=numeric, citestyle=numeric-comp, sorting=nyt, maxbibnames=99]{biblatex}
\renewbibmacro{in:}{}
\DeclareFieldFormat[article]{title}{#1} 
\DeclareFieldFormat[article]{number}{(#1)}
\DeclareFieldFormat[article]{pages}{#1}  
\renewbibmacro*{volume+number+eid}{%
  \printfield{volume}%
  \printfield{number}%
  \setunit{\addcomma\space}%
  \printfield{eid}}
\renewbibmacro*{journal+issuetitle}{%
  \usebibmacro{journal}%
  \setunit*{\addspace}%
  \iffieldundef{series}
    {}
    {\newunit
     \printfield{series}%
     \setunit{\addspace}}%
  \usebibmacro{volume+number+eid}%
  \setunit{\bibpagespunct}%
  \printfield{pages}%
  \setunit{\addspace}%
  \usebibmacro{issue+date}%
  \setunit{\addcolon\space}%
  \usebibmacro{issue}%
  \newunit}

\renewbibmacro*{note+pages}{%
  \printfield{note}%
  \newunit}
\newcommand{\N}{\mathbb{N}}

\newcommand{\R}{\mathbb{R}}

\newcommand{\sH}{\mathcal{H}}
 
\newtheorem{theorem}{Theorem}
\numberwithin{theorem}{section}

\newtheorem{lemma}[theorem]{Lemma}

\numberwithin{fact}{subsection}

\newtheorem*{theorem*}{Theorem}
\newtheorem*{lemma*}{Lemma}
\newtheorem*{corollary*}{Corollary}

\theoremstyle{definition}

\newtheorem{remark}[theorem]{Remark}
\newtheorem{example}[theorem]{Example}
\newtheorem*{remark*}{Remark}

\DeclareMathOperator{\id}{Id}

\allowdisplaybreaks

\title{Fast second-order dynamics with slow vanishing damping approaching the zeros of a monotone and continuous operator}
\author{
    Radu Ioan Bo\c{t}\footnote{Faculty of Mathematics, University of Vienna, Oskar-Morgenstern-Platz 1, 1090 Vienna, Austria, e-mail: \url{radu.bot@univie.ac.at}. Research partially supported by FWF (Austrian Science Fund), project W 1260, and by a grant of the Romanian Ministry of Research, Innovation and Digitization, CNCS -
UEFISCDI, project number PN-III-P1-1.1-TE-2021-0138, within PNCDI III.} 
    \and
    David Alexander Hulett\footnote{Faculty of Mathematics, University of Vienna, Oskar-Morgenstern-Platz 1, 1090 Vienna, Austria, e-mail: \url{david.alexander.hulett@univie.ac.at}. Research partially supported by the VGSCO (Vienna Graduate School on Computational Optimization), project W 1260.}
    \and
    Dang-Khoa Nguyen\footnote{Faculty of Mathematics and Computer Science, University of Science, Ho Chi Minh City, Vietnam, e-mail: \url{ndkhoa@hcmus.edu.vn}}
    \footnote{Vietnam National University, Ho Chi Minh City, Vietnam}
    }
\date{}
\begin{document}
 

\maketitle

\begin{abstract}
    In this work, we approach the problem of finding the zeros of a continuous and monotone operator through a second-order dynamical system with a damping term of the form $1/t^{r}$, where $r\in [0, 1]$. The system features the time derivative of the operator evaluated along the trajectory, which is a Hessian-driven type damping term when the governing operator comes from a potential. Also entering the system is a time rescaling parameter $\beta(t)$ which satisfies a certain growth condition. We derive $o\left(\frac{1}{t^{2r}\beta(t)}\right)$ convergence rates for the norm of the operator evaluated along the generated trajectories as well as for a gap function which serves as a measure of optimality for the associated variational inequality. The parameter $r$ enters the growth condition for $\beta(t)$: when $r < 1$, the damping $1/t^{r}$ approaches zero at a slower speed than Nesterov's $1/t$ damping; in this case, we are allowed to choose $\beta(t)$ to be an exponential function, thus having linear convergence rates for the involved quantities. We also show weak convergence of the trajectories towards zeros of the governing operator. Through a particular choice for the operator, we establish a connection with the problem of minimizing a smooth and convex function with linear constraints. The convergence rates we derived in the operator case are inherited by the objective function evaluated at the trajectories and for the feasibility gap. We also prove weak convergence of the trajectories towards primal-dual solutions of the problem. 
    
    A discretization of the dynamical system yields an implicit algorithm that exhibits analogous convergence properties to its continuous counterpart. 

    We complement our theoretical findings with two numerical experiments.
\end{abstract}


\tableofcontents

\section{Introduction}
\subsection{Problem statement and motivation}
In the setting of a real Hilbert space $\mathcal{H}$ and a monotone and continuous operator $V :\mathcal{H} \to \mathcal{H}$, we study the following problem:
\begin{equation}\label{eq:monotone equation}
    \text{find }z_{*}\in \mathcal{H}\text{ such that } V(z_{*}) = 0. 
\end{equation}
It is simple to see that the continuity and monotonicity of $V$ ensure that $z_{*}$ satisfies \eqref{eq:monotone equation} if and only if 
\begin{equation}\label{eq:variational inequality}
    \langle z - z_{*}, V(z)\rangle \geq 0 \quad \forall z\in \mathcal{H}. 
\end{equation}
One of the principal motivations to study the monotone inclusion \eqref{eq:monotone equation} comes from minimax problems. Indeed, consider
\begin{equation}\label{eq:minmax}
    \min_{x\in \mathcal{X}}\max_{y\in \mathcal{Y}} \Phi(x, y), 
\end{equation}
where $\mathcal{X}, \mathcal{Y}$ are real Hilbert spaces and $\Phi \colon \mathcal{X} \times \mathcal{Y} \to \R$ is continuously differentiable, convex in the first variable and concave in the second one. Solutions to \eqref{eq:minmax} are saddle points of $\Phi$: that is, a pair $(x_{*}, y_{*})\in \mathcal{X} \times \mathcal{Y}$ such that 
\[
    \Phi(x_{*}, y) \leq \Phi(x_{*}, y_{*}) \leq \Phi(x, y_{*}) \quad \forall (x, y) \in \mathcal{X} \times \mathcal{Y}. 
\]
This is equivalent to 
\[
    \begin{dcases}
        \nabla_{x} \Phi(x_{*}, y_{*}) &= 0, \\
        -\nabla_{y} \Phi(x_{*}, y_{*}) &= 0,
    \end{dcases}
\]
and this is nothing else than a monotone inclusion problem involving the monotone and continuous operator $V : \mathcal{X} \times \mathcal{Y} \to \mathcal{X} \times \mathcal{Y}$ given by 
\[
    V(x, y) := \Bigl( \nabla_{x} \Phi(x, y), -\nabla_{y} \Phi(x, y)\Bigr).
\]
Formulations \eqref{eq:monotone equation} and \eqref{eq:minmax} underlie numerous problems in different fields, such as optimization, economics, game theory, and partial differential equations, and are of particular interest to the machine learning community: for example, they play a fundamental role in areas such as multi-agent reinforcement learning \cite{ReinforcementLearning}, robust adversarial learning \cite{RobustLearning} and generative adversarial networks (GANs) \cite{GANs, GANs2}. An interesting example of \eqref{eq:minmax}, upon which we will elaborate further later, is linearly constrained convex minimization. Precisely, let us consider 
\begin{equation}\label{eq:constrained minimization}
    \begin{array}{rl}
		\min & f \left( x \right), \\
		\text{subject to} 	& Ax = b, 
	\end{array}
\end{equation}
where $f : \mathcal{X} \to \R$ is convex and continuously differentiable, $b\in \mathcal{Y}$ and $A : \mathcal{X} \to \mathcal{Y}$ is a bounded linear operator. Associated to \eqref{eq:constrained minimization} is the Lagrangian $\mathcal{L} : \mathcal{X} \times \mathcal{Y} \to \R$ given by
\[
    \mathcal{L}(x, y) := f(x) + \langle \lambda, Ax - b\rangle.
\]
Primal-dual solutions to \eqref{eq:constrained minimization}, that is, pairs $(x_{*}, \lambda_{*}) \in\mathcal{X} \times \mathcal{Y}$ which satisfy 
\[
    \begin{dcases}
        \nabla f(x_{*}) + A^{*}\lambda_{*} &= 0, \\
        b - Ax_{*} &= 0 
    \end{dcases}
\]
are precisely the saddle points of $\mathcal{L}$, i.e., the zeros of the monotone and continuous operator
\begin{equation}\label{eq:operator for constrained minimization}
    V(x, \lambda) := \Bigl( \nabla_{x} \mathcal{L}(x, y), -\nabla_{\lambda} \mathcal{L}(x, \lambda)\Bigr) = \Bigl( \nabla f(x) + A^{*}\lambda, b - Ax\Bigr).
\end{equation}
Attached to \eqref{eq:monotone equation}, we will investigate the asymptotic properties of the trajectories generated by a certain second-order dynamical system. The system features a vanishing damping of the form $\frac{\alpha}{t^{r}}$, where $r\in [0, 1]$, which approaches zero more slowly than Nesterov's classical $\frac{\alpha}{t}$ when $r < 1$. This is combined with the time derivative of the operator evaluated along the trajectory, which is also known as a Hessian-driven damping term when $V$ is the gradient of a continuously differentiable function. The convergence behaviour of the solutions is greatly affected by the presence of a time rescaling parameter $\beta(t)$ which is positive, nondecreasing, and continuously differentiable, and needs to satisfy a certain growth condition. For the generated trajectory $z(t)$, we will derive rates of convergence of $o\left(\frac{1}{t^{2r}\beta(t)}\right)$ for $\left\lVert V(z(t)) \right\rVert$ and for the restricted gap function associated to \eqref{eq:variational inequality} as $t\to +\infty$. When $r < 1$, we are allowed to choose $\beta(t)$ such that $\beta(t)$ grows exponentially with time; such a choice is not possible when $r = 1$, i.e., when we take Nesterov's damping. We will also show weak convergence of $z(t)$ towards a zero of $V$ as $t\to +\infty$. 

When $V$ is chosen to be \eqref{eq:operator for constrained minimization}, we recover a primal-dual system formulated in the spirit of He et al. \cite{HeHuFang}. However, it differs in the sense that our primal-dual system includes a Hessian-driven damping term. For the generated primal-dual trajectory $(x(t), \lambda(t))$, we will show convergence rates of $o\left(\frac{1}{t^{2r}\beta(t)}\right)$ for the functional values along the primal trajectory, the primal-dual gap formulated in terms of the Lagrangian $\mathcal{L}$, and for the feasibility gap which gauges how far the generated primal trajectory is from the constraints. Additionally, we will furnish the weak convergence of $(x(t), \lambda(t))$ towards a primal-dual solution of the minimization problem as $t\to +\infty$. 

A temporal discretization of our system will produce an implicit algorithm that mimics the convergence properties of the continuous time system. If $(\beta_{k})_{k\geq 1}$ is a positive, nondecreasing sequence which satisfies an analogous growth condition to the continuous one, and $(z^{k})_{k\geq 1}$ are the iterates generated by the algorithm, we will produce convergence rates of $o\left(\frac{1}{k^{2r} \beta_{k}}\right)$ for $\left\lVert V(z^{k}) \right\rVert$ and for the discrete gap function associated to \eqref{eq:variational inequality} as $k\to +\infty$. Similar to the continuous case, $r < 1$ allows us to choose $\beta_{k}$ such that $\beta_{k}$ grows exponentially as $k\to +\infty$. Furthermore, we will show that the iterates converge weakly to a zero of $V$. 

\subsection{Previous and related continuous time systems}
In the last years, there have been many advances in the study of continuous time systems attached to monotone inclusion problems. We briefly visit them in the following paragraphs. 

Extending the Heavy ball with friction dynamics introduced by Álvarez in \cite{Alvarez} for unconstrained minimization, Álvarez and Attouch \cite{AlvarezAttouch} and Attouch and Maingé \cite{AttouchMainge} studied the dynamics 
\begin{equation}\label{eq:HBF}
    \Ddot{z}(t) + \mu \Dot{z}(t) + A(z(t)) = 0, 
\end{equation}
where $A \colon \mathcal{H} \to \mathcal{H}$ is a $\lambda$-cocoercive operator and $\mu > 0$. The authors showed that under the assumption $\lambda \mu^{2} > 1$, a solution to \eqref{eq:HBF} weakly converges to a zero of $A$. Recall that for a maximally monotone (but not necessarily single-valued) operator $A \colon \mathcal{H} \to 2^{\mathcal{H}}$, its Yosida approximation of index $\lambda > 0$ is given by 
\[
    A_{\lambda} := \frac{1}{\lambda} \bigl( \id - J_{\lambda A}\bigr), 
\]
where $J_{\lambda A} := \bigl(\id + \lambda A\bigr)^{-1}$ is the resolvent of $\lambda A$. The operator $A_{\lambda}$ is single-valued, $\lambda$-cocoercive, and shares the same set of zeros of $A$. Therefore, if $\lambda \mu^{2} > 1$, then any solution to 
\[
    \Ddot{z}(t) + \mu \Dot{z}(t) + A_{\lambda}(z(t)) = 0
\]
weakly converges to a zero of $A$. 

Related to \eqref{eq:HBF}, Bo\c{t} and Csetnek \cite{BotCsetnek} studied the system 
\begin{equation}
    \Ddot{z}(t) + \gamma(t) \Dot{z}(t) + \nu(t) A(z(t)) = 0, 
\end{equation}
where $A \colon  \mathcal{H} \to \mathcal{H}$ is again $\lambda$-cocoercive. Under the assumption that $\gamma$ and $\nu$ are locally absolutely continuous, $\Dot{\gamma}(t) \leq 0\leq \Dot{\nu}(t)$ for almost every $t\in [0, +\infty)$ and $\inf_{t\geq 0} \frac{\gamma^{2}(t)}{\nu(t)} > \frac{1}{\lambda}$, the solutions to this system converge weakly to zeros of $A$. 

Linked with the Newton and Levenberg-Marquardt methods, in \cite{AttouchSvaiter} Attouch and Svaiter studied the following first-order differential inclusion:
\begin{equation}
    \begin{dcases}
        v(t) \in A(z(t)), \\
        \Dot{z}(t) + \beta(t) \Dot{v}(t) + \beta(t) v(t) = 0,
    \end{dcases}
\end{equation}
where $A : \mathcal{H} \to 2^{\mathcal{H}}$ is maximally monotone and $\beta$ is positive and locally absolutely continuous. Under the assumption $\sup_{t\geq 0} \frac{\Dot{\beta}(t)}{\beta(t)} \leq 1$, the authors prove that the solutions to this system converge weakly to zeros of $A$. 

Ever since the landmark work by Su et al. in \cite{SuBoydCandes}, where the authors provided a continuous time counterpart of Nesterov's accelerated gradient algorithm, including a damping term $\frac{\alpha}{t}$ attached to the velocity has been widely successful in having an accelerating behaviour for convex minimization problems in the unconstrained (see, for example, Attouch et al. \cite{AttouchPeypouquetRedont}) and linearly constrained settings (see He et al. \cite{HeHuFang, HeHuFang2}, Zeng et al. \cite{ZengLeiChen}, Attouch et al. \cite{ADMM}, Bo\c t and Nguyen \cite{BotNguyen}). Interestingly, the effects of this vanishing damping go beyond the optimization setting: for $A :\mathcal{H} \to 2^{\mathcal{H}}$ maximally monotone and $t\geq t_{0} > 0$, in \cite{AttouchPeypouquet} Attouch and Peypouquet addressed the Nesterov-like system 
\begin{equation*}
    \Ddot{z}(t) + \frac{\alpha}{t}\Dot{z}(t) + A_{\lambda(t)}(z(t)) = 0, 
\end{equation*}
where $\alpha > 1$ and $\lambda(t) = \mathcal{O}(t^{2})$ as $t\to +\infty$. Solutions to this system converge weakly to zeros of $A$. Furthermore, the velocity $\Dot{z}(t)$ and the acceleration $\Ddot{z}(t)$ vanish with rates $\mathcal{O}\left(\frac{1}{t}\right)$ and $\mathcal{O}\left(\frac{1}{t^{2}}\right)$ as $t\to +\infty$ respectively. Preserving these convergence properties, Attouch and László \cite{AttouchLaszlo} introduced a Hessian-driven damping term to the previous system: 
\begin{equation}\label{eq:AVD operators}
    \Ddot{z}(t) + \frac{\alpha}{t}\Dot{z}(t) + \xi \frac{d}{dt}\left( A_{\lambda(t)}(z(t))\right) + A_{\lambda(t)}(z(t)) = 0, 
\end{equation}
where $\xi \geq 0$. The introduction of this term has two effects. First, if $t\mapsto z(t)$ is a solution to \eqref{eq:AVD operators}, it allows to derive the rates 
\[
    \left\lVert \Dot{z}(t) \right\rVert = o\left(\frac{1}{t}\right), \quad \left\lVert A_{\lambda(t)}(z(t)) \right\rVert = o\left( \frac{1}{t^{2}}\right) \quad \text{and} \quad \left\lVert \frac{d}{dt} \left(A_{\lambda(t)}(z(t))\right) \right\rVert = \mathcal{O}\left(\frac{1}{t^{3}}\right)
\]
as $t\to +\infty$. Second, in practice, when $\xi = 0$ (meaning no Hessian-driven damping) the trajectories present an oscillatory behaviour that is attenuated as soon as $\xi > 0$ is chosen. 

For a continuous and monotone operator $V : \mathcal{H} \to \mathcal{H}$ and $t \geq t_{0} > 0$, Bo\c t et al. \cite{fOGDA} considered the second order dynamics 
\begin{equation}\label{eq:fOGDA}
    \Ddot{z}(t) + \frac{\alpha}{t}\Dot{z}(t) + \beta(t) \frac{d}{dt}V(z(t)) + \frac{1}{2}\left(\Dot{\beta}(t) + \frac{\alpha}{t}\beta(t)\right) = 0, 
\end{equation}
where $\alpha > 0$ and $\beta$ is positive, nondecreasing and continuously differentiable. Assuming $\alpha > 2$ and the growth condition
\begin{equation}\label{eq:growth condition intro}
    \sup_{t\geq t_{0}} \frac{t \Dot{\beta}(t)}{\beta(t)} < \alpha - 2, 
\end{equation}
for $z_{*}$ a zero of $V$ and $t\mapsto z(t)$ a solution to the authors produce the convergence rates
\[
    \bigl\| \Dot{z}(t)\bigr\| = o \left( \frac{1}{t}\right), \quad \| V(z(t))\| = o\left( \frac{1}{t\beta(t)}\right) \quad \text{and} \quad \langle z(t) - z_{*}, V(z(t))\rangle = o\left(\frac{1}{t\beta(t)}\right)
\]
as $t\to +\infty$. Furthermore, the trajectories converge weakly to zeros of $V$. 

\subsection{Our work}\label{subsec:our work}
In the context of linearly constrained convex minimization, for a primal-dual system like that in Zeng et al. \cite{ZengLeiChen}, Attouch et al. \cite{ADMM} considered a vanishing damping term of the form $\frac{\alpha}{t^{r}} + \frac{r}{t}$, while He et al. \cite{HeHuFang} measured the effects of a term $\frac{\alpha}{t^{r}}$. In both cases, the choice $r < 1$ allowed the authors to derive exponential rates of convergence for the primal-dual gap. We wanted to see if this influence carries over to a more general setting of a continuous and monotone operator $V \colon \mathcal{H} \to \mathcal{H}$. In the following, we state the system and the hypotheses on its parameters, which will be assumed throughout Section 2. 

\begin{tcolorbox}
    \textbf{Continuous time scheme}: for $t \geq t_{0} > 0$, we will study the asymptotic properties of the solutions to
    \begin{equation}\label{eq:system r=s in (0, 1)}
        \Ddot{z}(t) + \frac{\alpha}{t^{r}} \Dot{z}(t) + \theta t^{r} \beta(t) \frac{d}{dt}V(z(t)) + \beta(t) V(z(t)) = 0. 
    \end{equation}
    Here, $r\in [0, 1]$, and $\alpha, \theta > 0$ are taken as follows: 
    \begin{itemize}
        \item If $r\in [0, 1)$, then $\frac{2}{\alpha} < \theta$; 
        \item If $r = 1$, then $\frac{2}{\alpha + 1} \leq \theta < \frac{1}{2}$. 
    \end{itemize}
    $\beta: [t_{0}, +\infty)\to (0, +\infty)$ is a continuously differentiable and nondecreasing function which satisfies the following growth condition: 
    \begin{equation}\label{eq:growth condition case r=s}
        \sup_{t \geq t_{0}} t^{r} \left( \frac{\Dot{\beta}(t)}{\beta(t)} + \dfrac{2r}{t} \right) < \frac{1}{\theta} . 
    \end{equation}
\end{tcolorbox}
For $z_{*}$ an arbitrary zero of $V$ and $t\mapsto z(t)$ a solution to \eqref{eq:system r=s in (0, 1)}, we will show that 
\[
     \| V(z(t))\| = o\left(\frac{1}{t^{\rho + r} \beta(t)}\right), \quad \langle z(t) - z_{*}, V(z(t)) \rangle = o\left(\frac{1}{t^{2\rho}\beta(t)}\right) \quad \text{and} \quad \bigl\| \Dot{z}(t)\bigr\| = o\left(\frac{1}{t^{\rho}}\right)
\]
as $t\to +\infty$, where $\rho = 0$ or $\rho = 1$ if $r = 0$ or $r = 1$ respectively, and $\rho \in (0, r)$ when $r\in (0, 1)$. This allows us to get as close to the state-of-the-art rates as possible: if we assume that $z(t)$ is bounded when $r\in (0, 1)$, we are able to plug $\rho = r$. The weak convergence of $z(t)$ towards a zero of $V$ can be guaranteed in the cases $r = 0$ and $r = 1$, and in the case $r\in (0, 1)$ provided $z(t)$ is assumed bounded.
\begin{remark}
    When $r < 1$, the growth condition \eqref{eq:growth condition case r=s} holds if for some $0 < \delta < \frac{1}{\theta}$ the following differential equation is fulfilled by $\beta$:
    \[
        t^{r} \left( \frac{\Dot{\beta}(t)}{\beta(t)} + \frac{2r}{t}\right) = \frac{1}{\theta} - \delta \quad \forall t\geq t_{0}. 
    \]
    A solution to this equation is 
    \[
        \beta(t) = \frac{1}{t^{2r}} \cdot e^{\left(\frac{1}{\theta} - \delta\right) \frac{t^{1 - r}}{1 - r}},
    \]
    thus convergence rates of $o\left(e^{-\left(\frac{1}{\theta} - \delta\right) \frac{t^{1 - r}}{1 - r}}\right)$ are attained for $\| V(z(t))\|$ and $\langle z(t) - z_{*}, V(z(t))\rangle$ as $t\to +\infty$. To our knowledge, this is the first time such fast rates are achieved in the general context of equations governed by monotone and continuous operators.     
\end{remark}

A discretization of \eqref{eq:system r=s in (0, 1)} yields an implicit algorithm, which has analogous convergence properties to its continuous counterpart. We will only consider $r\in (0, 1]$.
\begin{tcolorbox}
    \textbf{Discrete time scheme}: let $z^{0}, z^{1}$ be initial points in $\mathcal{H}$. We iteratively define for all $k \geq 1$
    \begin{equation}\label{eq:algorithm intro}
        \begin{split}
            z^{k + 1} := & \: z^{k} + \frac{k^{r}}{\alpha - rk^{r - 1} + (k + 1)^{r}} \bigl( z^{k} - z^{k - 1}\bigr) - \frac{\theta k^{2r} \beta_{k - 1}}{\alpha - rk^{r - 1} + (k + 1)^{r}} \Bigl[ V(z^{k + 1}) - V(z^{k})\Bigr] \\
            &- \frac{\theta \Bigl\{ \Bigl[ (k + 1)^{2r} - k^{2r}\Bigr] - 2rk^{2r - 1}\Bigr\} + k^{r}}{\alpha - rk^{r - 1} + (k + 1)^{r}} \beta_{k} V(z^{k + 1}).
        \end{split}
    \end{equation}
    Here, $r\in (0, 1]$, $\alpha, \theta > 0$ are taken as follows:
    \begin{itemize}
        \item If $r\in (0, 1)$, then $\frac{2}{\alpha} < \theta$; 
        \item If $r = 1$, then $\frac{2}{\alpha + 1} \leq \theta < \frac{1}{4}$. 
    \end{itemize}
    The sequence $(\beta_{k})_{k\geq 0}$ is positive, nondecreasing, and it satisfies
    \begin{equation}\label{eq:discrete growth intro}
        \sup_{k\geq k_{0}}k^{r} \left(\frac{\beta_{k} - \beta_{k - 1}}{\beta_{k}} + \frac{2r}{k}\right) < \frac{1}{2\theta}
    \end{equation}
    for some positive integer $k_{0}$. 
\end{tcolorbox}
For $z_{*}$ an arbitrary zero of $V$ and $(z^{k})_{k\geq 0}$ given by \eqref{eq:algorithm intro}, we will prove that 
\[
    \| V(z^{k})\| = o\left(\frac{1}{k^{\rho + r} \beta_{k}}\right), \quad \langle z^{k} - z_{*}, V(z^{k})\rangle = o\left(\frac{1}{k^{2\rho} \beta_{k}}\right) \quad \text{and} \quad \bigl\| z^{k} - z^{k - 1}\bigr\| = o\left(\frac{1}{k^{\rho}}\right)
\]
as $k\to +\infty$, where $\rho = 1$ if $r = 1$, and $\rho \in (0, r)$ when $r\in (0, 1)$. Like in the continuous case, this allows us to get as close to the state-of-the-art rates as possible: if we assume that $(z^{k})_{k\geq 0}$  is bounded when $r\in (0, 1)$, we are able to plug $\rho = r$. The weak convergence of $(z^{k})_{k\geq 0}$  towards a zero of $V$ can be guaranteed in the cases $r = 0$ and $r = 1$, and in the case $r\in (0, 1)$ provided $(z^{k})_{k\geq 0}$  is assumed bounded.
\begin{remark}
    Similar to the continuous case, when $r\in (0, 1)$, we will later show that for some $0 < \delta < \frac{1}{2\theta}$ the sequence 
    \[
        \beta_{k} = \frac{1}{k^{2r}} \cdot e^{\left(\frac{1}{2\theta} - \delta\right) \frac{k^{1 - r}}{1 - r}}
\ \forall k \geq k_0    \]
    fulfills the discrete growth condition \eqref{eq:discrete growth intro}. Hence, a convergence rate of $o\left(e^{-\left(\frac{1}{2\theta} - \delta\right) \frac{k^{1 - r}}{1 - r}}\right)$ is exhibited by $\| V(z^{k})\|$ and $\langle z^{k} - z_{*}, V(z^{k})\rangle$ as $k\to +\infty$. To our knowledge, this is the first time such fast rates have been achieved in the general context of implicit algorithms governed by monotone and continuous operators. 
\end{remark}

\section{Continuous time system}

\subsection{Convergence rates and weak convergence of trajectories}
\textbf{Energy function}: to analyze the asymptotic properties of the solutions to \eqref{eq:system r=s in (0, 1)}, we will make use of the following energy function: 
\begin{align}
    \mathcal{E}_{\lambda}(t) = & \: \frac{1}{2} \Bigl\| 2\lambda t^{\rho - r}(z(t) - z_{*}) + 2t^{\rho}\Dot{z}(t) + \theta t^{\rho + r} \beta(t) V(z(t))\Bigr]\Bigr\|^{2} \hypertarget{eq:energy function r=s line 1}{\label{eq:energy function r=s line 1}} \\
    &+ 2\lambda t^{2(\rho - r)}\bigl(\alpha - (2\rho - r)t^{r - 1} - \lambda\bigr) \|z(t) - z_{*}\|^{2} \hypertarget{eq:energy function r=s line 2}{\label{eq:energy function r=s line 2}} \\
    &+ 2\lambda\theta t^{2\rho} \beta(t) \langle z(t) - z_{*}, V(z(t))\rangle \hypertarget{eq:energy function r=s line 3}{\label{eq:energy function r=s line 3}} \\
    &+ \frac{\theta^{2}}{2} t^{2(\rho + r)} \beta^{2}(t) \| V(z(t))\|^{2}. \hypertarget{eq:energy function r=s line 4}{\label{eq:energy function r=s line 4}}
\end{align}
Here, $\rho $ and $\lambda$ are taken as follows: 
\begin{itemize}
    \item If $r = 0$, then $0 < \lambda < \alpha$ and $\rho = r = 0$;
    \item If $r\in (0, 1)$, then $0 < \lambda < \alpha$ and $0 < \rho < r$;
    \item If $r = 1$, then $0 < \lambda < \alpha - 1$ and $\rho = r = 1$. 
\end{itemize}

\begin{theorem}\label{eq:main theorem continuous}
    Suppose that $\alpha, \theta > 0$, $r\in [0, 1]$ and that $\beta :[t_{0}, +\infty) \to (0, +\infty)$ satisfy the assumptions laid down in subsection \eqref{subsec:our work}. Let $t\mapsto z(t)$ be a solution to \eqref{eq:system r=s in (0, 1)} and let $z_{*}$ be a zero of $V$. Consider the following convergence rates as $t\to +\infty$: 
    \begin{equation}\label{eq:convergence rates of the main theorem}
        \| V(z(t))\| = o\left(\frac{1}{t^{\rho + r} \beta(t)}\right), \quad \langle z(t) - z_{*}, V(z(t))\rangle = o\left(\frac{1}{t^{2\rho}\beta(t)}\right), \quad \bigl\| \Dot{z}(t)\bigr\| = o\left(\frac{1}{t^{\rho}}\right). 
    \end{equation}
    The following statements are true:
    \begin{enumerate}[\rm (i)]
        \item If $r\in (0, 1)$, the above rates hold for $\rho \in (0, r)$. Furthermore, if $t\mapsto z(t)$ is bounded, then these rates hold for $\rho = r$. 
        \item If $r = 0$ or $r = 1$, then $t\mapsto z(t)$ is bounded; moreover, the above rates hold for $\rho = 0$ (if $r = 0$) and for $\rho = 1$ (if $r = 1$).
        \item If $r = 0$ or $r = 1$ or $\Bigl( r\in (0, 1)$ and $t\mapsto z(t)$ is bounded $\Bigr)$, then $z(t)$ converges weakly to a zero of $V$ as $t\to +\infty$.
    \end{enumerate}
\end{theorem}
\begin{proof}

Let us first briefly describe the idea of our proof to help the readers follow it more easily.
\paragraph{Sketch of the proof.}
We start with the energy function $\mathcal{E}_{\lambda}(t)$. By carefully examining its derivative, we can show that for some situations $\mathcal{E}_{\lambda}(t)$ is eventually nonincreasing and thus bounded.
Notice that the energy function can be decomposed into four smaller parts (lines \eqref{eq:energy function r=s line 1}-\eqref{eq:energy function r=s line 4}) which will be respectively denoted by $\mathcal{E}_{\lambda}^{i}$, with $i = 1, 2, 3, 4$. We will compute their time derivatives separately, and after some preliminary estimates, we see that $\frac{d}{dt}\mathcal{E}_{\lambda}(t)$ is bounded from above by either nonpositive terms or terms that can be combined to become nonpositive after completing the square.

From here, we can deduce the convergence rates as in \eqref{eq:convergence rates of the main theorem} by considering different scenarios.
For the critical cases $r=0$ and $r=1$, we can further show that the trajectories are bounded.

Finally, by utilizing the Opial's lemma \ref{lem:Opial}, we can show that the trajectories $z(t)$ converge weakly to a zero of $V$ as $t\to +\infty$, provided that it is bounded.
In order to do that, we need to produce in addition some integrability results and show that $\mathcal{E}_{\lambda}(t)$ converges for more than one value of $\lambda$.


\paragraph{Deriving the convergence rates \eqref{eq:convergence rates of the main theorem}.}
    First of all, notice that $\mathcal{E}_{\lambda}(t)$ is eventually nonnegative. Indeed, go back to line \eqref{eq:energy function r=s line 2}. If $r\in [0, 1)$, then $0 < \lambda < \alpha$ together with the fact that $t^{r - 1}\to 0$ as $t\to +\infty$ gives $\alpha - (2\rho - r)t^{r - 1} - \lambda \geq 0$ for large enough $t$. If $r = 1$, then $\alpha - (2\rho - r)t^{r - 1} - \lambda \equiv \alpha - 1 - \lambda > 0$ by assumption. 
    To carry out the proof, we compute the time derivative of $\mathcal{E}_{\lambda}$ at a $t\geq t_{0}$.  
    \begin{align*}
        \frac{d}{dt} \hyperlink{eq:energy function r=s line 1}{\mathcal{E}_{\lambda}^{1}(t)}  = & \: \Biggl\langle 2\lambda t^{\rho - r}(z(t) - z_{*}) + 2t^{\rho} \Dot{z}(t) + \theta t^{\rho + r} \beta(t) V(z(t)), \\
        &\quad\quad 2\lambda(\rho - r)t^{\rho - r - 1}(z(t) - z_{*}) + \bigl( 2\lambda t^{\rho - r} + 2\rho t^{\rho - 1} - 2t^{\rho - r} \alpha\bigr) \Dot{z}(t) \\
        &\quad\quad +\Bigl[ \bigl(\theta(\rho + r) t^{\rho + r - 1} - 2t^{\rho}\bigr)\beta(t) + \theta t^{\rho + r} \Dot{\beta}(t)\Bigr]V(z(t)) - \theta t^{\rho + r} \beta(t) \frac{d}{dt}V(z(t)) \Biggr\rangle \\
        = & \: 4\lambda^{2}(\rho - r)t^{2(\rho - r) - 1} \|z(t) - z_{*}\|^{2} \\
        &+ \Bigl[ 2\lambda t^{\rho - r} \bigl( 2\lambda t^{\rho - r} + 2\rho t^{\rho - 1} - 2t^{\rho - r} \alpha\bigr) + 4\lambda (\rho - r) t^{2\rho - r - 1}\Bigr] \bigl\langle z(t) - z_{*}, \Dot{z}(t)\bigr\rangle \\
        &+ \Biggl\{ \! 2\lambda t^{\rho - r} \Bigl[ \! \bigl(\theta(\rho \! + \! r) t^{\rho + r - 1} \! - 2t^{\rho}\bigr)\beta(t) \! + \! \theta t^{\rho + r}\Dot{\beta}(t) \! \Bigr] \! + \! 2\lambda \theta(\rho \! - \! r) t^{2\rho - 1} \beta(t) \! \Biggr\} \langle z(t) \! - \! z_{*}, V(z(t))\rangle \\ 
        &- 2\lambda\theta t^{2\rho} \beta(t) \left\langle z(t) - z_{*}, \frac{d}{dt}V(z(t))\right\rangle + 2t^{\rho} \bigl( 2\lambda t^{\rho - r} + 2\rho t^{\rho - 1} - 2 t^{\rho - r}\alpha\bigr) \bigl\| \Dot{z}(t)\bigr\|^{2} \\
        &+ \Biggl\{ 2t^{\rho} \Bigl[ \bigl( \theta(\rho + r) t^{\rho + r - 1} - 2t^{\rho}\bigr)\beta(t) + \theta t^{\rho + r} \Dot{\beta}(t)\Bigr] \\
        &\quad\quad +\theta t^{\rho + r} \beta(t) \bigl( 2\lambda t^{\rho - r} + 2\rho t^{\rho - 1} - 2t^{\rho - r} \alpha\bigr)\Biggr\} \bigl\langle \Dot{z}(t), V(z(t))\bigr\rangle \\
        &-2\theta t^{2\rho + r}\beta(t) \left\langle \Dot{z}(t), \frac{d}{dt}V(z(t))\right\rangle - \theta^{2} t^{2 \left( \rho + r \right)}\beta^{2}(t) \left\langle V(z(t)), \frac{d}{dt}V(z(t))\right\rangle \\
        &+ \theta t^{\rho + r}\beta(t) \Bigl[\bigl(\theta(\rho + r) t^{\rho + r - 1} - 2t^{\rho}\bigr)\beta(t) + \theta t^{\rho + r}\Dot{\beta}(t)\Bigr]\|V(z(t))\|^{2}, \\
        \frac{d}{dt} \hyperlink{eq:energy function r=s line 2}{\mathcal{E}_{\lambda}^{2}(t)} = & \: 2\lambda \Bigl[ 2\alpha (\rho - r) t^{2(\rho - r) - 1} - (2\rho - r)(2\rho - r - 1)t^{2\rho - r - 2} - 2(\rho - r) \lambda t^{2(\rho - r) - 1}\Bigr] \|z(t) - z_{*}\|^{2} \\
        &+ 4\lambda t^{2(\rho - r)} \bigl(\alpha - (2\rho - r)t^{r - 1} - \lambda\bigr) \bigl\langle z(t) - z_{*}, \Dot{z}(t)\bigr\rangle, \\
        \frac{d}{dt} \hyperlink{eq:energy function r=s line 3}{\mathcal{E}_{\lambda}^{3}(t)} = & \: \theta \Bigl[ 4\lambda \rho t^{2\rho - 1}\beta(t) + 2\lambda t^{2\rho}\Dot{\beta}(t)\Bigr] \langle z(t) - z_{*}, V(z(t))\rangle + 2\lambda \theta t^{2\rho} \beta(t) \bigl\langle \Dot{z}(t), V(z(t))\bigr\rangle \\
        &+ 2\lambda \theta t^{2\rho} \beta(t) \left\langle z(t) - z_{*}, \frac{d}{dt}V(z(t))\right\rangle, \\
        \frac{d}{dt} \hyperlink{eq:energy function r=s line 4}{\mathcal{E}_{\lambda}^{4}(t)} = & \: \frac{\theta^{2}}{2}\Bigl[ 2(\rho + r) t^{2(\rho + r) - 1} \beta^{2}(t) + 2t^{2(\rho + r)} \beta(t) \Dot{\beta}(t)\Bigr]\|V(z(t))\|^{2} \\
        &+ \theta^{2} t^{2(\rho + r)} \beta^{2}(t) \left\langle V(z(t)), \frac{d}{dt}V(z(t))\right\rangle.
    \end{align*}
    Putting everything together yields
    \begin{align}
        \frac{d}{dt}\mathcal{E}_{\lambda}(t) = & \: \Biggl\{4\lambda^{2}(\rho - r)t^{2(\rho - r) - 1} + 2\lambda \Bigl[ 2\alpha (\rho - r) t^{2(\rho - r) - 1} \nonumber\\
        &\quad\quad - (2\rho - r)(2\rho - r - 1)t^{2\rho - r - 2} - 2(\rho - r) \lambda t^{2(\rho - r) - 1}\Bigr]\Biggr\}\|z(t) - z_{*}\|^{2} \label{eq:d energy function r=s line 1}\\
        &+ \Biggl[ 2\lambda t^{\rho - r} \bigl( 2\lambda t^{\rho - r} + 2\rho t^{\rho - 1} - 2t^{\rho - r}\alpha\bigr) + 4\lambda (\rho - r)t^{2\rho - r - 1} \nonumber\\
        &\quad\quad + 4\lambda t^{2(\rho - r)} \bigl(\alpha - (2\rho - r)t^{r - 1} - \lambda\bigr)\Biggr] \bigl\langle z(t) - z_{*}, \Dot{z}(t)\bigr\rangle \label{eq:d energy function r=s line 2}\\
        &+ \Biggl\{ 2\lambda t^{\rho - r} \Bigl[ \bigl(\theta(\rho + r) t^{\rho + r - 1} - 2t^{\rho}\bigr)\beta(t) + \theta t^{\rho + r}\Dot{\beta}(t) \Bigr] + 2\lambda\theta(\rho - r) t^{2\rho - 1} \beta(t) \nonumber\\
        &\quad\quad + \theta \Bigl[ 4\lambda \rho t^{2\rho - 1}\beta(t) + 2\lambda t^{2\rho}\Dot{\beta}(t)\Bigr] \Biggr\} \langle z(t) - z_{*}, V(z(t))\rangle \label{eq:d energy function r=s line 3}\\
        &+ 2t^{\rho} \bigl( 2\lambda t^{\rho - r} + 2\rho t^{\rho - 1} - 2 t^{\rho - r}\alpha\bigr) \bigl\| \Dot{z}(t)\bigr\|^{2} \label{eq:d energy function r=s line 4}\\
        &+ \Biggl\{ 2t^{\rho} \Bigl[ \bigl(\theta(\rho + r) t^{\rho + r - 1} - 2t^{\rho}\bigr)\beta(t) +\theta t^{\rho + r} \Dot{\beta}(t)\Bigr] \nonumber\\
        &\quad\quad + \theta t^{\rho + r} \beta(t) \bigl( 2\lambda t^{\rho - r} + 2\rho t^{p - 1} - 2t^{\rho - r} \alpha\bigr) + 2\lambda\theta t^{2\rho}\beta(t) \Biggr\} \bigl\langle \Dot{z}(t), V(z(t))\bigr\rangle \label{eq:d energy function r=s line 5}\\
        &- 2\theta t^{2\rho + r} \beta(t) \left\langle \Dot{z}(t), \frac{d}{dt}V(z(t))\right\rangle \label{eq:d energy function r=s line 6}\\
        &+ \Biggl\{ \theta t^{\rho + r} \beta(t) \Bigl[ \bigl( \theta(\rho + r) t^{\rho + r - 1} - 2t^{\rho} \bigr)\beta(t) + \theta t^{\rho + r} \Dot{\beta}(t)\Bigr] \nonumber\\
        &\quad\quad +\theta^{2} \Bigl[ (\rho + r) t^{2(\rho + r) - 1} \beta^{2}(t) + t^{2(\rho + r)} \beta(t) \Dot{\beta}(t)\Bigr] \Biggr\} \|V(z(t))\|^{2}. \label{eq:d energy function r=s line 7}
    \end{align}
    We will check that lines \eqref{eq:d energy function r=s line 1}, \eqref{eq:d energy function r=s line 3}, \eqref{eq:d energy function r=s line 4}, \eqref{eq:d energy function r=s line 6} and \eqref{eq:d energy function r=s line 7} are nonpositive for large enough $t$, and that line \eqref{eq:d energy function r=s line 2} vanishes. 

    \eqref{eq:d energy function r=s line 1}: we have
    \begin{align*}
        & \: 4\lambda^{2} (\rho - r) t^{2(\rho - r) - 1} + 2\lambda \Bigl[2\alpha (\rho - r) t^{2(\rho - r) - 1} - (2\rho - r)(2\rho - r - 1)t^{2\rho - r - 2} -2\lambda(\rho - r)t^{2(\rho - r) - 1}\Bigr] \\
        = & \: 4\lambda^{2} (\rho - r) t^{2(\rho - r) - 1} \! + 2\lambda \Bigl[2\alpha (\rho - r) t^{2(\rho - r) - 1} \! - (2\rho - r)(2\rho - r - 1)t^{2\rho - r - 2}\Bigr] - 4\lambda^{2} (\rho - r) t^{2(\rho - r) - 1} \\
        = & \: 2\lambda t^{2(\rho - r) - 1} \Big[ 2\alpha (\rho - r) - (2\rho - r)(2\rho - r - 1)t^{r - 1}\Bigr]. 
    \end{align*}
    If $r = 0$, then $\rho = 0$ and the line vanishes. If $r \in (0, 1)$, then $\rho < r$, which gives $2\alpha(\rho - r) < 0$. Since $t^{r - 1} \to 0$ as $t\to +\infty$, for large enough $t$ it holds $2\alpha(\rho - r) - (2\rho - r)(2\rho - r - 1)t^{r - 1} < 0$. If $r = 1$, then $\rho = 1$ and this line again vanishes.

    \eqref{eq:d energy function r=s line 2}: we have 
    \begin{align*}
        &\: 2\lambda t^{\rho - r} \bigl( 2\lambda t^{\rho - r} + 2\rho t^{\rho - 1} - 2\rho t^{\rho - r}\alpha\bigr) + 4\lambda (\rho - r) t^{2\rho - r - 1}  + 4\lambda t^{2(\rho - r)}\bigl(\alpha - (2\rho - r)t^{r - 1} - \lambda\bigr)\\
        = &\: 4\lambda t^{2(\rho - r)}\bigl(\lambda + (2\rho - r)t^{r - 1} - \alpha\bigr) + 4\lambda t^{2(\rho - r)}\bigl(\alpha - (2\rho - r)t^{r - 1} - \lambda\bigr),
    \end{align*}
    which means that this line is identically zero. 

    \eqref{eq:d energy function r=s line 3}: the monotonicity of $V$ ensures $\langle z(t) - z_{*}, V(z(t))\rangle\geq 0$ for all $t\geq t_{0}$. We have
    \begin{align*}
        &\: 2\lambda t^{\rho - r}\Bigl[\bigl(\theta(\rho + r)t^{\rho + r - 1} - 2t^{\rho}\bigr)\beta(t) + \theta t^{\rho + r}\Dot{\beta}(t)\Bigr] + 2\lambda (\rho - r) \theta t^{2\rho - 1}\beta(t) \\
        &+ \theta \Bigl[ 4\lambda\rho t^{2\rho - 1}\beta(t) + 2\lambda t^{2\rho}\Dot{\beta}(t)\Bigr] \\
        = & \: 2\lambda t^{2\rho - r} \Bigl\{ \Bigl[ \theta\bigl(\rho + r + 2\rho\bigr)t^{r - 1} - 2 + (\rho - r)\theta t^{r - 1}\Bigr]\beta(t) + 2\theta t^{r} \Dot{\beta}(t)\Bigr\} \\
        = & \: 4\lambda t^{2\rho - r} \Bigl[ \bigl(2 \rho \theta t^{r - 1} - 1\bigr)\beta(t) + \theta t^{r} \Dot{\beta}(t)\Bigr] \\
        \leq & \: 4\lambda t^{2\rho - r} \Bigl[ \bigl(2r \theta t^{r - 1} - 1\bigr)\beta(t) + \theta t^{r} \Dot{\beta}(t)\Bigr] \\
        \leq & \: 0,
    \end{align*}
    where the last equality comes from $\rho + r + 2 \rho + \rho - r = 4 \rho$, and the last inequality is a consequence of the growth condition on $\beta$.

    \eqref{eq:d energy function r=s line 4}: arguing as we did at the beginning of the proof, we have
    \[
        2t^{\rho} \bigl( 2\lambda t^{\rho - r} + 2\rho t^{\rho - 1} - 2t^{\rho - r} \alpha\bigr) = 4 t^{2\rho - r}\bigl( \lambda + \rho t^{r - 1} - \alpha\bigr) \leq 0
    \] 
    for large enough $t$. 
    
    \eqref{eq:d energy function r=s line 6}: since the monotonicity of $V$ ensures $\left\langle \Dot{z}(t), \frac{d}{dt}V(z(t))\right\rangle\geq 0$ for every $t\geq t_{0}$, this line is evidently nonpositive.

    \eqref{eq:d energy function r=s line 7}: we have 
    \begin{align*}
        &\: \theta t^{2\rho + r} \beta(t) \Bigl[ \bigl(\theta (\rho + r)t^{r - 1} - 2\bigr)\beta(t) + \theta t^{r} \Dot{\beta}(t)\Bigr] + \theta t^{2\rho + r}\beta(t) \Bigl[ \theta(\rho + r) t^{r - 1} \beta(t) + \theta t^{r} \Dot{\beta}(t)\Bigr] \\
        = &\: 2\theta t^{2\rho + r} \beta(t) \Bigl[ \bigl(\theta(\rho + r) t^{r - 1} - 1\bigr)\beta(t) + \theta t^{r} \Dot{\beta}(t)\Bigr] \\
        \leq &\: 2\theta t^{2\rho + r} \beta(t) \Bigl[ \bigl(2r\theta t^{r - 1} - 1\bigr)\beta(t) + \theta t^{r} \Dot{\beta}(t)\Bigr] \\
        \leq &\: 0, 
    \end{align*}
    where we used $\rho + r \leq 2r$, and the growth condition on $\beta$. 

    Additionally, we need the term accompanying $\bigl\langle \Dot{z}(t), V(z(t))\bigr\rangle$, which according to \eqref{eq:d energy function r=s line 5} reads 
    \begin{align*}
        &\: 2t^{\rho}\Bigl[ \bigl(\theta(\rho + r) t^{\rho + r - 1} - 2t^{\rho}\bigr)\beta(t) + \theta t^{\rho + r} \Dot{\beta}(t)\Bigr] + \theta t^{\rho + r} \beta(t)\bigl(2\lambda t^{\rho - r} + 2\rho t^{\rho - 1} - 2t^{\rho - r}\alpha\bigr) + 2\lambda \theta t^{2\rho} \beta(t) \\
        = &\: 2t^{2\rho} \Bigl\{ \Bigl[ \bigl(\theta (\rho + r) t^{r - 1} - 2\bigr) + \theta\bigl( \lambda + \rho t^{r - 1} - \alpha\bigr) + \lambda\theta\Bigr] \beta(t)  + \theta t^{r}\Dot{\beta}(t)\Bigr\} \\
        = &\: 2t^{2\rho} \left\{ \left[ 2\theta \left(\lambda + \frac{\rho + r}{2} t^{r - 1} - \alpha\right) + \bigl( \theta \alpha + \theta \rho t^{r - 1} - 2\bigr)\right]\beta(t) + \theta t^{r} \Dot{\beta}(t)\right\}. 
    \end{align*}
    
    According to the computations we have now made for \eqref{eq:d energy function r=s line 1}-\eqref{eq:d energy function r=s line 7}, and taking into account that $\rho \leq \frac{\rho + r}{2}$, for large enough $t$ we have  
    \begin{align}
        \frac{d}{dt}\mathcal{E}_{\lambda}(t) \leq &\: 2\lambda t^{2(\rho - r) - 1} \Big[ 2\alpha (\rho - r) - (2\rho - r)(2\rho - r - 1)t^{r - 1}\Bigr] \|z(t) - z_{*}\|^{2} \label{eq:d general energy function line 1 for quad}\\
        &+ 4\lambda t^{2\rho - r} \Bigl[ \bigl(2r\theta t^{r - 1} - 1\bigr)\beta(t) + \theta t^{r} \Dot{\beta}(t)\Bigr] \langle z(t) - z_{*}, V(z(t))\rangle \label{eq:d general energy function line 2 for quad}\\
        &+ 4 t^{2\rho - r}\left( \lambda + \frac{\rho + r}{2} t^{r - 1} - \alpha\right) \bigl\| \Dot{z}(t)\bigr\|^{2} \label{eq:d general energy function line 3 for quad}\\
        &+ 2t^{2\rho} \left\{ \left[ 2\theta \left(\lambda + \frac{\rho + r}{2} t^{r - 1} - \alpha\right) + \bigl( \theta \alpha + \theta \rho t^{r - 1} - 2\bigr)\right]\beta(t) + \theta t^{r} \Dot{\beta}(t)\right\} \bigl\langle \Dot{z}(t), V(z(t))\bigr\rangle \label{eq:d general energy function line 4 for quad}\\
        &+ 2\theta t^{2\rho + r} \beta(t) \Bigl[ \bigl(2r\theta t^{r - 1} - 1\bigr)\beta(t) + \theta t^{r} \Dot{\beta}(t)\Bigr] \|V(z(t))\|^{2}. \label{eq:d general energy function line 5 for quad}
    \end{align}
    We will work with lines \eqref{eq:d general energy function line 3 for quad}-\eqref{eq:d general energy function line 5 for quad}. Define
    \[
        \varepsilon(t) := \alpha - \frac{\rho + r}{2} t^{r - 1} - \lambda, \quad c(t) := \theta\alpha + \theta \rho t^{r - 1} - 2. 
    \]
    Using our assumptions on $\alpha, \theta$ and $\lambda$, we have that $\varepsilon(t) > 0$ and $c(t)\geq 0$ for large enough $t$. We will show that for certain $\lambda$, for $t$ sufficiently large it holds
    \begin{align}
        0 \geq & \: -3\varepsilon(t) t^{2\rho - r} \bigl\| \Dot{z}(t)\bigr\|^{2} + 2t^{2\rho} \Bigl[ \bigl(-2\theta \varepsilon(t) + c(t)\bigr) \beta(t) + \theta t^{r} \Dot{\beta}(t)\Bigr] \bigl\langle \Dot{z}(t), V(z(t))\bigr\rangle \nonumber\\
        &+ \frac{4}{3} \theta t^{2\rho + r} \beta(t) \Bigl[ \bigl(2r\theta t^{r - 1} - 1\bigr)\beta(t) + \theta t^{r} \Dot{\beta}(t)\Bigr] \|V(z(t))\|^{2}. \label{eq:quad statement implied by B^2 - AC}
    \end{align}
    We will use Lemma \ref{lem:quad} with $X = \Dot{z}(t)$ and $Y = V(z(t))$, and $A, B, C$ chosen as follows: 
    \begin{gather*}
        B = t^{2\rho} \Bigl[ \bigl(-2\theta \varepsilon(t) + c(t)\bigr) \beta(t) + \theta t^{r} \Dot{\beta}(t)\Bigr], \quad A = -3\varepsilon(t) t^{2\rho - r}, \\
        C = \frac{4}{3} \theta t^{2\rho + r} \beta(t) \Bigl[ \bigl(2r\theta t^{r - 1} - 1\bigr)\beta(t) + \theta t^{r} \Dot{\beta}(t)\Bigr].
    \end{gather*}
    According to this lemma, it is sufficient to show that $B^{2} - AC \leq 0$. We write 
    \begin{align*}
        B^{2} - AC = & \: t^{4\rho} \Bigl[ \bigl(-2\theta \varepsilon(t) \! + \! c(t)\bigr) \beta(t) \! + \! \theta t^{r} \Dot{\beta}(t) \Bigr]^{2} \! \! \! + \! 3 \! \cdot \! \frac{4}{3} \varepsilon(t) \theta t^{2\rho - r} \! \cdot \! t^{2\rho + r} \beta(t) \Bigl[ \bigl(2r\theta t^{r - 1} \! - \! 1\bigr)\beta(t) \! + \! \theta t^{r} \Dot{\beta}(t) \Bigr] \\
        = &\: t^{4p} \Bigl[ \bigl(-2\theta \varepsilon(t) + c(t)\bigr)^{2}\beta^{2}(t) + 2\theta t^{r} \bigl(-2\theta \varepsilon(t) + c(t)\bigr)\beta(t)\Dot{\beta}(t) + \theta^{2} t^{2r} \bigl( \Dot{\beta}(t)\bigr)^{2}\Bigr] \\
        &+ 4 \theta \bigl( 2r\theta t^{r - 1} - 1\bigr)\varepsilon(t) t^{4\rho}\beta^{2}(t) + 4\theta^{2} \varepsilon(t) t^{4\rho + r} \beta(t) \Dot{\beta}(t) \\
        = & \: t^{4\rho} \Bigl[ \bigl(-2\theta \varepsilon(t) + c(t)\bigr)^{2} \beta^{2}(t) \! + \! 2\theta t^{r} c(t) \beta(t) \Dot{\beta}(t) + \theta^{2} t^{2r} \bigl( \Dot{\beta}(t)\bigr)^{2} \! + \! 4\theta \bigl(2r\theta t^{r - 1} - 1\bigr)\varepsilon(t) \beta^{2}(t)\Bigr].
    \end{align*}
    Now, according to the growth condition \eqref{eq:growth condition case r=s}, there exists $\delta > 0$ such that 
    \begin{equation}\label{eq:growth condition delta}
        \frac{t^{r} \Dot{\beta}(t)}{\beta(t)} \leq \frac{1}{\theta} - 2r t^{r - 1} - \delta \quad \forall t\geq t_{0}. 
    \end{equation}
    This entails
    \begin{equation} \label{eq:derived ineq from growth condition r=s}
        \theta t^{r} \beta(t) \Dot{\beta}(t) \leq \Bigl((1 - 2r\theta t^{r - 1}) - \delta\theta\Bigr)\beta^{2}(t) \quad \text{and} \quad \theta^{2} t^{2r} \bigl( \Dot{\beta}(t)\bigr)^{2} \leq \Bigl((1 - 2 r \theta t^{r - 1}) - \delta\theta\Bigr)^{2} \beta^{2}(t).
    \end{equation}
    Since $c(t) \geq 0$, we now arrive at 
    \begin{align}
        B^{2} - AC \leq & \: t^{4\rho} \beta^{2}(t)\Bigl[ \bigl(-2\theta \varepsilon(t) + c(t)\bigr)^{2} + 2c(t)\Bigl( (1 - 2 r \theta t^{r - 1}) - \delta\theta\Bigr) \nonumber\\
        &\quad\quad + \Bigl( (1 - 2 r \theta t^{r - 1}) - \delta\theta\Bigr)^{2} + 4\theta(2 r \theta t^{r - 1} - 1)\varepsilon(t)\Bigr] \nonumber\\
        = & \: t^{4\rho} \beta^{2}(t) \left\{ 4\theta^{2} \varepsilon^{2}(t) - 4\theta \left( 1 - 2 r \theta t^{r - 1} + c(t) \right) \varepsilon(t) + \left[ c(t) + \left( \left( 1 - 2 r \theta t^{r - 1} \right) - \delta\theta \right) \right] ^{2} \right\} \nonumber\\
        = & \: t^{4\rho} \beta^{2}(t) \pi(t), \label{eq:improved B2 - AC}
    \end{align}
    where $\pi(t)$ is the term between curly brackets. We will analyze $\pi(t)$ separating the cases $r\in [0, 1)$ and $r = 1$. 

    \fbox{Case $r \in [0, 1)$:} define 
    \[
        c_{r} := \theta \alpha - 2, \quad \varepsilon_{r} := \alpha - \lambda, \quad \mu(t) := 4\theta^{2} t^{2}  - 4\theta \left( c_{r} + 1 \right) t + \left( c_{r} + 1 - \delta\theta \right)^{2}. 
    \]
    Notice that $c(t) \to c_{r}$, $\varepsilon(t) \to \varepsilon_{r}$ and, therefore, $\pi(t) \to \mu(\varepsilon_{r})$ as $t\to +\infty$. The discriminant of the quadratic equation $\mu(t) = 0$ reads 
    \[
        \Delta = 16 \theta^{2} \left[ \left( c_{r} + 1 \right) ^{2} - \left( c_{r} + 1 - \delta\theta \right)^{2} \right]. 
    \]
    Since $0 < 1 - \delta\theta$, we have $c_{r} + 1 > c_{r} + 1 - \delta\theta > 0$ and thus $\Delta > 0$. This means that $\mu$ has two distinct roots: by setting $\Tilde{\Delta} := \frac{1}{8\theta^{2}} \sqrt{\Delta}$, these are given by 
    \[
        \underline{\varepsilon} := \frac{c_{r} + 1}{2\theta} - \Tilde{\Delta}, \quad \overline{\varepsilon} := \frac{c_{r} + 1}{2\theta} + \Tilde{\Delta}. 
    \]
    The midpoint between the two roots is given by $\varepsilon_{r, 1} := \frac{c_{r} + 1}{2\theta}$, which fulfills
    \[
        0 < \varepsilon_{r, 1} = \frac{c_{r} + 1}{2\theta} = \frac{\theta\alpha - 1}{2\theta} = \frac{\alpha}{2} - \frac{1}{2\theta} < \alpha. 
    \]
    Choose $\varepsilon_{r, 2} > 0$ such that $\varepsilon_{r, 1} - \Tilde{\Delta} < \varepsilon_{r, 2} < \varepsilon_{r, 1}$. Set $\lambda_{r, i} := \alpha - \varepsilon_{r, i}$, $i = 1, 2$. As we remarked previously, we have $\pi(t) \to \mu (\varepsilon_{r, i})$ as $t\to +\infty$. Therefore, for large enough $t$, we have 
    \[
        \pi(t) < \frac{\mu(\varepsilon_{r, i})}{2} = \frac{\mu(\alpha - \lambda_{r, i})}{2} < 0 \quad \text{for }i=1, 2. 
    \]
    Going back to \eqref{eq:improved B2 - AC}, this means that 
    \[
        B^{2} - AC \leq t^{4\rho} \beta^{2}(t) \frac{\mu(\alpha - \lambda_{r, i})}{2} < 0 \quad \text{for large enough }t\text{ and }i=1, 2.
    \]

    \fbox{Case $r = 1$:} here, we have 
    \[
        \varepsilon(t) \equiv \varepsilon := \alpha - 1 - \lambda, \quad \text{and} \quad c(t) \equiv c := \theta\alpha + \theta - 2.
    \]
    In this case, we define
    \[
        \mu(t) := 4\theta^{2} t^{2} - 4\theta\left(1 - 2\theta + c\right)t + \left( c + 1 - 2\theta - \delta\theta\right) ^{2}. 
    \]
    Notice that with this definition we have $\pi(t) \equiv \mu(\varepsilon)$. The discriminant of the quadratic equation $p(t)=0$ reads 
    \[
        \Delta = 16\theta^{2} \left[ \left(c + 1 - 2\theta\right))^{2} - \left(c + 1 - 2\theta - \delta\theta\right))^{2} \right].
    \]
    Since $0 < 1 - 2\theta - \delta\theta$, we have $c + 1 - 2\theta > c + 1 - 2\theta - \delta\theta > 0$ and thus $\Delta > 0$. Again, define $\Tilde{\Delta}$ as before. The roots of $\mu$ are given by
    \[
        \underline{\varepsilon} := \frac{c + 1 - 2\theta}{2\theta} - \Tilde{\Delta}, \quad \overline{\varepsilon} := \frac{c + 1 - 2\theta}{2\theta} + \Tilde{\Delta}.
    \]
    The midpoint between the roots is $\varepsilon_{1, 1} := \frac{c + 1 - 2\theta}{2\theta}$, which satisfies
    \[
        0 < \varepsilon_{1, 1} = \frac{c + 1 - 2\theta}{2\theta} = \frac{\theta \alpha - \theta - 1}{2\theta} = \frac{\alpha - 1}{2} - \frac{1}{2\theta} < \alpha - 1. 
    \]
    Similar as before, choose $\varepsilon_{1, 2} > 0$ such that $\varepsilon_{1, 1} - \Tilde{\Delta} < \varepsilon_{1, 2} < \varepsilon_{1, 1}$. Set $\lambda_{1, i} := \alpha - 1 - \varepsilon_{1, i}$, $i = 1, 2$. Since $\mu(\alpha - 1 - \lambda_{1, i}) = \mu(\varepsilon_{1, i}) < 0$ for $i = 1, 2$, going back to \eqref{eq:improved B2 - AC} gives 
    \[
        B^{2} - AC \leq t^{4\rho}\beta^{2}(t) \mu(\alpha - 1 - \lambda_{1, i}) < 0 \quad \text{for large enough }t\text{ and }i=1, 2.
    \]

    Summarizing, for the choices $\lambda = \lambda_{r, i}$, $r \in[0, 1]$, $i = 1, 2$ it holds, recalling lines \eqref{eq:d general energy function line 1 for quad}-\eqref{eq:d general energy function line 5 for quad}, together with inequality \eqref{eq:quad statement implied by B^2 - AC},  
    \begin{align}
        \frac{d}{dt}\mathcal{E}_{\lambda_{r, i}}(t) = &\: 2\lambda_{r, i} t^{2(\rho - r) - 1}\Bigl[ 2\alpha(\rho - r) - (2\rho - r)(2\rho - r - 1)t^{r - 1}\Bigr]\| z(t) - z_{*}\|^{2} \nonumber\\
        &+ 4\lambda_{r, i} t^{2\rho - r} \Bigl[ \bigl(2 r \theta t^{r - 1} - 1\bigr)\beta(t) + \theta t^{r} \Dot{\beta}(t)\Bigr] \langle z(t) - z_{*}, V(z(t))\rangle \nonumber\\
        &+ t^{2\rho - r}\left(\lambda_{r, i} + \frac{\rho + r}{2}t^{r - 1} - \alpha\right)\bigl\| \Dot{z}(t)\bigr\|^{2} \nonumber\\
        &+ \frac{2}{3}\theta t^{2\rho + r}\beta(t) \Bigl[ \bigl(2 r \theta t^{r - 1} - 1\bigr)\beta(t) + \theta t^{r} \Dot{\beta}(t)\Bigr] \| V(z(t))\|^{2} \nonumber\\
        \leq &\: 0, \label{eq:d energy function before changing rho for r}
    \end{align}
    where $t$ is taken large enough, say $t\geq T \geq t_{0}$. This means that $\mathcal{E}_{\lambda_{r, i}}$ monotonically decreases on $[T, +\infty)$, and thus 
    \begin{equation} \label{eq: energy function bounded}
        0 \leq \mathcal{E}_{\lambda_{r, i}}(t) \leq \mathcal{E}_{\lambda_{r, i}}(T) \text{ for }t\geq T\text{, }r\in [0, 1]\text{ and }i=1, 2.
    \end{equation}

\paragraph{Deducing convergence rates from the estimates.}
    We will now proceed by distinguishing between the cases $r\in (0, 1)$, $r = 0$ and $r = 1$.  

    \fbox{Case $r\in (0, 1)$:} according to the definition of the energy function \eqref{eq:energy function r=s line 1}-\eqref{eq:energy function r=s line 4} and \eqref{eq: energy function bounded}, for $t\geq T$ we have 
    \begin{equation}\label{eq: rates for <z(t) - z*, V(z(t))>, ||V(z(t))||}
        \langle z(t) - z_{*}, V(z(t))\rangle \leq \frac{\mathcal{E}_{\lambda_{r, 1}}(T)}{2\lambda_{r, 1} \theta}\cdot\frac{1}{t^{2\rho}\beta(t)} \quad \text{and} \quad \| V(z(t))\| \leq \frac{\sqrt{2\mathcal{E}_{\lambda_{r, 1}}(T)}}{\theta} \cdot \frac{1}{t^{\rho + r} \beta(t)}.
    \end{equation}
    Furthemore, say that for $t\geq T$ we have $(2\rho - r)t^{r - 1} < \xi < \alpha - \lambda_{r, 1}$. Going back to \eqref{eq:energy function r=s line 2}, this means that for $t\geq T$ it holds 
    \[
        2\lambda_{r, 1} t^{2(\rho - r)} \bigl(\alpha - \xi - \lambda_{r, 1}\bigr) \| z(t) - z_{*}\|^{2} \leq 2\lambda_{r, 1} t^{2(\rho - r)} \bigl(\alpha - (2\rho - r)t^{r - 1}) - \lambda_{r, 1}\bigr) \| z(t) - z_{*}\|^{2} \leq \mathcal{E}_{\lambda_{r, 1}}(T)
    \]
    and thus 
    \[
        2\lambda_{r, 1}t^{\rho - r}\| z(t) - z_{*}\| \leq \sqrt{\frac{2 \mathcal{E}_{\lambda_{r, 1}}(T)}{ \bigl(\alpha - \xi - \lambda_{r, 1}\bigr)}}.
    \]
    Combining this together with \eqref{eq:energy function r=s line 1}, \eqref{eq: energy function bounded} and \eqref{eq: rates for <z(t) - z*, V(z(t))>, ||V(z(t))||} yields
    \begin{align*}
        2t^{\rho} \bigl\| \Dot{z}(t)\bigr\| \leq &\: 2\lambda_{r, 1} t^{\rho - r} \| z(t) - z_{*}\| + \theta t^{\rho + r} \beta(t) \| V(z(t))\| \\
        &+ \Bigl\| 2\lambda_{r, 1} t^{\rho - r}(z(t) - z_{*}) + 2t^{\rho} \Dot{z}(t) + \theta t^{\rho + r} \beta(t) V(z(t))\Bigr\| \\
        \leq &\: \sqrt{\frac{2 \mathcal{E}_{\lambda_{r, 1}}(T)}{ \bigl(\alpha - \xi - \lambda_{r, 1}\bigr)}} + 2\sqrt{2 \mathcal{E}_{\lambda_{r, 1}}(T)}.
    \end{align*}
    The previous inequality, together with \eqref{eq: rates for <z(t) - z*, V(z(t))>, ||V(z(t))||} and the fact that $\rho\in (0, r)$ was arbitrary, tells us that so far we have, as $t\to +\infty$,
    \begin{equation}\label{eq:to change rho by r}
        \| V(z(t))\| = o\left(\frac{1}{t^{\rho + r} \beta(t)}\right), \quad \langle z(t) - z_{*}, V(z(t))\rangle = o\left(\frac{1}{t^{2\rho}\beta(t)}\right), \quad \bigl\| \Dot{z}(t)\bigr\| = o\left(\frac{1}{t^{\rho}}\right). 
    \end{equation}

    \fbox{Case $r=0$ or $r=1$:} going back to \eqref{eq:energy function r=s line 2}, we have, for $t\geq T$,
    \[
        \| z(t) - z_{*}\| \leq \sqrt{\frac{\mathcal{E}_{\lambda_{0, 1}}(T)}{2\lambda_{0, 1}(\alpha - \lambda_{0, 1})}} \ \ (\mbox{when} \ r = 0) \ \ \text{or} \ \ \| z(t) - z_{*}\| \leq \sqrt{\frac{\mathcal{E}_{\lambda_{1, 1}}(T)}{2\lambda_{1, 1}(\alpha - 1 - \lambda_{1, 1})}} \ \ (\mbox{when} \ r = 1),
    \]
    which gives the boundedness of $t\mapsto z(t)$. Using the boundedness of $\mathcal{E}_{\lambda_{0, 1}}$ (when $r = 0$) and that of $\mathcal{E}_{\lambda_{1, 1}}$ (when $r = 1$) and arguing exactly as in the case $r\in (0, 1)$, we obtain
    \begin{equation}\label{eq:to change O by o}
        \| V(z(t))\| =  \mathcal{O}\left(\frac{1}{t^{2r} \beta(t)}\right), \quad \langle z(t) - z_{*}, V(z(t))\rangle = \mathcal{O}\left(\frac{1}{t^{2r}\beta(t)}\right), \quad \bigl\| \Dot{z}(t)\bigr\| = \mathcal{O}\left(\frac{1}{t^{r}}\right). 
    \end{equation}
    as $t\to +\infty$ for $r = 0$ and $r = 1$. 

\paragraph{Weak convergence of the trajectories.}
    In the following steps, we will assume that $t\mapsto z(t)$ is bounded when $r\in (0, 1)$. This will allow us to set $\rho = r$ in \eqref{eq:to change rho by r} and to show weak convergence of the trajectories towards a zero of $V$. Additionally, we will improve $\mathcal{O}$ to $o$ in \eqref{eq:to change O by o} and we will again show weak convergence of the trajectories towards a zero of $V$. 

    Taking $\rho \in (0, r)$ when $r\in (0, 1)$ is done only in order to obtain that the coefficient accompanying $\| z(t) - z_{*}\|^{2}$ eventually becomes nonpositive, therefore producing $\frac{d}{dt}\mathcal{E}_{\lambda}(t) \leq 0$ (recall line \eqref{eq:d energy function r=s line 1}). Setting $\rho = r$ creates no other issues otherwise, but changes the bound for $\frac{d}{dt}\mathcal{E}_{\lambda}(t)$: indeed, going back to \eqref{eq:d energy function before changing rho for r} and plugging $\rho = r$ gives, for $t\geq T$ and $i = 1, 2$,  
    \begin{align}
        \frac{d}{dt}\mathcal{E}_{\lambda_{r, i}}(t) = &\: 2\lambda_{r, i} r(1 - r) t^{r - 2}\| z(t) - z_{*}\|^{2} \nonumber\\
        &+ 4\lambda_{r, i} t^{r} \Bigl[ \bigl(2 r \theta t^{r - 1} - 1\bigr)\beta(t) + \theta t^{r} \Dot{\beta}(t)\Bigr] \langle z(t) - z_{*}, V(z(t))\rangle \nonumber\\
        &+ t^{r}\big(\lambda_{r, i} + r t^{r - 1} - \alpha\bigr)\bigl\| \Dot{z}(t)\bigr\|^{2} \nonumber\\
        &+ \frac{2}{3}\theta t^{3r}\beta(t) \Bigl[ \bigl(2 r \theta t^{r - 1} - 1\bigr)\beta(t) + \theta t^{r} \Dot{\beta}(t)\Bigr] \| V(z(t))\|^{2} \nonumber\\
        \leq &\: 2\lambda_{r, i} r(1 - r) t^{r - 2}\| z(t) - z_{*}\|^{2}. \label{eq:d energy function when rho = r}
    \end{align}
    We have already established the boundedness of $t\mapsto z(t)$ when $r = 0$ or $r = 1$, and we will assume it when $r\in (0, 1)$. Say that $\| z(t) - z_{*}\|^{2} \leq M$ for all $t\geq T$ for some $M\geq 0$. According to \eqref{eq:d energy function when rho = r}, for $t\geq T$ we now have 
    \[
        \frac{d}{dt}\mathcal{E}_{\lambda_{r, i}}(t) \leq 2\lambda_{r, i} r(1 - r) M t^{r - 2} 
    \]
    and therefore 
    \[
        \frac{d}{dt}\left[ \mathcal{E}_{\lambda_{r, i}}(t) + 2\lambda_{r, i} r M t^{r - 1}\right] = \frac{d}{dt}\mathcal{E}_{\lambda_{r, i}}(t) + 2\lambda_{r, i} r (r - 1) M t^{r - 2} \leq 0.
    \]
    As a consequence of this, $t\mapsto \mathcal{E}_{\lambda_{r, i}}(t) + 2\lambda_{r, i} r Mt^{r - 1}$ is nonnegative and monotonically decreasing on $[T, +\infty)$, which means it has a limit as $t\to +\infty$. Since $2\lambda_{r, i} r (1 - r) M t^{r - 1}$ also has a limit as $t\to +\infty$, we come to the following: for $\rho = r$, we have
    \begin{align}
        0 \leq \mathcal{E}_{\lambda_{r, i}}(t) \leq \mathcal{E}_{\lambda_{r, i}}(T) + 2\lambda_{r, i} r (1 - r) M T^{r - 1} \text{ for }t\geq T\text{, }r\in [0, 1]\text{ and }i=1, 2 ;& \label{eq: energy function bounded when rho=r}\\
        \lim_{t\to +\infty}\mathcal{E}_{\lambda_{r, i}}(t) \text{ exists for }r\in [0, 1]\text{ and }i = 1, 2.& \label{eq:limit of continuous E when rho=r exists}
    \end{align}
    By going back to lines \eqref{eq:energy function r=s line 1}-\eqref{eq:energy function r=s line 4}, setting $\rho = r$ and reasoning exactly as we did before, we arrive at 
    \begin{equation}\label{eq:to change O by o 2}
        \| V(z(t))\| =  \mathcal{O}\left(\frac{1}{t^{2r} \beta(t)}\right), \quad \langle z(t) - z_{*}, V(z(t))\rangle = \mathcal{O}\left(\frac{1}{t^{2r}\beta(t)}\right), \quad \bigl\| \Dot{z}(t)\bigr\| = \mathcal{O}\left(\frac{1}{t^{r}}\right). 
    \end{equation}
    as $t\to +\infty$, where these rates now hold for all $r\in [0, 1]$.

    For $r\in [0, 1]$, we proceed to show three integrability results that we will need later. Recall \eqref{eq:growth condition delta}: for $t\geq T$, we have
    \begin{align*}
        \frac{t^{r} \Dot{\beta(t)}}{\beta(t)} \leq \frac{1}{\theta} - 2r t^{r - 1} - \delta \:\: &\Rightarrow \:\: \theta t^{r} \Dot{\beta}(t) \leq \bigl(1 - 2r\theta t^{r - 1}\bigr) \beta(t) - \delta\theta \beta(t) \\
        &\Rightarrow \:\: \delta \theta \beta(t) \leq (1 - 2 r \theta t^{r - 1}) \beta(t) - \theta t^{r} \Dot{\beta}(t).
    \end{align*}
    Going back to \eqref{eq:d energy function when rho = r}, we choose $i = 1$ and integrate the inequality from $T$ to $t\geq T$:
    \begin{align*}
        &\: \int_{T}^{t}s^{r} \bigl(\alpha - r s^{r - 1} - \lambda_{r, 1}\bigr) \bigl\| \Dot{z}(s)\bigr\|^{2} ds + 4\lambda_{r, 1} \int_{T}^{t} s^{r} \beta(s) \langle z(s) - z_{*}, V(z(s))\rangle ds \\
        & +\frac{2}{3}\delta\theta^{2} \int_{t_{5}}^{t}s^{3r} \beta^{2}(s) \| V(z(s))\|^{2} ds \\
        \leq & \: \int_{T}^{t} s^{r} \bigl(\alpha - r s^{r - 1} - \lambda_{r, 1}\bigr) \bigl\| \Dot{z}(s)\bigr\|^{2} ds \\
        & +4\lambda_{r, 1} \delta\theta \int_{T}^{t} s^{r} \Bigl[ \bigl(1 - 2 r\theta s^{r - 1}\bigr)\beta(s) - \theta s^{r} \Dot{\beta}(s)\Bigr] \langle z(s) - z_{*}, V(z(s))\rangle ds \\
        &+ \frac{2}{3} \theta \int_{T}^{t} s^{3} \beta(s) \Bigl[ \bigl(1 - 2 r\theta s^{r - 1}\bigr)\beta(s) - \theta s^{r} \Dot{\beta}(s)\Bigr] \| V(z(s))\|^{2} ds \\
        \leq & \: -\int_{T}^{t} \frac{d}{ds}\mathcal{E}_{\lambda_{r, 1}}(s) ds + 2\lambda_{r, 1} r(1 - r) \int_{t_{5}}^{t} s^{r - 2} \|z(s) - z_{*}\|^{2} ds \\
        \leq &\: \mathcal{E}_{\lambda_{r, 1}}(T) + 2\lambda_{r, 1} r(1 - r) \int_{T}^{t} s^{r - 2} \|z(s) - z_{*}\|^{2} ds, 
    \end{align*}
    where in the last inequality we drop the nonpositive term $- \mathcal{E}_{\lambda_{r, 1}}(t)$. The second summand in the last line vanishes in the cases $r = 0$ and $r = 1$, and in the case $r\in (0, 1)$ we are assuming the boundedness of the trajectory in order for this integral to be finite as $t\to +\infty$. This produces the following statements for $r\in [0, 1]$:
    \begin{align}
        \int_{T}^{+\infty} t^{r} \bigl\| \Dot{z}(t)\bigr\|^{2} dt &< +\infty, \label{eq: int 1} \\
        \int_{T}^{+\infty} t^{r} \beta(t) \langle z(t) - z_{*}, V(z(t))\rangle dt &< +\infty, \label{eq: int 2}\\
        \int_{T}^{+\infty} t^{3r} \beta^{2}(t) \| V(z(t))\|^{2} dt &< +\infty. \label{eq: int 3}
    \end{align}

    Now, we proceed to show the $o$ rates. According to the formula for the energy function, for $\rho = r$, $r\in [0, 1]$ and $t\geq T$ we have  
    \begin{align}
        \mathcal{E}_{\lambda_{r, 2}}(t) - \mathcal{E}_{\lambda_{r, 1}}(t) = & \: \frac{1}{2}\Biggl[ \Bigl\|2\lambda_{r, 2} (z(t) - z_{*}) + t^{r}\Bigl[ 2\Dot{z}(t) + \theta t^{r} \beta(t) V(z(t))\Bigr]\Bigl\|^{2} \nonumber\\
        &\quad -\Bigl\|2\lambda_{r, 1} (z(t) - z_{*}) + t^{r}\Bigl[ 2\Dot{z}(t) + \theta t^{r} \beta(t) V(z(t))\Bigr]\Bigl\|^{2}\Biggr] \nonumber\\
        & +2\Bigl[ \lambda_{r, 2} \bigl( \alpha - r t^{r - 1} - \lambda_{r, 2}\bigr) - \lambda_{r, 1} \bigl( \alpha - r t^{r - 1} - \lambda_{r, 1}\bigr)\Bigr] \| z(t) - z_{*}\|^{2} \nonumber\\
        &+ 2(\lambda_{r, 2} - \lambda_{r, 1}) \theta t^{2r} \beta(t) \langle z(t) - z_{*}, V(z(t))\rangle \nonumber\\
        = &\: \frac{1}{2}\Biggl[ 4 \bigl(\lambda_{r, 2}^{2} - \lambda_{r, 1}^{2}\bigr)\| z(t) - z_{*}\|^{2} + 4(\lambda_{r, 2} - \lambda_{r, 1}) \bigl\langle z(t) - z_{*}, 2\Dot{z}(t) + \theta t^{r} \beta(t) v(z(t))\bigr\rangle\Biggr] \nonumber\\
        &+ \Bigl[ 2(\lambda_{r, 2} - \lambda_{r, 1})\bigl( \alpha - r t^{r - 1}\bigr) - 2\bigl( \lambda_{r, 2}^{2} - \lambda_{r, 1}^{2}\bigr)\Bigr] \| z(t) - z_{*}\|^{2} \nonumber\\
        &+ 2(\lambda_{r, 2} - \lambda_{r, 1}) \theta t^{2r} \beta(t) \langle z(t) - z_{*}, V(z(t))\rangle \nonumber\\
        = & \: 4(\lambda_{r, 2} - \lambda_{r, 1}) \left[ \frac{1}{2}\bigl(\alpha - r t^{r - 1}\bigr) \| z(t) - z_{*}\|^{2} + t^{r} \bigl\langle z(t) - z_{*}, \Dot{z}(t) + \theta t^{r} \beta(t) V(z(t))\bigr\rangle\right]. \label{eq: difference of energy functions}
    \end{align}
    Define, for $t\geq T$, 
    \begin{equation}\label{eq: formula for p}
        p(t) := \frac{1}{2}\bigl(\alpha - rt^{r - 1}\bigr)\|z(t) - z_{*}\|^{2} + t^{r} \bigl\langle z(t) - z_{*}, \Dot{z}(t) + \theta t^{r} \beta(t) V(z(t))\bigr\rangle. 
    \end{equation}
    Since we already established that $\lim_{t\to +\infty} \mathcal{E}_{\lambda_{r, i}}$ exists for $i=1, 2$, and $\lambda_{r, 2} - \lambda_{r, 1}\neq 0$, we obtain the existence of $\lim_{t\to +\infty}p(t)$. We may express $\mathcal{E}_{\lambda_{r, 1}}(t)$ as 
    \begin{align*}
        \mathcal{E}_{\lambda_{r, 1}}(t) = & \: \frac{1}{2} \Bigl\| 2\lambda_{r, 1} (z(t) - z_{*}) + t^{r} \Bigl[ 2\Dot{z}(t) + \theta t^{r} \beta(t) V(z(t))\Bigr]\Bigr\|^{2} \\
        &+ 2\lambda_{r, 1} \bigl(\alpha - r t^{r - 1} - \lambda_{r, 1}\bigr) \|z(t) - z_{*}\|^{2} \\
        &+ 2\lambda_{r, 1}\theta t^{2r} \beta(t) \langle z(t) - z_{*}, V(z(t))\rangle + \frac{\theta^{2}}{2} t^{4r} \beta^{2}(t) \| V(z(t))\|^{2} \\
        = &\: \frac{1}{2}\Biggl[ 4\lambda_{r, 1}^{2} \|z(t) - z_{*}\|^{2} + 2\lambda_{r, 1} t^{r}\bigl\langle z(t) - z_{*}, 2\Dot{z}(t) + \theta t^{r} \beta(t) V(z(t))\bigr\rangle \\
        &\quad\quad + t^{2r}\Bigl\| 2\Dot{z}(t) + \theta t^{r} \beta(t) V(z(t))\Bigr\|^{2}\Biggr] \\
        &+ 2\lambda_{r, 1} \bigl( \alpha - rt^{r - 1}\bigr) \| z(t) - z_{*}\|^{2} + 2\lambda_{r, 1}^{2} \| z(t) - z_{*}\|^{2} + 2\lambda_{r, 1}\theta t^{2r} \beta(t) \langle z(t) - z_{*}, V(z(t))\rangle \\
        &+ \frac{\theta^{2}}{2} t^{4r} \beta^{2}(t) \| V(z(t))\|^{2} \\
        = & \: 4\lambda_{r, 1} p(t) + \frac{t^{2r}}{2} \Bigl\| 2\Dot{z}(t) + \theta t^{r} \beta(t) V(z(t))\Bigr\|^{2} + \frac{\theta^{2}}{2} t^{4r} \beta^{2}(t) \| V(z(t))\|^{2} \\
        = &\: 4\lambda_{r, 1} p(t) + t^{2r} \Bigl\| \Dot{z}(t) + \theta t^{r} \beta(t) V(z(t))\Bigr\|^{2} + t^{2r} \bigl\| \Dot{z}(t)\bigr\|^{2}. 
    \end{align*}
    Define $h(t) := t^{2r} \bigl\| \Dot{z}(t) + \theta t^{r} \beta(t)V(z(t))\bigr\|^{2} + t^{2r} \bigl\| \Dot{z}(t)\bigr\|^{2}$. Since both $\lim_{t\to +\infty} \mathcal{E}_{\lambda_{r, 1}}(t)$ and \\$\lim_{t\to +\infty} p(t)$ exists, so does $\lim_{t\to +\infty} h(t)$. Observe that
    \[
        \int_{T}^{t} \frac{1}{s^{r}} h(s) ds \leq 2\int_{T}^{t} s^{r} \bigl\| \Dot{z}(s)\bigr\|^{2} ds + 2\theta^{2} \int_{T}^{t} s^{3r} \beta^{2}(s) \| V(z(s))\|^{2} ds, 
    \]
    and, according to \eqref{eq: int 1} and \eqref{eq: int 3}, the integrals on the right-hand side become finite as $t\to +\infty$. Therefore, $\int_{T}^{+\infty} \frac{1}{t^{r}}h(t) dt < +\infty$, which gives
    \[
        \lim_{t\to +\infty} h(t) = 0, 
    \]
    which in particular yields
    \[
        \lim_{t\to +\infty} t^{2r}\bigl\| \Dot{z}(t)\bigr\|^{2} = 0, \text{ equivalently } \bigl\| \Dot{z}(t)\bigr\| = o\left(\frac{1}{t^{r}}\right) \text{ as }t\to +\infty.
    \]
    and plugging this into $\lim_{t\to +\infty} t^{2r}\bigl\| \Dot{z}(t) + \theta t^{r} \beta(t) V(z(t))\bigr\|^{2} = 0$ produces
    \[
        \lim_{t\to +\infty} t^{4r} \beta^{2}(t) \| V(z(t))\|^{2} = 0, \text{ equivalently } \| V(z(t))\| = o\left(\frac{1}{t^{2r} \beta(t)}\right) \text{ as }t\to +\infty. 
    \]
    Using the previous result together with the boundedness of $\| z(t) - z_{*}\|$ and the Cauchy-Schwarz inequality yields 
    \[
        \langle z(t) - z_{*}, V(z(t))\rangle = o\left( \frac{1}{t^{2r} \beta(t)}\right) \text{ as }t\to +\infty. 
    \]
    as $t\to +\infty$. We have now shown statements (i) and (ii) of the theorem. 

    We now focus on the weak convergence of the trajectories. For this end, we will make use of Opial's lemma (see Lemma \ref{lem:Opial}). Define, for $t\geq T$, 
    \[
        q(t) := \frac{1}{2} \| z(t) - z_{*}\|^{2} + \theta \int_{T}^{t} s^{r} \beta(s) \langle z(s) - z_{*}, V(z(s))\rangle ds. 
    \]
    We distinguish between the cases $r\in [0, 1)$ and $r = 1$.
    
    \fbox{Case $r \in [0, 1)$:} recalling formula \eqref{eq: formula for p}, notice that 
    \begin{align*}
        & \:\alpha q(t) + t^{r} \Dot{q}(t) \\
        = & \: \frac{\alpha}{2} \| z(t) - z_{*}\|^{2} + t^{r} \bigl\langle z(t) - z_{*}, \Dot{z}(t) + \theta t^{r} \beta(t) V(z(t))\bigr\rangle + \theta \alpha \int_{T}^{t} s^{r} \beta(s) \langle z(s) - z_{*}, V(z(s))\rangle ds \\
        = & \: \frac{r}{2} t^{r - 1} \| z(t) - z_{*}\|^{2} + p(t) + \theta\alpha \int_{T}^{t} s^{r} \beta(s) \langle z(s) - z_{*}, V(z(s))\rangle ds.
    \end{align*}
    Because of \eqref{eq: int 2}, the integral in the previous sum converges as $t\to +\infty$. We have already established the existence of $\lim_{t\to +\infty} p(t)$, and using that $t^{r - 1} \to 0$ as $t\to +\infty$ and the boundedness of $\|z(t) - z_{*}\|$ allows us to deduce the existence of $\lim_{t\to +\infty} \alpha q(t) + t^{r} \Dot{q}(t)$. Using Lemma \ref{lem:q}, we obtain the existence of $\lim_{t\to +\infty} q(t)$, and again using that $\int_{T}^{t} s^{r} \beta(s) \langle z(s) - z_{*}, V(z(s))\rangle ds$ has a limit as $t\to +\infty$, we come to the existence of 
    \[
        \lim_{t \to +\infty} \| z(t) - z_{*}\|.
    \] 

    \fbox{Case $r = 1$:} we repeat the reasoning of the previous case. We obtain the existence of \\
    $\lim_{t\to +\infty} (\alpha - 1)q(t) + t \Dot{q}(t)$, and once again we use Lemma \ref{lem:q} to produce the existence of $\lim_{t\to +\infty} \| z(t) - z_{*}\|$. 

    In both cases, we have verified the first condition of Opial's Lemma. For the second condition, let $\overline{z}$ be a sequential cluster point of the trajectory $z(t)$ as $t\to +\infty$, which means there exists a sequence $(t_{n})_{n\in\N} \subseteq [t_{0}, +\infty)$ such that $t_{n} \to +\infty$ as $n\to +\infty$ and  
    \[
        z(t_{n}) \rightharpoonup \overline{z} \quad \text{as} \quad n\to +\infty,
    \]
    where $\rightharpoonup$ denotes weak convergence in $\mathcal{H}$. The convergence rate attained for $\| V(z(t))\|$ ensures $V(z(t_{n})) \to 0$ as $n\to +\infty$. Since the graph of $V$ is sequentially closed in $\mathcal{H}^{\text{weak}} \times \mathcal{H}^{\text{strong}}$, this finally allows us to conclude that 
    \[
        V(\overline{z}) = 0, 
    \]
    thereby fulfilling the second condition of Opial's Lemma. We have thus concluded the proof of this theorem. 
\end{proof}

\subsection{Linearly constrained convex minimization as a particular case}
Recall the problem stated in the introduction
\begin{equation}\label{eq:constrained minimization 2}
    \begin{array}{rl}
		\min & f \left( x \right), \\
		\text{subject to} 	& Ax = b, 
	\end{array}
\end{equation}
where $\mathcal{X}$ and $\mathcal{Y}$ are real Hilbert spaces, $b\in \mathcal{Y}$, $A :\mathcal{X} \to \mathcal{Y}$ is a bounded linear operator and $f : \mathcal{X} \to \R$ is a convex and continuously differentiable function. Formulating \eqref{eq:system r=s in (0, 1)} in terms of the operator $V$ given by 
\begin{equation}\label{eq:operator linearly constrained minimization}
    V(x, \lambda) = \Bigl(\nabla f(x) + A^{*}\lambda, b - Ax\Bigr)
\end{equation}
yields the primal-dual system 
\begin{equation}\label{eq:minimization primal-dual system}
    \begin{dcases}
        \Ddot{x}(t) + \frac{\alpha}{t^{r}}\Dot{x}(t) + \theta t^{r} \beta(t)\frac{d}{dt}\nabla f(x(t)) + \beta(t) A^{*} \bigl(\lambda(t) + \theta t^{r}\Dot{\lambda}(t)\bigr) + \beta(t)\nabla f(x(t)) &= 0, \\
        \Ddot{\lambda}(t) + \frac{\alpha}{t^{r}}\Dot{\lambda}(t) - \beta(t)\Bigl[ A\bigl(\Dot{x}(t) + \theta t^{r} \Dot{x}(t)\bigr) - b\Bigr] &= 0. 
    \end{dcases}
\end{equation}
Let $(x_{*}, \lambda_{*}) \in \mathcal{X} \times \mathcal{Y}$ be a primal-dual solution to \eqref{eq:constrained minimization 2}. Equivalently, $(x_{*}, \lambda_{*})$ is zero of $V$, or a saddle point of the Lagrangian $\mathcal{L}(x, \lambda) = f(x) + \langle \lambda, Ax - b\rangle$. Let $t\mapsto (x(t), \lambda(t))$ be a solution to \eqref{eq:minimization primal-dual system}. Using the gradient inequality for convex functions, we have, for every $t\geq t_{0}$, 
\begin{align}
    0 &\leq \mathcal{L}(x(t), \lambda_{*}) - \mathcal{L}(x_{*}, \lambda(t)) \nonumber\\
    &= f(x(t)) - f(x_{*}) + \langle \lambda_{*}, Ax(t) - b\rangle \nonumber \\
    &\leq \langle x(t) - x_{*}, \nabla f(x(t))\rangle + \langle \lambda_{*}, Ax(t) - b\rangle \nonumber \\
    &= \langle x(t) - x_{*}, \nabla f(x(t))\rangle + \langle x(t) - x_{*}, A^{*}\lambda(t)\rangle + \langle \lambda(t) - \lambda_{*}, b - Ax(t)\rangle \nonumber \\
    &= \Bigl\langle (x(t), \lambda(t)) - (x_{*}, \lambda_{*}), V(x(t), \lambda(t))\Bigr\rangle. \label{eq:ineq for primal-dual gap}
\end{align}
By exploiting the results in Theorem \ref{eq:main theorem continuous}, and by taking into consideration additional potential information in this particular case, we come to obtain the following result. 
\begin{theorem}
    Suppose that $\alpha, \theta > 0$, $r\in [0, 1]$ and that $\beta : [t_{0}, +\infty) \to (0, +\infty)$ satisfy the same assumptions as in Theorem \ref{eq:main theorem continuous}. Let $(x_{*}, \lambda_{*}) \in \mathcal{X}\times \mathcal{Y}$ be a primal-dual solution to \eqref{eq:constrained minimization 2} and $t\mapsto (x(t), \lambda(t))$ a solution to \eqref{eq:minimization primal-dual system}. Consider the following convergence rates as $t\to +\infty$:
    \begin{gather*}
        \mathcal{L}(x(t), \lambda_{*}) - \mathcal{L}(x_{*}, \lambda(t)) = o\left(\frac{1}{t^{2\rho} \beta(t)}\right), \quad |f(x(t)) - f(x_{*})| = o\left(\frac{1}{t^{2\rho} \beta(t)}\right), \\
        \| Ax(t) - b\| = o\left(\frac{1}{t^{\rho + r} \beta(t)}\right), \quad \bigl\| \Dot{x}(t)\bigr\| = o\left(\frac{1}{t^{\rho}}\right), \quad \bigl\| \Dot{\lambda}(t)\bigr\| = o\left(\frac{1}{t^{\rho}}\right).
    \end{gather*}
    Additionally, if $\nabla f$ is $L$-Lipschitz continuous for some $L > 0$, consider the following rates as $t\to +\infty$:
    \[
        \| \nabla f(x(t)) - \nabla f(x_{*})\| = o\left( \frac{1}{t^{\rho} \sqrt{\beta(t)}}\right), \quad \| A^{*}(\lambda(t) - \lambda_{*})\| = o\left(\frac{1}{t^{\rho} \sqrt{\beta(t)}}\right).
    \]
    The following statements are true:    
    \begin{enumerate}[\rm (i)]
        \item 
        \label{thm:minimization primal-dual:i}
        If $r \in (0, 1)$, the above rates hold for $\rho \in (0, r)$. Furthermore, if $t\mapsto (x(t), \lambda(t))$ is bounded, then these rates hold for $\rho = r$. 
        \item 
        \label{thm:minimization primal-dual:ii}
        if $r = 0$ or $r = 1$, then $t\mapsto (x(t), \lambda(t))$ is bounded; moreover, the above rates hold for $\rho = 0$ (if $r = 0$) and for $\rho = 1$ (if $r = 1$). 
        \item 
        \label{thm:minimization primal-dual:iii}
        If $r = 0$ or $r = 1$ or $\Bigl(r\in (0, 1)$ and $t\mapsto (x(t), \lambda(t))$ is bounded $\Bigr)$, then $(x(t), \lambda(t))$ converges weakly to a primal-dual solution to \eqref{eq:constrained minimization 2} as $t\to +\infty$. 
    \end{enumerate}
\end{theorem}
\begin{proof}
    Parts \ref{thm:minimization primal-dual:i} and \ref{thm:minimization primal-dual:ii} are a direct consequence of Theorem \ref{eq:main theorem continuous} applied to the operator $V$ defined in \eqref{eq:operator linearly constrained minimization}. Indeed, endow $\mathcal{X} \times \mathcal{Y}$ with the norm $\|(x, \lambda)\| := \|x\| + \|\lambda\|$. First, from \eqref{eq:convergence rates of the main theorem} we obtain
    \[
        \bigl\| \Dot{x}(t)\bigr\| + \bigl\| \Dot{\lambda}(t)\bigr\| = \bigl\| \bigl(\Dot{x}(t), \Dot{\lambda}(t)\bigr)\bigr\| = o\left(\frac{1}{t^{\rho}}\right) \quad \text{as }t\to +\infty.
    \]
    From \eqref{eq:ineq for primal-dual gap} and \eqref{eq:convergence rates of the main theorem} we get
    \begin{equation}\label{eq:convergence rate primal-dual gap}
        0 \leq \mathcal{L}(x(t), \lambda_{*}) - \mathcal{L}(x_{*}, \lambda(t)) \leq \Bigl\langle (x(t), \lambda(t)) - (x_{*}, \lambda_{*}), V(x(t), \lambda(t))\Bigr\rangle = o\left(\frac{1}{t^{2\rho}\beta(t)}\right) \quad \text{as }t\to +\infty.
    \end{equation}
    Again using \eqref{eq:convergence rates of the main theorem}, we produce
    \begin{equation}\label{eq:rate norm primal-dual operator}
        \| \nabla f(x(t)) + A^{*}\lambda(t)\| + \| Ax(t) - b\| = \| V(x(t), \lambda(t))\| = o\left(\frac{1}{t^{\rho + r}\beta(t)}\right) 
    \end{equation}
    and thus 
    \begin{equation}\label{eq:feasibility gap}
        \| Ax(t) - b\| = o\left(\frac{1}{t^{\rho + r}\beta(t)}\right) \quad \text{as }t\to +\infty.
    \end{equation}
    Combining \eqref{eq:convergence rate primal-dual gap} and \eqref{eq:feasibility gap} we come to 
    \begin{align}
        |f(x(t)) - f(x_{*})| &\leq |f(x(t)) - f(x_{*}) + \langle \lambda_{*}, Ax(t) - b\rangle| + |\langle \lambda_{*}, Ax(t) - b\rangle| \nonumber\\
        &\leq \mathcal{L}(x(t), \lambda_{*}) - \mathcal{L}(x_{*}, \lambda(t)) + \|\lambda_{*}\| \| Ax(t) - b\| \nonumber\\
        &= o\left(\frac{1}{t^{2\rho}\beta(t)}\right) \quad \text{as }t\to +\infty, \label{eq:convergence rate for functional values}
    \end{align}
    since we always have $2\rho \leq \rho + r$. Assume now that $\nabla f$ is $L$-Lipschitz continuous. We have now a stronger gradient inequality
    \begin{equation}\label{eq:strong gradient inequality}
        0 \leq \frac{1}{2L} \| \nabla f(x(t)) - \nabla f(x_{*})\|^{2} \leq \langle x(t) - x_{*}, \nabla f(x(t))\rangle -  \bigl(f(x(t)) - f(x_{*}) \bigr).
    \end{equation}
    We will show that $\langle x(t) - x_{*}, \nabla f(x(t))\rangle = o\left(\frac{1}{t^{2\rho} \beta(t)}\right)$ as $t\to +\infty$. Indeed, we have 
    \begin{align*}
        &\:\langle x(t) - x_{*}, \nabla f(x(t))\rangle + \langle \lambda_{*}, Ax(t) - b\rangle \\
        = &\: \langle x(t) - x_{*}, \nabla f(x(t))\rangle + \langle x(t) - x_{*}, A^{*}\lambda(t)\rangle + \langle \lambda(t) - \lambda_{*}, b - Ax(t)\rangle \\
        = &\: \Bigl\langle (x(t), \lambda(t)) - (x_{*}, \lambda_{*}), V(x(t), \lambda(t))\Bigr\rangle  \\
        = &\: o\left(\frac{1}{t^{2\rho} \beta(t)}\right) \quad \text{as }t\to +\infty, 
    \end{align*}
    so using this together with \eqref{eq:feasibility gap} yields $\langle x(t) - x_{*}, \nabla f(x(t))\rangle = o\left(\frac{1}{t^{2\rho} \beta(t)}\right)$ as $t\to +\infty$ (again, recall that $2\rho \leq \rho + r$). Plugging this rate into \eqref{eq:strong gradient inequality} yields 
    \begin{equation}\label{eq:convergence rate for gradient}
        \| \nabla f(x(t)) - \nabla f(x_{*})\| = o\left(\frac{1}{t^{\rho} \sqrt{\beta(t)}}\right) \quad \text{as }t\to +\infty. 
    \end{equation}
    Now, 
    \begin{align*}
        \| A^{*}(\lambda(t) - \lambda_{*})\| &\leq \| \nabla f(x(t)) - \nabla f(x_{*}) + A^{*}\lambda(t) - A^{*}\lambda_{*}\| + \| \nabla f(x(t)) - \nabla f(x_{*})\| \\
        &= \| \nabla f(x(t)) + A^{*}\lambda(t)\| + \| \nabla f(x(t)) - \nabla f(x_{*})\|.
    \end{align*}
    According to \eqref{eq:rate norm primal-dual operator}, the left summand is of order $o\left(\frac{1}{t^{\rho + r}\beta(t)}\right)$ as $t\to +\infty$; according to \eqref{eq:convergence rate for gradient}, the right summand is of order $o\left(\frac{1}{t^{\rho}\sqrt{\beta(t)}}\right)$ as $t\to +\infty$. Since $\beta$ is nondecreasing on $[t_{0}, +\infty)$, we have $\sqrt{\beta(t)} = \frac{\beta(t)}{\sqrt{\beta(t)}}\leq \frac{\beta(t)}{\sqrt{\beta(t_{0})}}$. Since $t^{\rho} \leq t^{\rho + r}$ for large $t$, this allows us to write the inequality
    \[
        \frac{\sqrt{\beta(t_{0})}}{t^{\rho + r}\beta(t)} \leq \frac{1}{t^{\rho}\sqrt{\beta(t)}} \quad \text{for large enough }t,
    \]
    from which we deduce that 
    \[
        \| A^{*}(\lambda(t) - \lambda_{*})\| = o\left( \frac{1}{t^{\rho}\sqrt{\beta(t)}}\right) \quad \text{as }t\to +\infty. 
    \]
    We have thus shown parts \ref{thm:minimization primal-dual:i} and \ref{thm:minimization primal-dual:ii} of this theorem. Part \ref{thm:minimization primal-dual:iii}, i.e., the weak convergence of the trajectories to a primal-dual solution to \eqref{eq:constrained minimization 2}, is again a direct corollary of part \ref{thm:minimization primal-dual:iii} of Theorem \ref{eq:main theorem continuous}. 
\end{proof}

\begin{remark}
    The system \eqref{eq:minimization primal-dual system} resembles the one studied by He et al. in \cite{HeHuFang}, for the case where $r = s$. There, the authors consider an extrapolation parameter of the form $\theta t^{s}$, where $s\in [r, 1]$. The only difference between our system and theirs lies in the inclusion of the Hessian-driven damping term attached to the velocity of the primal trajectory $x(t)$. The rates we obtained are identical to those in \cite{HeHuFang}, but our system further allows us to show the weak convergence of the trajectories towards primal-dual solutions of the minimization problem. 
\end{remark}
\begin{remark}
    When $r = 1$ and $\beta(t)\equiv 1$, we obtain a system similar to the one addressed by Bo\c t and Nguyen in \cite{BotNguyen}. Again, our system features an extra Hessian-driven damping term, which does not appear in \cite{BotNguyen}. Our rates coincide with those in this work, but, as an interesting point, our system does not require to assume that $\nabla f$ is $L$-Lipschitz continuous to show the weak convergence of the generated trajectories. 
\end{remark}

\subsection{Some numerical experiments}
In this subsection, we will complement the theoretical results with two numerical examples. The first one is the minimization of a strongly convex function under linear constraints, and the second one is finding the saddle points of a certain convex-concave function. 
\begin{example}
    Consider the minimization problem 
    \begin{equation*}
        \begin{array}{rl}
	    \min & f(x_{1}, x_{2}, x_{3}, x_{4}) := (x_{1} - 1)^{2} + (x_{2} - 1)^{2} +        x_{3}^{2} + x_{4}^{2} \\
	    \textrm{subject to} 	& x_{1} - x_{2} - x_{3} = 0 \\
	    & x_{2} - x_{4} = 0.
        \end{array}
    \end{equation*}
    The optimality conditions can be calculated and lead to the primal-dual solution pair
    \[
        x_{*} = 
        \begin{bmatrix}
	   0.8 \\
	   0.6 \\
	   0.2 \\
	   0.6
        \end{bmatrix}
        \qquad
        \text{and}
        \qquad 
        \lambda_{*} = 
        \begin{bmatrix}
	   0.4 \\
	   1.2
        \end{bmatrix}.
    \]
    For $t \geq t_{0} = 1$, we plot the functional values as well as the feasibility gap along the trajectories generated by \eqref{eq:minimization primal-dual system}. In Figure \ref{fig:strongly convex beta(t) = 1} we have parameters $\alpha = 8$, $\theta = \frac{1}{4}$ and $\beta(t)\equiv 1$. The predicted convergence rates in this case are of $o\left(\frac{1}{t^{2r}}\right)$ as $t\to +\infty$, so as expected, we can see faster convergence behaviour as $r$ goes from $0.2$ to $1$. 
    
    \begin{figure}[H] 
        \minipage{0.5\textwidth}
        \includegraphics[width=\linewidth]{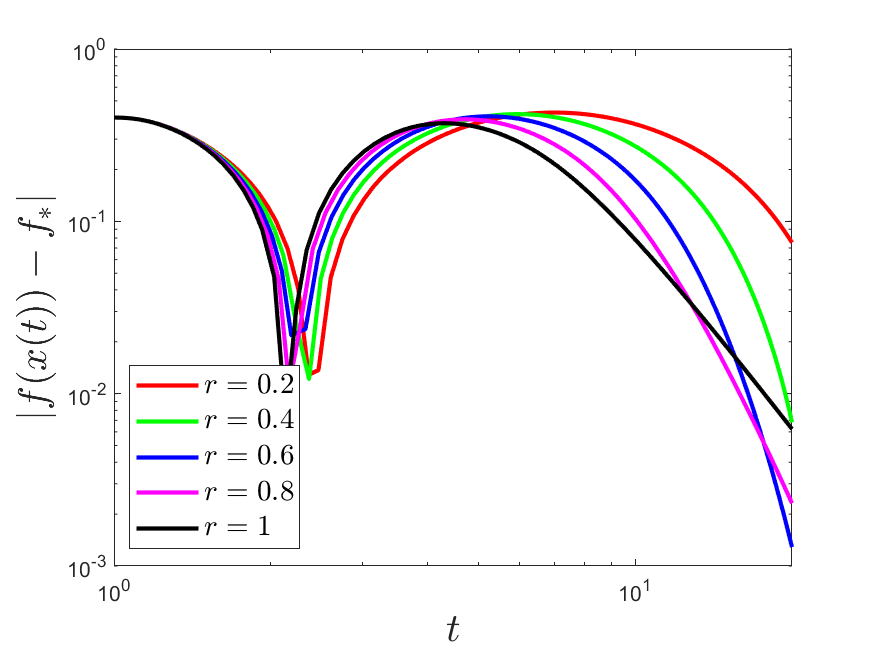}
        \endminipage\hfill
        \minipage{0.5\textwidth}
        \includegraphics[width=\linewidth]{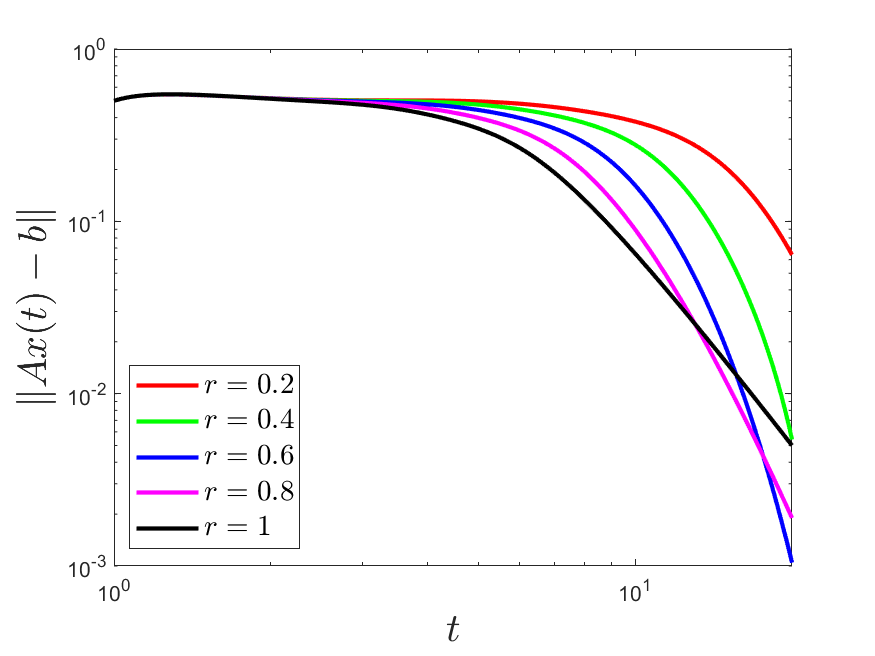}
        \endminipage
        \caption{Influence of $r$ on the functional values and feasibility gap along the trajectories, $\beta(t)\equiv 1$}
        \label{fig:strongly convex beta(t) = 1}
    \end{figure}

    In Figure \ref{fig:strongly convex beta(t) = exp} we plot the same quantities, but this time, we choose $\alpha = 8$, $\theta = \frac{1}{3}$, $\delta = 2$ and $\beta(t) =  \frac{1}{t^{2r}} \exp \left[ \left(\frac{1}{\theta} - \delta\right) \frac{t^{1 - r}}{1 - r}\right] $. Here, the predicted convergence rates are of $o\left(\exp\left[-\left(\frac{1}{\theta} - \delta\right)\frac{t^{1 - r}}{1 - r}\right]\right)$ as $t\to +\infty$. As expected, since $t^{1 - r}$ grows to infinity slower as $r$ approaches $1$, the convergence behaviour speeds up as $r$ goes from near $1$ to $0$.

    \begin{figure}[H] 
        \minipage{0.5\textwidth}
        \includegraphics[width=\linewidth]{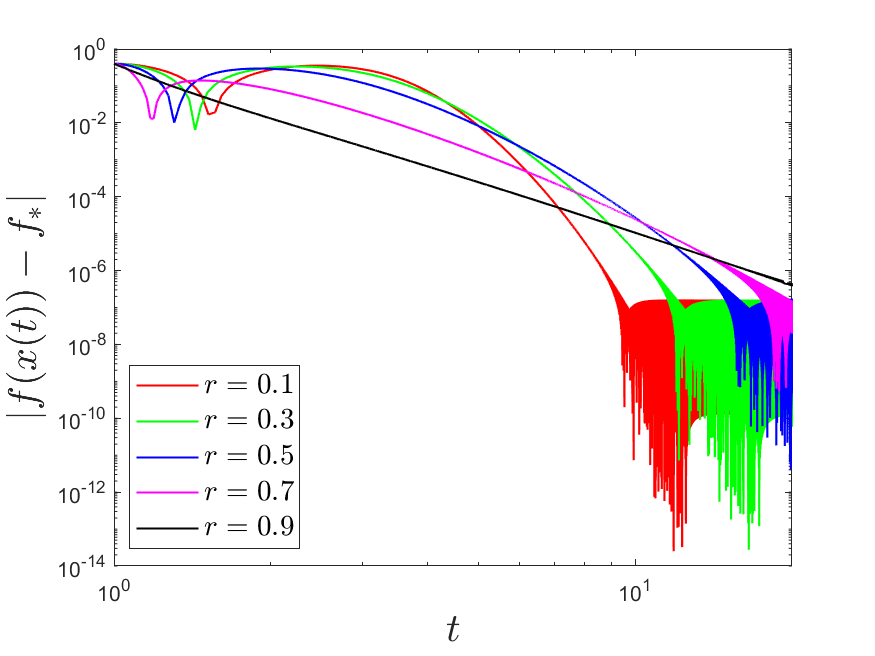}
        \endminipage\hfill
        \minipage{0.5\textwidth}
        \includegraphics[width=\linewidth]{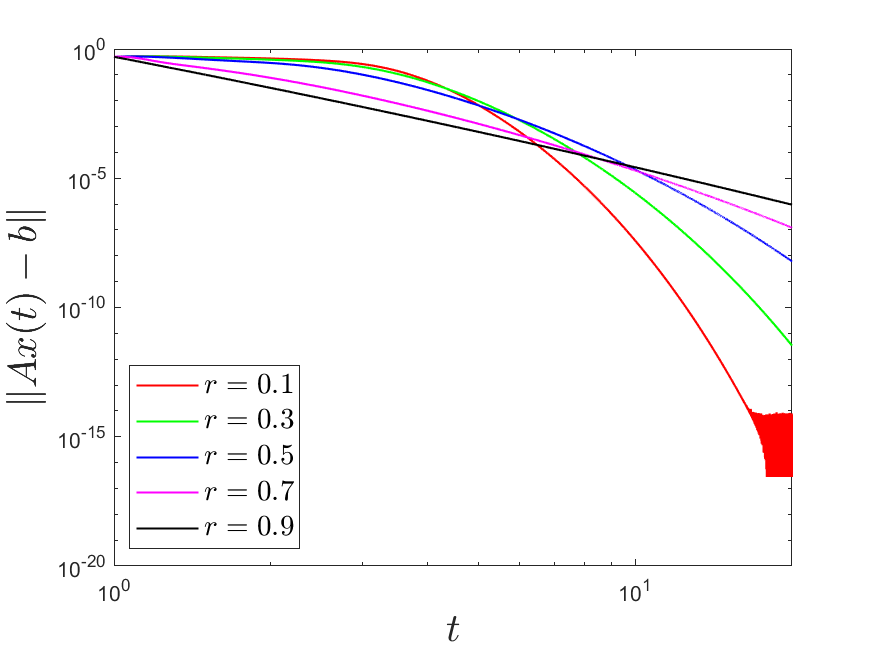}
        \endminipage
        \caption{Influence of $r$ on the functional values and feasibility gap along the trajectories, $\beta(t) = \frac{1}{t^{2r}} \exp \left[ \left(\frac{1}{\theta} - \delta\right) \frac{t^{1 - r}}{1 - r}\right]$}
        \label{fig:strongly convex beta(t) = exp}
    \end{figure}
\end{example}
\begin{example}
    We address the saddle point problem 
    \[
        \min_{x\in\R^{n}}\max_{y\in\R^{n}} \Phi(x, y) := \frac{1}{2} \langle x, Hx\rangle - \langle x, h\rangle - \langle y, Ax - b\rangle, 
    \]
    where 
    \[
        A := \frac{1}{4}
        \begin{bmatrix}
            & & & -1 & 1 \\
            & & \scalebox{-1}[1]{$\ddots$} & \scalebox{-1}[1]{$\ddots$} & \\
            & -1 & 1 & & \\
            -1 & 1 & & & \\
            1 & & & &
        \end{bmatrix} \in\R^{n\times n}, 
        \quad 
        H := 2A^{*} A, 
        \quad
        b := \frac{1}{4}
        \begin{bmatrix}
            1 \\
            1 \\
            \vdots \\
            1 \\
            1
        \end{bmatrix} \in \R^{n}, 
        \quad 
        h := \frac{1}{4}
        \begin{bmatrix}
            0 \\
            0 \\
            \vdots \\
            0 \\
            1
        \end{bmatrix} \in \R^{n}.
    \]
    The corresponding continuous and monotone operator whose zeros we wish to find is given by 
    \[
        V(x, y) := 
        \begin{bmatrix}
            \nabla_{x} \Phi(x, y) \\
            -\nabla_{y} \Phi(x, y)
        \end{bmatrix}
        = 
        \begin{bmatrix}
            Hx - h - A^{*}y \\
            Ax - b
        \end{bmatrix}.
    \]
    In Figure \ref{fig:operator}, for $t \geq t_{0} = 1$ we plot the norm of the operator $V$ along the trajectories generated by \eqref{eq:system r=s in (0, 1)}. We choose $\alpha = 8$, $\theta = \frac{1}{4}$ and $\beta(t) \equiv 1$ and $\beta(t) = t$ for the first and second plots respectively, and $\delta = 3$ and $\beta(t) =  \frac{1}{t^{2r}} \exp \left[ \left(\frac{1}{\theta} - \delta\right) \frac{t^{1 - r}}{1 - r}\right]$ for the third plot. The predicted rates for the first two plots are $o\left(\frac{1}{t^{2r}}\right)$ and $o\left(\frac{1}{t^{2r + 1}}\right)$ as $t\to +\infty$ respectively, so as expected we see an overall faster convergence behaviour in the second plot, and in both plots the quantities approach zero faster $r$ moves from near $0$ to $1$. As we argued in the first example, in the third plot we have rates of $o\left(\exp\left[-\left(\frac{1}{\theta} - \delta\right)\frac{t^{1 - r}}{1 - r}\right]\right)$ as $t\to +\infty$, so we see a trend of faster convergence the smaller $r$ is. 
    \begin{figure}[H] 
        \minipage{0.33\textwidth}
        \includegraphics[width=\linewidth]{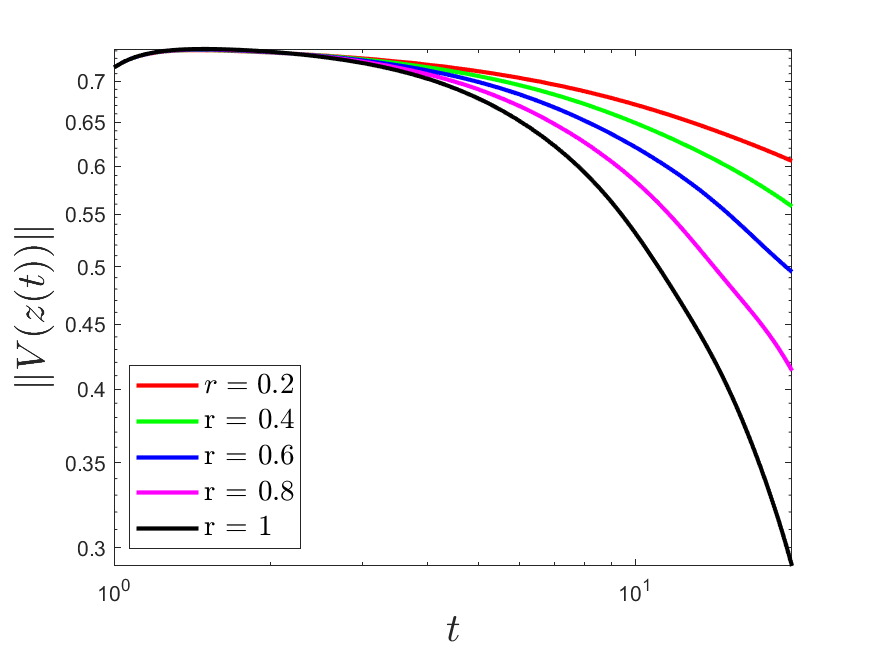}
        \endminipage\hfill
        \minipage{0.33\textwidth}
        \includegraphics[width=\linewidth]{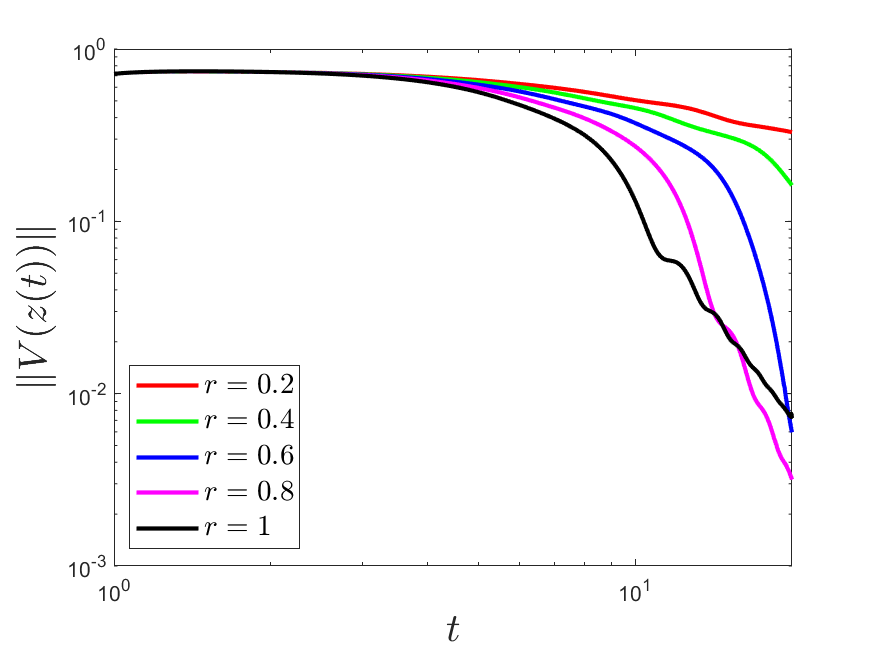}
        \endminipage\hfill
        \minipage{0.33\textwidth}
        \includegraphics[width=\linewidth]{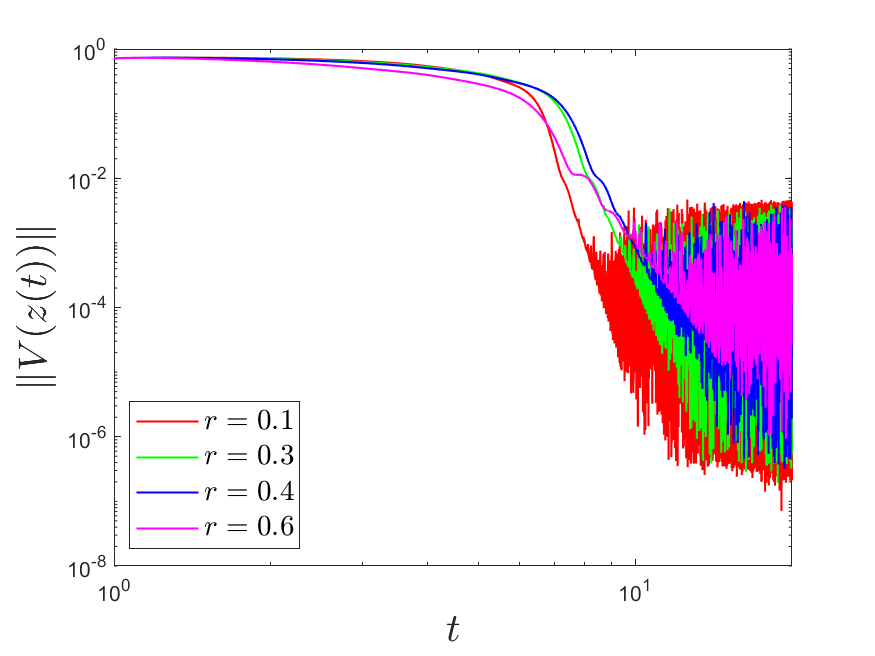}
        \endminipage
        \caption{Influence of r on the norm of the operator along the trajectories, with parameter $\beta$ chosen respectively as $\beta(t)\equiv 1$, $\beta(t) = t$ and $\beta(t) = \frac{1}{t^{2r}} \exp \left[ \left(\frac{1}{\theta} - \delta\right) \frac{t^{1 - r}}{1 - r}\right]$}
        \label{fig:operator}
    \end{figure}
\end{example}

\section{An implicit discretization of the continuous system}

\subsection{Discretization of an equivalent first-order reformulation}

\noindent Given the second order system 
\begin{equation}\label{eq:continuous system}
    \Ddot{z}(t) + \frac{\alpha}{t^{r}} \Dot{z}(t) + \theta t^{r} \beta(t) \frac{d}{dt}V(z(t)) + \beta(t) V(z(t)) = 0, 
\end{equation}
we readily see it is equivalent to the following first-order reformulation:  
\begin{equation}\label{eq:first order reformulation}
    \begin{dcases}
        \Dot{u}(t) &= 2t^{r} \Bigl[ \bigl(2r\theta t^{r - 1} - 1 \bigr) \beta(t) + \theta t^{r} \Dot{\beta}(t)\Bigr] V(z(t)) + 2r (1 - r) t^{r - 2} z(t) \\
        u(t) &= 2\bigl(\alpha - r t^{r - 1}\bigr)z(t) + 2t^{r} \Dot{z}(t) + 2\theta t^{2r} \beta(t) V(z(t)). 
    \end{dcases}
\end{equation}
A seemingly useful observation to obtain a discretization which works is to write 
\begin{equation}\label{eq:2r(1 - r)t^(r - 2)z(t)}
    2r(1 - r) t^{r - 2} z(t) = 2r \frac{d}{dt}\bigl(-t^{r - 1}\bigr) z(t).
\end{equation}
Taking \eqref{eq:2r(1 - r)t^(r - 2)z(t)} into account, we consider the following discretization of \eqref{eq:first order reformulation}: 
\begin{equation}\label{eq:discretization}
    \begin{dcases}
        u^{k + 1} - u^{k} &= 2k^{r}\Bigl[\bigl(2r\theta k^{r - 1} - 1\bigr)\beta_{k} + \theta k^{r}\bigl(\beta_{k} - \beta_{k - 1}\bigr) \Bigr] V(z^{k + 1}) + 2r \Bigl[ k^{r - 1} - (k + 1)^{r - 1}\Bigr]z^{k + 1} \\
        u^{k + 1} &= 2\bigl[ \alpha - r(k + 1)^{r - 1}\bigr]z^{k + 1} + 2 (k + 1)^{r} \bigl(z^{k + 1} - z^{k}\bigr) + 2\theta (k + 1)^{2r} \beta_{k} V(z^{k + 1}). 
    \end{dcases}
\end{equation}
Now, from the second line of \eqref{eq:discretization} we deduce 
\begin{align*}
    u^{k + 1} - u^{k} = & \: 2\bigl[ \alpha - r(k + 1)^{r - 1}\bigr]z^{k + 1} + 2(k + 1)^{r}\bigl(z^{k + 1} - z^{k}\bigr) + 2\theta (k + 1)^{2r} \beta_{k}V(z^{k + 1}) \\
    &- 2\bigl[ \alpha - rk^{r - 1}\bigr]z^{k} - 2k^{r} \bigl(z^{k} - z^{k - 1}\bigr) - 2\theta k^{2r} \beta_{k - 1}V(z^{k}) \\
    = &\: 2\alpha \bigl( z^{k + 1} - z^{k}\bigr) - 2r \Bigl[ (k + 1)^{r - 1} z^{k + 1} - k^{r - 1} z^{k}\Bigr] + 2(k + 1)^{r} \bigl( z^{k + 1} - z^{k}\bigr) \\
    &- 2k^{r} \bigl( z^{k} - z^{k - 1}\bigr) + 2\theta \Bigl[ (k + 1)^{2r}\beta_{k} V(z^{k + 1}) - k^{2r}\beta_{k - 1} V(z^{k})\Bigr] \\
    = &\: 2\alpha \bigl( z^{k + 1} - z^{k}\bigr) - 2r \Bigl[ \bigl((k + 1)^{r - 1} - k^{r - 1}\bigr) z^{k + 1}  + k^{r - 1}\bigl(z^{k + 1} - z^{k}\bigr)\Bigr] \\
    &+ 2(k + 1)^{r} \bigl( z^{k + 1} - z^{k}\bigr) - 2k^{r} \bigl( z^{k} - z^{k - 1}\bigr) + 2\theta \Bigl[ (k + 1)^{2r} \beta_{k}V(z^{k + 1}) - k^{2r} \beta_{k - 1}V(z^{k})\Bigr] \\
    = &\: 2r \Bigl[ k^{r - 1} - (k + 1)^{r - 1}\Bigr]z^{k + 1} + 2\Bigl[ \alpha - rk^{r - 1} + (k + 1)^{r}\Bigr] \bigl( z^{k + 1} - z^{k}\bigr) - 2k^{r} \bigl( z^{k} - z^{k - 1}\bigr) \\
    &+ 2\theta \Bigl[ (k + 1)^{2r}\beta_{k} - k^{2r}\beta_{k - 1}\Bigr] V(z^{k + 1}) + 2\theta k^{2r} \beta_{k - 1} \Bigl[ V(z^{k +1}) - V(z^{k})\Bigr] \\
    = &\: 2k^{r} \Bigl[\bigl(2r\theta k^{r - 1} - 1\bigr)\beta_{k} + \theta k^{r} \bigl( \beta_{k} - \beta_{k - 1}\bigr) \Bigr]V(z^{k + 1}) + 2r \Bigl[ k^{r - 1} - (k + 1)^{r - 1}\Bigr]z^{k + 1}, 
\end{align*}
where the last equality comes from the first line of \eqref{eq:discretization}. Cancelling the second summand in the last line, we come to 
\begin{align}
    &\:2\Bigl[ \alpha - rk^{r - 1} + (k + 1)^{r}\Bigr] \bigl( z^{k + 1} - z^{k}\bigr) - 2k^{r} \bigl( z^{k} - z^{k - 1}\bigr) +  2\theta \Bigl[ (k + 1)^{2r}\beta_{k} - k^{2r}\beta_{k - 1}\Bigr] V(z^{k + 1}) \nonumber\\
    &+ 2\theta k^{2r} \beta_{k - 1}\Bigl[ V(z^{k +1}) - V(z^{k})\Bigr] \nonumber\\
    = &\: 2k^{r} \Bigl[ \bigl( 2r\theta k^{r - 1} - 1\bigr)\beta_{k} + \theta k^{r} \bigl(\beta_{k} - \beta_{k - 1}\bigr) \Bigr] V(z^{k + 1}). \label{eq:algorithm}
\end{align}
Rearranging the terms of the previous equality leads to the formulation \eqref{eq:algorithm intro} presented in the introduction:
\begin{align*}
    z^{k + 1} = & \: z^{k} + \frac{k^{r}}{\alpha - rk^{r - 1} + (k + 1)^{r}} \bigl( z^{k} - z^{k - 1}\bigr) - \frac{\theta k^{2r} \beta_{k - 1}}{\alpha - rk^{r - 1} + (k + 1)^{r}} \Bigl[ V(z^{k + 1}) - V(z^{k})\Bigr] \\
    &- \frac{\theta \Bigl\{ \Bigl[ (k + 1)^{2r} - k^{2r}\Bigr] - 2rk^{2r - 1}\Bigr\} + k^{r}}{\alpha - rk^{r - 1} + (k + 1)^{r}} \beta_{k} V(z^{k + 1}).
\end{align*}

\subsection{Convergence rates and weak convergence of iterates}

Before we start with the analysis, we need to remark some inequalities involving the sequence $(\beta_{k})_{k\geq 1}$. 
\begin{remark}\label{rem:growth condition}
    In \eqref{eq:discrete growth intro}, the growth conditions reads $\sup_{k\geq k_{0}} k^{r} \left(\frac{\beta_{k} - \beta_{k - 1}}{\beta_{k}} + \frac{2r}{k}\right) < \frac{1}{2\theta} < \frac{1}{\theta}$. Throughout the majority of the proofs, we will use a couple of inequalities involving the sequence $\bigl( \beta_{k}\bigr)_{k\geq 0}$ which are entailed by assuming the supremum is strictly less than $\frac{1}{\theta}$. More precisely, there exists $\delta > 0$ such that $\delta < \frac{1}{\theta}$ and 
    \begin{equation}\label{eq:supremum < than 1/theta - delta}
        \sup_{k\geq k_{0}} k^{r} \left[ \frac{\beta_{k} - \beta_{k - 1}}{\beta_{k}} + \frac{2r}{k}\right] < \frac{1}{\theta} - \delta.
    \end{equation}
    It follows that $\frac{k^{r} \bigl( \beta_{k} - \beta_{k - 1}\bigr)}{\beta_{k}} + 2rk^{r - 1} < \frac{1}{\theta} - \delta$ and therefore
    \begin{equation}\label{eq:G1}
        \bigl(2r\theta k^{r - 1} - 1\bigr)\beta_{k} + \theta k^{r} \bigl(\beta_{k} - \beta_{k - 1}\bigr) \leq -\delta \theta \beta_{k} < 0 \qquad \forall k\geq k_{0}.
    \end{equation}
    The previous inequality also produces
    \begin{equation}\label{eq:G2}
        \theta k^{r} \bigl( \beta_{k} - \beta_{k - 1}\bigr) \leq \bigl(1 - 2r\theta k^{r - 1} - \delta\theta\bigr) \beta_{k} \qquad \forall k\geq k_{0}. 
    \end{equation}
    Regrouping the terms with $\beta_{k}$ and $\beta_{k - 1}$ separately yields $\left(2r k^{r - 1} - \frac{1}{\theta} + \delta + k^{r}\right)\beta_{k} \leq k^{r} \beta_{k - 1}$ and thus
    \begin{equation}\label{eq:G3}
        \beta_{k} \leq \frac{k^{r}}{2r k^{r - 1} - \frac{1}{\theta} + \delta + k^{r}} \beta_{k - 1} \leq M_{\beta} \beta_{k - 1},
    \end{equation}
    where $M_{\beta} > 1$ but may be taken arbitrarily close to $1$ for $k$ large enough. 
\end{remark}
\begin{remark}
    As we stated in the introduction, when $r \in (0, 1)$ then it is possible to choose $\beta_{k}$ such that $k^{2r} \beta_{k}$ grows exponentially. We claim that if $0 < \delta < \frac{1}{2\theta}$, then 
    \[
        \beta_{k} = \frac{1}{k^{2r}} \cdot e^{\left(\frac{1}{2\theta} - \delta\right) \frac{k^{1 - r}}{1 - r}}
    \]
    satisfies growth condition \eqref{eq:discrete growth intro}. Indeed, first we notice that $\beta_{k} = \beta(k)$, where $\beta(t)$ is defined in continuous-time as 
    \[
        \beta(t) = \frac{1}{t^{2r}} \cdot e^{\left( \frac{1}{2\theta} - \delta\right) \frac{t^{1 - r}}{1 - r}}. 
    \]
    By the mean value theorem, there exists $\xi_{k} \in (k - 1, k)$ such that 
    \[
        \beta_{k} - \beta_{k - 1} = \beta(k) - \beta(k - 1) = \Dot{\beta}(\xi_{k}) = \exp \left[\left(\frac{1}{2\theta} - \delta\right) \frac{\xi_{k}^{1 - r}}{1 - r}\right] \left[ -\frac{2r}{\xi_{k}^{2r + 1}} + \left( \frac{1}{2\theta} - \delta\right) \frac{1}{\xi_{k}^{3r}}\right].
    \]
    Therefore, for every $k$ we have 
    \begin{align}
        k^{r} \left(\frac{\beta_{k} - \beta_{k - 1}}{\beta_{k}} + \frac{2r}{k}\right) = &\: \exp\left[ \left(\frac{1}{2\theta} - \delta\right) \frac{\xi_{k}^{1 - r} - k^{1 - r}}{1 - r}\right] \left( \frac{1}{2\theta} - \delta\right) \frac{k^{3r}}{\xi_{k}^{3r}} \label{eq:exponential discrete line 1}\\
        &- \exp\left[ \left(\frac{1}{2\theta} - \delta\right) \frac{\xi_{k}^{1 - r} - k^{1 - r}}{1 - r}\right] \frac{2r k^{3r}}{\xi_{k}^{2r + 1}} + \frac{2r}{k^{1 - r}}. \label{eq:exponential discrete line 2}
    \end{align}
    First, we look at line \eqref{eq:exponential discrete line 1}. Again by the mean value theorem, we have $k^{1 - r} - \xi_{k}^{1 - r} = (1 - r)\tau_{k}^{-r}(k - \xi_{k})$ for some $\tau_{k}\in (\xi_{k}, k)\subseteq (k - 1, k)$. It follows that $\frac{1}{k^{r}} \leq \frac{1}{\tau_{k}^{r}} \leq\frac{1}{(k - 1)^{r}}$ and thus $\xi_{k}^{1 - r} - k^{1 - r} \to 0$ as $k\to +\infty$. Since $1 \leq \frac{k^{3r}}{\xi_{k}^{3r}} \leq \left(\frac{k}{k - 1}\right)^{3r}$, we have $\frac{k^{3r}}{\xi_{k}^{3r}}\to 1$ as $k\to +\infty$. All in all, line \eqref{eq:exponential discrete line 1} approaches $\frac{1}{2\theta} - \delta$ as $k\to +\infty$. Regarding line \eqref{eq:exponential discrete line 2}, as we just saw, the exponential term approaches $1$ and $\frac{2r}{k^{1 - r}} \to 0$ as $k\to +\infty$. Since $\frac{1}{k^{1 - r}}\leq \frac{k^{3r}}{\xi_{k}^{2r + 1}} \leq \frac{k^{3r}}{k^{2r + 1}}$, we have $\frac{2r k^{3r}}{k^{2r + 1}} \to 0$ as $k\to +\infty$. Summarizing, we have $k^{r} \left(\frac{\beta_{k} - \beta_{k - 1}}{\beta_{k}} + \frac{2r}{k}\right) \to  \frac{1}{2\theta} - \delta$ as $k\to +\infty$, which means that for large enough $k_{0}$ the supremum condition \eqref{eq:supremum < than 1/theta - delta} is fulfilled. 
\end{remark}

\noindent\textbf{Energy function}: For analizing the convergence properties of the implicit algorithm \eqref{eq:algorithm}, we make use of the following discrete energy function: 

\begin{align}
    \mathcal{E}_{\lambda}^{k} := &\:\frac{1}{2} \left\| k^{\rho - r}2\lambda (z^{k} - z_{*}) + 2k^{\rho} \bigl(z^{k} - z^{k - 1}\bigr) + \theta k^{\rho + r} \beta_{k - 1}V(z^{k})\right\|^{2} \label{eq:energy function line 1}\\
    &+ 2\lambda k^{2(\rho - r)}\bigl( \alpha - (2\rho - r)k^{r - 1} - \lambda\bigr) \| z^{k} - z_{*}\|^{2} \label{eq:energy function line 2}\\
    &+ 2\lambda \theta k^{2\rho} \beta_{k - 1}\langle z^{k} - z_{*}, V(z^{k})\rangle \label{eq:energy function line 3}\\
    &+ \frac{\theta^{2}}{2} (k + 1)^{2r} k^{2\rho} \beta_{k}\beta_{k -  1} \| V(z^{k})\|^{2}. \label{eq:energy function line 4} 
\end{align}
Here, $\rho $ and $\lambda$ are taken as follows: 
\begin{itemize}
    \item If $r\in (0, 1)$, then $0 < \lambda < \alpha$ and $0 < \rho < r$;
    \item If $r = 1$, then $0 < \lambda < \alpha - 1$ and $\rho = r = 1$. 
\end{itemize}
The following theorem shows that the $\mathcal{O}$ convergence rates, obtained as an intermediate step for the continuous dynamical system, carry over to the discrete setting. 

\begin{theorem}\label{th:big O rates}
    Suppose that $\alpha, \theta > 0$, $r\in (0, 1]$ and $(\beta_{k})_{k\geq 1} \subseteq (0, +\infty)$ satisfy the assumptions laid down in subsection \eqref{subsec:our work}. Let $z^{0}, z^{1}$ be initial points in $\mathcal{H}$, let $\bigl(z^{k}\bigr)_{k\geq 2}$ be the sequence generated by the implicit algorithm \eqref{eq:algorithm}, equivalently, by \eqref{eq:algorithm intro}, and let $z_{*}$ be a zero of $V$. Consider the following convergence rates as $k\to +\infty$:  
    \[
        \| V(z^{k})\| = \mathcal{O}\left( \frac{1}{k^{\rho + r} \beta_{k}}\right) \quad \langle z^{k + 1} - z_{*}, V(z^{k + 1})\rangle = \mathcal{O}\left(\frac{1}{k^{2\rho} \beta_{k}}\right), \quad \bigl\| z^{k} - z^{k - 1}\bigr\| = \mathcal{O}\left(\frac{1}{k^{\rho}}\right).
    \]
    The following statements are true: 
    \begin{enumerate}[\rm (i)]
        \item If $r\in (0, 1)$, the above rates hold for $\rho \in (0, r)$. Furthermore, if $\bigl(z^{k}\bigr)_{k\geq 2}$ is bounded, then these rates hold for $\rho = r$. 
        \item If $r = 1$, then $\bigl(z^{k}\bigr)_{k\geq 2}$ is bounded; moreover, the above rates hold for $\rho = 1$. 
    \end{enumerate}
\end{theorem}

To produce the proof for this theorem, we first exhibit two intermediate results. 

\begin{lemma}\label{lem:discrete derivative of energy function}
    Under the same hypotheses of Theorem \ref{th:big O rates}, for large enough $k$ it holds
    \begin{align*}
        &\:\mathcal{E}_{\lambda}^{k + 1} - \mathcal{E}_{\lambda}^{k} \\
        \leq &\: \Biggl\{ 2\lambda^{2} \Bigl[ (k + 1)^{2(\rho - r)} - k^{2(\rho - r)}\Bigr] + 2\lambda (2\rho - r) \Bigl[ k^{r - 1} - (k + 1)^{r - 1}\Bigr] k^{2(\rho - r)} \\
        &\quad\quad+ 2\lambda \bigl(\alpha - (2\rho - r)(k + 1)^{r - 1} - \lambda\bigr) \Bigl[ (k + 1)^{2(\rho - r)} - k^{2(\rho - r)}\Bigr] \Biggr\} \| z^{k + 1} - z_{*}\|^{2} \\
        &+ \Biggl\{ 8\lambda (r - \rho) k^{2(\rho - r)} + 4\lambda \Bigl[ (k + 1)^{2(\rho - r)} - k^{2(\rho - r)}\Bigr](k + 1)^{r}\Biggr\} \bigl\langle z^{k + 1} - z_{*}, z^{k + 1} - z^{k}\bigr\rangle \\
        &+ 4\lambda k^{2\rho - r} \Bigl[ \bigl(2r\theta k^{r - 1} - 1\bigr)\beta_{k} + \theta k^{r} \bigl( \beta_{k} - \beta_{k - 1}\bigr)\Bigr] \langle z^{k + 1} - z_{*}, V(z^{k + 1})\rangle \\
        &+ 4k^{2\rho - r} \bigl( \lambda + rk^{r - 1} - \alpha\bigr)\bigl\| z^{k + 1} - z^{k}\bigr\|^{2} \\
        &+ \Biggl\{ k^{2(\rho - r )}\Bigl[ 2\eta_{k} (k + 1)^{r} + 2\theta \bigl( \lambda + rk^{r - 1} - \alpha\bigr)(k + 1)^{2r} \beta_{k} - 2\bigl( \lambda + rk^{r - 1} - \alpha\bigr)\eta_{k} \Bigr] \\
        &\quad\quad+ 2\theta\Bigl[ (k + 1)^{2(\rho - r)} - k^{2(\rho - r)}\Bigr](k + 1)^{3r} \beta_{k} +2\lambda\theta k^{2\rho}\beta_{k - 1} \Biggr\}\bigl\langle z^{k + 1} - z^{k}, V(z^{k + 1})\bigr\rangle \\
        &+ \Biggl\{ k^{2(\rho - r)} \left[\theta (k + 1)^{2r} \eta_{k}\beta_{k} - \frac{1}{2} \eta_{k}^{2}\right] + \frac{\theta^{2}}{2} \Bigl[ (k + 1)^{2(\rho - r)} - k^{2(\rho - r)}\Bigr](k + 1)^{4r}\beta_{k}^{2} \\
        &\quad\quad +\frac{\theta^{2}}{2}\Bigl[ (k + 2)^{2r}(k + 1)^{2\rho} \beta_{k + 1} \beta_{k} - (k + 1)^{2r} k^{2\rho} \beta_{k} \beta_{k - 1}\Bigr]\Biggr\} \| V(z^{k + 1})\|^{2} \\ 
        &+ \theta k^{2\rho} \beta_{k - 1} \eta_{k} \bigl\langle V(z^{k + 1}), V(z^{k + 1}) - V(z^{k})\bigr\rangle \\
        &- k^{2(\rho - r)} \frac{\theta^{2}}{2} \Bigl[ k^{4r} \beta_{k - 1}^{2} + (k + 1)^{2r} k^{2r} \beta_{k}\beta_{k - 1}\Bigr] \bigl\| V(z^{k + 1}) - V(z^{k})\bigr\|^{2}, 
    \end{align*}
    where 
    \[
        \eta_{k} := 2k^{r} \Bigl[ \bigl(2 r \theta k^{r - 1} - 1\bigr)\beta_{k} + \theta k^{r} \bigl( \beta_{k} - \beta_{k - 1}\bigr)\Bigr] - \theta \Bigl[ (k + 1)^{2r} \beta_{k} - k^{2r} \beta_{k - 1}\Bigr].
    \]
\end{lemma}
\begin{proof}
    First of all, just as in the continuous case, notice that $\mathcal{E}_{\lambda}^{k}$ is eventually nonnegative. Indeed, go back to line \eqref{eq:energy function line 2}. If $r\in (0, 1)$, then $0 < \lambda < \alpha$ together with the fact that $k^{r - 1}\to 0$ as $k\to +\infty$ gives $\alpha - (2\rho - r)k^{r - 1} - \lambda \geq 0$ for large enough $k$. If $r = 1$, then $\alpha - (2\rho - r)k^{r - 1} - \lambda \equiv \alpha - 1 - \lambda > 0$ by assumption. We will compute the difference $\mathcal{E}_{\lambda}^{k + 1} - \mathcal{E}_{\lambda}^{k}$ line by line. 

    Line \eqref{eq:energy function line 1}: Define, for each $k$,
    \[
        u_{\lambda}^{k} := 2\lambda (z^{k} - z_{*}) + 2k^{r} \bigl(z^{k} - z^{k - 1}\bigr) + \theta k^{2r} \beta_{k - 1}V(z^{k}).
    \]
    Notice that \eqref{eq:energy function line 1} can be written as 
    \[
        \frac{1}{2}\bigl\| k^{\rho - r} u_{\lambda}^{k}\bigr\|^{2}.
    \]
    We have 
    \begin{align}
        &\:\frac{1}{2}\bigl\| (k + 1)^{\rho - r} u_{\lambda}^{k + 1}\bigr\|^{2} - \frac{1}{2}\bigl\| k^{\rho - r} u_{\lambda}^{k}\bigr\|^{2} \nonumber\\
        = &\: k^{2(\rho - r)} \left[\frac{1}{2}\|u_{\lambda}^{k + 1}\|^{2} - \frac{1}{2}\| u_{\lambda}^{k}\|^{2}\right] + \frac{1}{2} \Bigl[ (k + 1)^{2(\rho - r)} - k^{2(\rho - r)}\Bigr] \| u_{\lambda}^{ k + 1}\|^{2} \nonumber\\
        = &\: k^{2(\rho - r)} \left[ \bigl\langle u_{\lambda}^{k + 1}, u_{\lambda}^{k + 1} - u_{\lambda}^{k}\bigr\rangle - \frac{1}{2} \bigl\| u_{\lambda}^{k + 1} - u_{\lambda}^{k}\bigr\|^{2}\right] + \frac{1}{2} \Bigl[ (k + 1)^{2(\rho - r)} - k^{2(\rho - r)}\Bigr] \| u_{\lambda}^{ k + 1}\|^{2}, \label{eq:difference of energy function line 1}
    \end{align}
    so we first compute $u_{\lambda}^{k + 1} - u_{\lambda}^{k}$: 
    \begin{align*}
        u_{\lambda}^{k + 1} - u_{\lambda}^{k} = & \: 2\lambda(z^{k + 1} - z_{*}) + 2(k + 1)^{r} \bigl( z^{k + 1} - z^{k}\bigr) + \theta (k + 1)^{2r} \beta_{k}V(z^{k + 1}) \\
        &- 2\lambda(z^{k} - z_{*}) - 2k^{r} \bigl( z^{k} - z^{k - 1}\bigr) - \theta k^{2r} \beta_{k - 1}V(z^{k}) \\
        = &\: 2\lambda \bigl( z^{k + 1} - z^{k}\bigr) + 2\Bigl[ (k + 1)^{r} \bigl( z^{k + 1} - z^{k}\bigr) - k^{r} \bigl( z^{k} - z^{k - 1}\bigr)\Bigr] \\
        &+ \theta \Bigl[ (k + 1)^{2r}\beta_{k} V(z^{k + 1}) - k^{2r} \beta_{k - 1}V(z^{k})\Bigr] \\
        = &\: 2\bigl( \lambda + rk^{r - 1} - \alpha\bigr)\bigl(z^{k + 1} - z^{k}\bigr) + 2 \Bigl[ \alpha - rk^{r - 1} + (k + 1)^{r}\Bigr]\bigl( z^{k + 1} - z^{k}\bigr) - 2k^{r}\bigl( z^{k} - z^{k - 1}\bigr) \\
        &+ \theta  \Bigl[ (k + 1)^{2r}\beta_{k} - k^{2r}\beta_{k - 1}\Bigr] V(z^{k + 1}) + \theta k^{2r} \beta_{k - 1}\Bigl[ V(z^{k + 1}) - V(z^{k})\Bigr] \\
        \text{use \eqref{eq:algorithm}} \:\: = &\: 2\bigl( \lambda + r k^{r - 1} - \alpha\bigr) \bigl(z^{k + 1} - z^{k}\bigr) + 2k^{r} \Bigl[\bigl(2r\theta k^{r - 1} - 1\bigr)\beta_{k} + \theta k^{r} \bigl(\beta_{k} - \beta_{k - 1}\bigr) \Bigr] V(z^{k + 1}) \\
        &-\theta \Bigl[ (k + 1)^{2r}\beta_{k} - k^{2r}\beta_{k - 1}\Bigr] V(z^{k + 1}) - \theta k^{2r} \beta_{k - 1}\Bigl[ V(z^{k + 1}) - V(z^{k})\Bigr] \\
        = &\: 2\bigl( \lambda + r k^{r - 1} - \alpha\bigr)\bigl( z^{k + 1} - z^{k}\bigr) -\theta k^{2r} \beta_{k - 1} \Bigl[ V(z^{k + 1}) - V(z^{k})\Bigr] \\
        &+ \underbrace{\Bigl\{ 2k^{r} \Bigl[\bigl( 2r\theta k^{r - 1} - 1\bigr)\beta_{k} + \theta k^{r}\bigl( \beta_{k} - \beta_{k - 1}\bigr) \Bigr] - \theta \Bigl[(k + 1)^{2r} - k^{2r}\Bigr]\Bigr\}}_{ = \eta_{k}} V(z^{k + 1}).
    \end{align*}
    With this at hand, we compute $\langle u_{\lambda}^{k + 1}, u_{\lambda}^{k + 1} - u_{\lambda}^{k}\rangle$: 
    \begin{align*}
        &\:\langle u_{\lambda}^{k + 1}, u_{\lambda}^{k + 1} - u_{\lambda}^{k}\rangle \\
        = &\: \Bigl\langle 2\lambda (z^{k + 1} - z_{*}) + 2(k + 1)^{r} \bigl( z^{k + 1} - z^{k}\bigr) + \theta (k + 1)^{2r} \beta_{k} V(z^{k + 1}), \\
        &\quad\quad 2\bigl( \lambda + r k^{r - 1} - \alpha\bigr) \bigl( z^{k + 1} - z^{k}\bigr) + \eta_{k} V(z^{k + 1}) - \theta k^{2r} \beta_{k - 1}\Bigl[ V(z^{k + 1}) - V(z^{k})\Bigr] \Bigr\rangle \\
        = &\: 4\lambda \bigl(\lambda + r k^{r - 1} - \alpha\bigr) \bigl\langle z^{k + 1} - z_{*}, z^{k + 1} - z^{k}\bigr\rangle + 2\lambda \eta_{k} \langle z^{k + 1} - z_{*}, V(z^{k + 1})\rangle \\
        &- 2\lambda\theta k^{2r} \beta_{k - 1} \bigl\langle z^{k + 1} - z_{*}, V(z^{k + 1}) - V(z^{k})\bigr\rangle + 4\bigl( \lambda + r k^{r - 1} - \alpha\bigr) (k + 1)^{r} \bigl\| z^{k + 1} - z^{k}\bigr\|^{2} \\
        &+ \Bigl[ 2\eta_{k} (k + 1)^{r} + 2\theta \bigl( \lambda + r k^{r - 1} - \alpha\bigr)(k + 1)^{2r} \beta_{k}\Bigr]\bigl\langle z^{k + 1} - z^{k}, V(z^{k + 1})\bigr\rangle \\
        &-2\theta (k + 1)^{r} k^{2r} \beta_{k - 1}\bigl\langle z^{k + 1} - z^{k}, V(z^{k + 1}) - V(z^{k})\bigr\rangle + \theta (k + 1)^{2r} \eta_{k} \beta_{k}\| V(z^{k + 1})\|^{2} \\
        &- \theta^{2} (k + 1)^{2r} k^{2r} \beta_{k} \beta_{k - 1}\bigl\langle V(z^{k + 1}), V(z^{k + 1}) - V(z^{k})\bigr\rangle. 
    \end{align*}
    Now, $\|u_{\lambda}^{k + 1} - u_{\lambda}^{k}\|^{2}$:
    \begin{align*}
        \:\| u_{\lambda}^{k + 1} - u_{\lambda}^{k}\|^{2} = &\: \Bigl\| 2\bigl(\lambda + r k^{r - 1} - \alpha\bigr)\bigl( z^{k + 1} - z^{k}\bigr) + \eta_{k} V(z^{k + 1}) - \theta k^{2r} \beta_{k - 1}\Bigl[ V(z^{k + 1}) - V(z^{k})\Bigr]\Bigr\|^{2} \\
        = &\: 4\bigl( \lambda + r k^{r - 1} - \alpha\bigr)^{2} \bigl\| z^{k + 1} - z^{k}\bigr\|^{2} + \eta_{k}^{2} \| V(z^{k + 1})\|^{2} \\
        &+ \theta^{2} k^{4r} \beta_{k - 1}^{2} \bigl\| V(z^{k + 1}) - V(z^{k})\bigr\|^{2} + 4\bigl( \lambda + r k^{r - 1} - \alpha\bigr) \eta_{k} \bigl\langle z^{k + 1} - z^{k}, V(z^{k + 1})\bigr\rangle \\
        &- 4\theta \bigl( \lambda + r k^{r - 1} - \alpha\bigr) k^{2r} \beta_{k - 1}\bigl\langle z^{k + 1} - z^{k}, V(z^{k + 1}) - V(z^{k})\bigr\rangle \\
        &- 2\theta \eta_{k} k^{2r} \beta_{k - 1} \bigl\langle V(z^{k + 1}), V(z^{k + 1}) - V(z^{k})\bigr\rangle. 
    \end{align*}
    Now, recalling \eqref{eq:difference of energy function line 1}, we put everything together:
    \begin{align*}
        &\:\frac{1}{2}\bigl\| (k + 1)^{\rho - r} u_{\lambda}^{k + 1}\bigr\|^{2} - \frac{1}{2}\bigl\| k^{\rho - r} u_{\lambda}^{k}\bigr\|^{2} \\
        = &\: k^{2(\rho - r)} \Biggl\{ 4\lambda \bigl( \lambda + r k^{r - 1} - \alpha\bigr) \bigl\langle z^{k + 1} - z_{*}, z^{k + 1} - z^{k}\bigr\rangle + 2\lambda \eta_{k} \langle z^{k + 1} - z_{*}, V(z^{k + 1})\rangle \\
        &\quad\quad - 2\lambda\theta k^{2r} \beta_{k - 1} \bigl\langle z^{k + 1} - z_{*}, V(z^{k + 1}) - V(z^{k})\bigr\rangle \\
        &\quad\quad+ \Bigl[ 4\bigl( \lambda + r k^{r - 1} - \alpha\bigr)(k + 1)^{r} - 2\bigl( \lambda + rk^{r - 1} - \alpha\bigr)^{2}\Bigr] \bigl\| z^{k + 1} - z^{k}\bigr\|^{2} \\
        &\quad\quad+ \Bigl[ 2\eta_{k} (k + 1)^{r} + 2\theta \bigl( \lambda + r k^{r - 1} - \alpha\bigr)(k + 1)^{2r} \beta_{k} - 2\bigl( \lambda + r k^{r - 1} - \alpha\bigr)\eta_{k}\Bigr] \bigl\langle z^{k + 1} - z^{k}, V(z^{k + 1})\bigr\rangle \\
        &\quad\quad+ \Bigl[ -2\theta (k + 1)^{r} k^{2r}\beta_{k - 1} + 2\theta \bigl( \lambda + r k^{r - 1} - \alpha\bigr) k^{2r} \beta_{k - 1}\Bigr] \bigl\langle z^{k + 1} - z^{k}, V(z^{k + 1}) - V(z^{k})\bigr\rangle \\
        &\quad\quad+ \left[ \theta(k + 1)^{2r} \eta_{k} \beta_{k} - \frac{1}{2}\eta_{k}^{2}\right] \| V(z^{k + 1})\|^{2} -\frac{\theta^{2}}{2} k^{4r} \beta_{k - 1}^{2} \bigl\| V(z^{k + 1}) - V(z^{k})\bigr\|^{2} \\
        &\quad\quad+ \Bigl[ -\theta^{2}(k + 1)^{2r}k^{2r}\beta_{k} \beta_{k - 1} + \theta k^{2r}\eta_{k} \beta_{k - 1}\Bigr] \bigl\langle V(z^{k + 1}), V(z^{k + 1}) - V(z^{k})\bigr\rangle \Biggr\} \\
        &+ \frac{1}{2} \Bigl[ (k + 1)^{2(\rho - r)} - k^{2(\rho - r)}\Bigr] \Biggl\{  4\lambda^{2} \| z^{k + 1} \! - z_{*}\|^{2} \! + \! 4(k + 1)^{r} \bigl\| z^{k + 1} \! - z^{k}\bigr\|^{2} \! + \!\theta^{2} (k + 1)^{4r} \beta_{k}^{2} \| V(z^{k + 1})\|^{2} \\
        &\quad\quad + 8\lambda (k + 1)^{r} \bigl\langle z^{k + 1} - z_{*}, z^{k + 1} - z^{k}\bigr\rangle + 4\lambda\theta (k + 1)^{2r} \beta_{k}\langle z^{k + 1} - z_{*}, V(z^{k + 1})\rangle \\
        &\quad\quad+ 4\theta (k + 1)^{3r} \beta_{k} \langle z^{k + 1} - z_{k}, V(z^{k + 1})\rangle \Biggr\}. 
    \end{align*}

    Line \eqref{eq:energy function line 2}: we have 
    \begin{align*}
        &\:2\lambda (k \! + \! 1)^{2(\rho - r)} \bigl( \alpha \! - \! (2\rho \! - \! r) (k \! + \! 1)^{r - 1} \! - \lambda\bigr) \| z^{k + 1} \! - z_{*}\|^{2} \! - 2\lambda k^{2(\rho -  r)} \bigl( \alpha \! - \!(2\rho \! - \! r) k^{r - 1} \! - \lambda\bigr) \| z^{k} \! - z_{*}\|^{2} \\
        = &\:k^{2(\rho - r)} \Biggl\{ 2\lambda (2\rho - r) \Bigl[ k^{r - 1} - r (k + 1)^{r - 1} \Bigr] \| z^{k + 1} - z_{*}\|^{2}  - 2\lambda \bigl( \alpha - (2\rho - r)k^{r - 1} - \lambda\bigr)\bigl\| z^{k + 1} - z^{k}\bigr\|^{2}\\
        &\qquad\qquad + 4\lambda \bigl( \alpha - (2\rho - r) k^{r - 1} - \lambda\bigr) \bigl\langle z^{k + 1} - z_{*}, z^{k + 1} - z^{k + 1} - z^{k}\bigr\rangle \Biggr\}\\
        &+ 2\lambda \Bigl[ (k + 1)^{2(\rho - r)} - k^{2(\rho - r)}\Bigr] \left( \alpha - (2\rho - r)(k + 1)^{r - 1} - \lambda \right) \left\lVert z^{k + 1} - z_{*} \right\rVert ^{2}. 
    \end{align*}

    Line \eqref{eq:energy function line 3}: we can write 
    \begin{align*}
        &\:2\lambda \theta (k + 1)^{2\rho} \beta_{k}\bigl\langle z^{k + 1} - z_{*}, V(z^{k + 1})\bigr\rangle - 2\lambda\theta k^{2\rho} \beta_{k - 1} \bigl\langle z^{k} - z_{*}, V(z^{k})\bigr\rangle \\
        = &\: 2\lambda\theta \Bigl[\! (k \! + \! 1)^{2\rho}\beta_{k} \! - \! k^{2\rho}\beta_{k - 1} \!\Bigr] \! \bigl\langle z^{k + 1} \! - z_{*}, V(z^{k + 1})\bigr\rangle \! + \! 2\theta k^{2\rho} \beta_{k - 1} \! \Bigl[ \! \bigl\langle z^{k + 1} \! - z_{*}, V(z^{k + 1})\bigr\rangle \! - \! \bigl\langle z^{k} \! - z_{*}, V(z^{k})\bigr\rangle \!\Bigr] \\
        = &\: 2\lambda\theta \Bigl[ (k + 1)^{2\rho}\beta_{k} - k^{2\rho}\beta_{k - 1}\Bigr] \bigl\langle z^{k + 1} - z_{*}, V(z^{k + 1})\bigr\rangle + 2\lambda\theta k^{2\rho} \beta_{k - 1}\bigl\langle z^{k + 1} - z_{*}, V(z^{k + 1}) - V(z^{k})\bigr\rangle \\
        &+ 2\lambda\theta k^{2\rho} \beta_{k - 1}\bigl\langle z^{k + 1} - z^{k}, V(z^{k})\bigr\rangle \\
        = &\: 2\lambda\theta \Bigl[ (k + 1)^{2\rho}\beta_{k} - k^{2\rho}\beta_{k - 1}\Bigr] \langle z^{k + 1} - z_{*}, V(z^{k + 1})\rangle + 2\lambda\theta k^{2\rho} \beta_{k - 1}\bigl\langle z^{k + 1} - z_{*}, V(z^{k + 1}) - V(z^{k})\bigr\rangle \\
        &- 2\lambda\theta k^{2\rho} \beta_{k - 1}\bigl\langle z^{k + 1} - z^{k}, V(z^{k + 1}) - V(z^{k})\bigr\rangle + 2\lambda\theta k^{2\rho} \beta_{k - 1}\bigl\langle z^{k + 1} - z^{k}, V(z^{k + 1})\bigr\rangle.
    \end{align*}

    Line \eqref{eq:energy function line 4}: we obtain 
    \begin{align*}
        &\:\frac{\theta^{2}}{2} \Bigl[ (k + 2)^{2r} (k + 1)^{2\rho} \beta_{k + 1}\beta_{k} \| V(z^{k + 1})\|^{2} - (k + 1)^{2r} k^{2\rho} \beta_{k} \beta_{k - 1} \| V(z^{k})\|^{2}\Bigr] \\
        = &\: \frac{\theta^{2}}{2}\Bigl[(k + 2)^{2r} (k + 1)^{2\rho} \beta_{k + 1}\beta_{k} - (k + 1)^{2r} k^{2\rho} \beta_{k} \beta_{k - 1} \Bigr] \| V(z^{k + 1})\|^{2} \\
        &+ \frac{\theta^{2}}{2} (k + 1)^{2r} k^{2\rho} \beta_{k} \beta_{k - 1} \Bigl[ \| V(z^{k + 1})\|^{2} - \| V(z^{k})\|^{2}\Bigr] \\
        = &\: \frac{\theta^{2}}{2} \Bigl[ (k + 2)^{2r} (k + 1)^{2\rho} \beta_{k + 1} \beta_{k} - (k + 1)^{2r} k^{2\rho} \beta_{k} \beta_{k - 1}\Bigr] \| V(z^{k + 1})\|^{2} \\
        &+ (k + 1)^{2r} k^{2\rho} \beta_{k} \beta_{k - 1} \left[ \bigl\langle V(z^{k + 1}), V(z^{k + 1}) - V(z^{k})\bigr\rangle - \frac{1}{2} \bigl\| V(z^{k + 1}) - V(z^{k})\bigr\|^{2}\right] \\
        = &\: \frac{\theta^{2}}{2} \Bigl[ (k + 2)^{2r} (k + 1)^{2\rho} \beta_{k + 1} \beta_{k} - (k + 1)^{2r} k^{2\rho} \beta_{k} \beta_{k - 1}\Bigr] \| V(z^{k + 1})\|^{2} \\
        &+ \theta^{2} (k \! + \! 1)^{2r} k^{2\rho} \beta_{k} \beta_{k - 1} \bigl\langle V(z^{k + 1}), V(z^{k + 1}) - V(z^{k})\bigr\rangle - \frac{\theta^{2}}{2}(k \! + \! 1)^{2r} k^{2\rho} \beta_{k} \beta_{k - 1} \bigl\| V(z^{k + 1}) - V(z^{k})\bigr\|^{2}.
    \end{align*}

    Now, we combine everything: 
    \begin{align}
        &\:\mathcal{E}_{\lambda}^{k + 1} - \mathcal{E}_{\lambda}^{k} \nonumber\\
        = &\: \Biggl\{ 2\lambda^{2} \Bigl[ (k + 1)^{2(\rho - r)} - k^{2(\rho - r)}\Bigr] + 2\lambda (2\rho - r) k^{2(\rho - r)}\Bigl[ k^{r - 1} - (k + 1)^{r - 1}\Bigr] \nonumber\\
        &\quad\quad + 2\lambda \Bigl[ (k + 1)^{2(\rho - r)} - k^{2(\rho - r)}\Bigr]\bigl( \alpha - (2\rho - r)(k + 1)^{r - 1} - \lambda\bigr) \Biggr\} \| z^{k + 1} - z_{*}\|^{2} \label{eq:dE line 1}\\
        &+ \Biggl\{ 4\lambda k^{2(\rho - r)}\bigl( \lambda + r k^{r - 1} - \alpha\bigr) + 4\lambda k^{2(\rho - r)} \bigl( \alpha - (2\rho - r) k^{r - 1} - \lambda\bigr) \nonumber\\
        &\quad\quad + 4\lambda \Bigl[ (k + 1)^{2(\rho - r)} - k^{2(\rho - r)}\Bigr](k + 1)^{r}\Biggr\} \bigl\langle z^{k + 1} - z_{*}, z^{k + 1} - z^{k}\bigr\rangle \label{eq:dE line 2}\\
        &+ \Biggl\{ 2\lambda k^{2(\rho - r)} \eta_{k} + 2\lambda\theta \Bigl[ (k + 1)^{2\rho}\beta_{k} - k^{2\rho}\beta_{k - 1}\Bigr] \nonumber\\
        &\quad\quad+ 2\lambda\theta \Bigl[ (k + 1)^{2(\rho - r)} - k^{2(\rho - r)}\Bigr](k + 1)^{2r}\beta_{k} \Biggr\} \langle z^{k + 1} - z_{*}, V(z^{k + 1})\rangle \label{eq:dE line 3}\\
        &+ \Bigl[k^{2(\rho - r)} \bigl(-2\lambda\theta k^{2r}\beta_{k - 1}\bigr) + 2\lambda\theta k^{2\rho} \beta_{k - 1}\Bigr] \bigl\langle z^{k + 1} - z_{*}, V(z^{k + 1}) - V(z^{k})\bigr\rangle \label{eq:dE line 4}\\
        &+ \Biggl\{ k^{2(\rho - r)} \Bigl[4\bigl(\lambda + r k^{r - 1} - \alpha\bigr)(k + 1)^{r} - 2\bigl( \lambda + r k^{r - 1} - \alpha\bigr)^{2}\Bigr] + 2\lambda \Bigl[ (k + 1)^{2(\rho - r)} - k^{2(\rho - r)}\Bigr] \nonumber\\
        &\quad\quad- 2\lambda k^{2(\rho - r)} \bigl(\alpha - (2\rho - r) k^{r - 1} - \lambda\bigr)\Biggr\} \bigl\| z^{k + 1} - z^{k}\bigr\|^{2} \label{eq:dE line 5}\\
        &+ \Biggl\{ k^{2(\rho - r)}\Bigl[ 2\eta_{k} (k + 1)^{r} + 2\theta \bigl( \lambda + rk^{r - 1} - \alpha\bigr)(k + 1)^{2r} \beta_{k} - 2\bigl( \lambda + rk^{r - 1} - \alpha\bigr)\eta_{k}\Bigr] \nonumber\\
        &\quad\quad + 2\theta \Bigl[ (k + 1)^{2(\rho - r)} - k^{2(\rho - r)}\Bigr](k + 1)^{3r} \beta_{k} + 2\lambda\theta k^{2\rho} \beta_{k - 1}\Biggr\} \bigl\langle z^{k + 1} - z^{k}, V(z^{k + 1})\bigr\rangle \label{eq:dE line 6}\\
        &+ \Biggl\{ k^{2(\rho - r)}\Bigl[ -2\theta (k + 1)^{r} k^{2r} \beta_{k - 1} + 2\theta \bigl( \lambda + rk^{r - 1} - \alpha\bigr) k^{2r}\beta_{k - 1} \Bigr] \nonumber\\
        &\quad\quad- 2\lambda\theta k^{2\rho} \beta_{k - 1}\Biggr\} \bigl\langle z^{k + 1}  - z^{k}, V(z^{k + 1}) - V(z^{k})\bigr\rangle \label{eq:dE line 7}\\
        &+ \Biggl\{ k^{2(\rho - r)} \left[\theta (k + 1)^{2r} \eta_{k} \beta_{k} - \frac{1}{2} \eta_{k}^{2}\right] + \frac{\theta^{2}}{2}\Bigl[ (k + 1)^{2(\rho - r)} - k^{2(\rho - r)}\Bigr](k + 1)^{4r}\beta_{k}^{2} \nonumber\\
        &\quad\quad + \frac{\theta^{2}}{2} \Bigl[ (k + 2)^{2r} (k + 1)^{2\rho} \beta_{k + 1} \beta_{k} - (k + 1)^{2r} k^{2\rho} \beta_{k} \beta_{k - 1}\Bigr]\Biggr\} \| V(z^{k + 1})\|^{2} \label{eq:dE line 8}\\
        &+ \Biggl\{ k^{2(\rho - r)}\Bigl[ -\theta^{2} (k + 1)^{2r} k^{2r} \beta_{k} \beta_{k - 1} + \theta k^{2r} \eta_{k} \beta_{k - 1}\Bigr] \nonumber\\
        &\quad\quad+ \theta^{2} (k + 1)^{2r} k^{2\rho} \beta_{k} \beta_{k - 1} \Biggr\} \bigl\langle V(z^{k + 1}), V(z^{k + 1}) - V(z^{k})\bigr\rangle \label{eq:dE line 9}\\
        &+ \left[ k^{2(\rho - r)} \left(-\frac{\theta^{2}}{2}k^{4r} \beta_{k - 1}^{2}\right) - \frac{\theta^{2}}{2} (k + 1)^{2r} k^{2\rho} \beta_{k}\beta_{k - 1}\right] \bigl\| V(z^{k + 1}) - V(z^{k})\bigr\|^{2}. \label{eq:dE line 10}
    \end{align}
    We analyze each line separately. Line \eqref{eq:dE line 2}: evidently, we have 
    \[
        4\lambda \bigl( \lambda + rk^{r - 1} - \alpha\bigr) + 4\lambda\bigl( \alpha - (2\rho - r)k^{r - 1} - \lambda\bigr) + 4\lambda \Bigl[ (k + 1)^{2(\rho - r)} - k^{2(\rho - r)}\Bigr](k + 1)^{r} = 8\lambda (r - \rho) k^{2\rho - r - 1}.
    \]

    Line \eqref{eq:dE line 3}: recalling the definition for $\eta_{k}$, and the fact that $(k + 1)^{2(\rho - r)} - k^{2(\rho - r)} \leq 0$ (since $2(\rho - r)\leq 0$), the coefficient attached to $\langle z^{k + 1} - z_{*}, V(z^{k + 1})\rangle$ is less or equal than 
    \begin{align*}
        &\:2\lambda k^{2(\rho - r)}\eta_{k} + 2\lambda\theta \Bigl[ (k + 1)^{2\rho}\beta_{k} - k^{2\rho} \beta_{k - 1}\Bigr] \\
        = &\: 2\lambda k^{2(\rho - r)} \eta_{k} + 2\lambda\theta k^{2(\rho - r)} \Bigl[ (k + 1)^{2r}\beta_{k} - k^{2r}\beta_{k}\Bigr] \\
        &+ 2\lambda\theta \Bigl\{ \Bigl[ (k + 1)^{2\rho}\beta_{k} - k^{2\rho}\beta_{k - 1}\Bigr] - k^{2(\rho - r)}\Bigl[(k + 1)^{2r}\beta_{k} - k^{2r}\beta_{k}\Bigr]\Bigr\} \\
        = &\: 2\lambda k^{2(\rho - r)} \Bigl\{2k^{r} \Bigl[ \bigl(2r\theta k^{r - 1} - 1\bigr)\beta_{k} + \theta k^{r} \bigl( \beta_{k} - \beta_{k - 1}\bigr)\Bigr] - \theta \Bigl[(k + 1)^{2r}\beta_{k} - k^{2r}\beta_{k - 1}\Bigr]\Bigr\} \\
        &+ 2\lambda\theta k^{2(\rho - r)}\Bigl[ (k + 1)^{2r}\beta_{k} - k^{2r}\beta_{k - 1}\Bigr] + 2\lambda\theta \Bigl[ (k + 1)^{2(\rho - r)} - k^{2(\rho - r)}\Bigr] (k + 1)^{2r}\beta_{k} \\
        \leq &\: 4\lambda k^{2\rho - r} \Bigl[\bigl( 2r\theta k^{r - 1} - 1\bigr)\beta_{k} + \theta k^{r} \bigl(\beta_{k} - \beta_{k - 1}\bigr) \Bigr], 
    \end{align*}
    again using that $(k + 1)^{2(\rho - r)} - k^{2(\rho - r)}\leq 0$. 

    Line \eqref{eq:dE line 4}: we have 
    \[
        k^{2(\rho - r)} \bigl(-2\lambda\theta k^{2r}\bigr) + 2\lambda\theta k^{2\rho} = -2\lambda\theta k^{2\rho} + 2\lambda\theta k^{2\rho} = 0. 
    \]

    Line \eqref{eq:dE line 5}: the coefficient accompanying $\bigl\| z^{k + 1} - z^{k}\bigr\|^{2}$ is less than or equal to 
    \begin{align*}
        &\:k^{2(\rho - r)} \Bigl\{ 4\bigl( \lambda + r k^{r - 1} - \alpha\bigr)(k + 1)^{r} - 2\bigl( \lambda + r k^{r - 1} - \alpha\bigr)^{2} - 2\lambda \bigl(\alpha - (2\rho - r) k^{r - 1} - \lambda\bigr) \Bigr\} \\
        = &\: k^{2(\rho - r)}\bigl( \lambda + rk^{r - 1} - \alpha\bigr) \Bigl[ 4(k + 1)^{r} - 2\lambda - 2rk^{r - 1} + 2\alpha + 2\lambda\Bigr] \\
        = &\: k^{2(\rho - r)} \bigl( \lambda + rk^{r - 1} - \alpha\bigr) \Bigl[ 4(k + 1)^{r} + 2(\alpha - rk^{r - 1})\Bigr] \\
        \leq &\: 4k^{2\rho - r} \bigl( \lambda + rk^{r - 1} - \alpha\bigr), 
    \end{align*} 
    where we used the fact that eventually $\lambda + rk^{r - 1} - \alpha \leq 0$ (argue like at the beginning of the proof), together with $4(k + 1)^{r} \geq 4k^{r}$ and $\alpha - rk^{r - 1} > 0$ for large enough $k$. 

    Line \eqref{eq:dE line 7}: we notice that   
    \[
        k^{2(\rho - r)}\Bigl[-2\theta (k + 1)^{r} k^{2r} + 2\theta \bigl( \lambda + r k^{r - 1} - \alpha\bigr)k^{2r} \beta_{k - 1} - 2\lambda\theta k^{2r} \beta_{k - 1}\Bigr] - 2\lambda\theta k^{2\rho}\beta_{k - 1}
    \]
    is the sum of nonpositive terms for large enough $k$. Since $V$ is monotone, $\bigl\langle z^{k + 1} - z^{k}, V(z^{k + 1}) - V(z^{k})\bigr\rangle \geq 0$ for every $k$. 

    Line \eqref{eq:dE line 9}: it holds 
    \[
        k^{2(\rho - r)}\Bigl[ -\theta^{2}(k + 1)^{2r} k^{2r} \beta_{k} \beta_{k - 1} + \theta k^{2r} \eta_{k} \beta_{k - 1} \Bigr] + \theta^{2} (k + 1)^{2r} k^{2\rho} \beta_{k} \beta_{k - 1} = \theta k^{2\rho} \beta_{k - 1} \eta_{k}.
    \]
    This proves the lemma. 
\end{proof}

In the following lemma, we work with the coefficients attached to $\| z^{k + 1} - z_{*}\|^{2}$, $\bigl\langle z^{k + 1} - z_{*}, z^{k + 1} - z^{k}\bigr\rangle$, $\bigl\langle z^{k + 1} - z^{k}, V(z^{k + 1})\bigr\rangle$ and $\| V(z^{k + 1})\|^{2}$, which require further analysis. The proof is quite technical and lengthy; for the sake of readability, we have moved it to the Appendix.
\begin{lemma}\label{lem:coefs of <z(k+1)-z(k),V(z(k+1))> and |V(z(k+1))|}
    Consider the upper bound for $\mathcal{E}_{\lambda}^{k + 1} - \mathcal{E}_{\lambda}^{k}$ provided by Lemma \ref{lem:discrete derivative of energy function}. The following hold for large enough $k$:
    \begin{enumerate}[\rm (i)]
        \item The coefficient attached to $\| z^{k + 1} - z_{*}\|^{2}$ is bounded above by 
        \[
            4\lambda(\rho - r)k^{2(\rho - r) - 1} + O_{k}, 
        \]
        where $O_{k} = \mathcal{O}(k^{2(\rho - r) + r - 2})$ as $k\to +\infty$. 
        \item The coefficient attached to $\bigl\langle z^{k + 1} - z_{*}, z^{k + 1} - z^{k}\bigr\rangle$ can be written as $2 P_{k}$, with $P_{k} = \mathcal{O}(k^{2\rho - r - 2})$ as $k\to +\infty$.
        \item The coefficient attached to $\bigl\langle z^{k + 1} - z^{k}, V(z^{k + 1})\bigr\rangle$ can be written as 
        \[
            k^{2(\rho - r)} \Biggl\{ 2k^{2r} \Bigl\{ \Bigl[ 2\theta \bigl( \lambda + rk^{r - 1} - \alpha\bigr) + \bigl( \theta \alpha + \theta rk^{r - 1} - 2\bigr)\Bigr] \beta_{k} + \theta k^{r} \bigl(\beta_{k} - \beta_{k - 1}\bigr)\Bigr\} + 2(Q_{1, k} + Q_{2, k})\beta_{k}\Biggr\}, 
        \]
        where $|Q_{1, k}| = \mathcal{O}(k^{r})$ and $|Q_{2, k}| = \mathcal{O}(k^{3r - 1})$ as $k\to +\infty$. 
        \item The coefficient attached to $\| V(z^{k + 1})\|^{2}$ is bounded above by
        \[
            k^{2(\rho - r)} \Bigl[-2\delta \theta^{2} k^{3r} \beta_{k}^{2} + R_{k} \beta_{k}^{2}\Bigr], 
        \]
        where $R_{k} = o(k^{3r})$ as $k\to +\infty$.  
    \end{enumerate}
\end{lemma}

With these two intermediate lemmas, we can prove the first theorem. 

\begin{proof}[Proof of Theorem \ref{th:big O rates}]
    According to Lemma \ref{lem:discrete derivative of energy function} and Lemma \ref{lem:coefs of <z(k+1)-z(k),V(z(k+1))> and |V(z(k+1))|}, for large enough $k$ we may write 
    \begin{align}
        &\:\mathcal{E}_{\lambda}^{k + 1} - \mathcal{E}_{\lambda}^{k} \nonumber\\
        \leq &\: \Bigl[ \lambda(\rho - r)k^{2(\rho - r) - 1} + O_{k}\Bigr] \| z^{k + 1} - z_{*}\|^{2} \nonumber\\
        &\hspace{-0.28em}\left.\begin{aligned}
            &+ 3\lambda(\rho - r)k^{2(\rho - r) - 1} \| z^{k + 1} - z_{*}\|^{2} \\
            &+ 2 P_{k} \bigl\langle z^{k + 1} - z_{*}, z^{k + 1} - z^{k}\bigr\rangle + \frac{1}{3} k^{2\rho - r} \bigl( \lambda + rk^{r - 1} - \alpha\bigr) \bigl\| z^{k + 1} - z^{k}\bigr\|^{2} 
        \end{aligned}\right\}\label{eq:quad 1}\\
        &+ 4\lambda k^{2\rho - r} \Bigl[ \bigl(2r\theta k^{r - 1} - 1\bigr)\beta_{k} + \theta k^{r} \bigl( \beta_{k} - \beta_{k - 1}\bigr)\Bigr] \langle z^{k + 1} - z_{*}, V(z^{k + 1})\rangle \nonumber\\
        &+ \frac{2}{3} k^{2\rho - r} \bigl( \lambda + rk^{r - 1} - \alpha\bigr)\bigl\| z^{k + 1} - z^{k}\bigr\|^{2} + k^{2(\rho - r)}\left[-\frac{1}{3}\delta\theta^{2} k^{3r} + R_{k}\right]\beta_{k}^{2} \| V(z^{k + 1})\|^{2} \nonumber\\
        &\hspace{-0.28em}\left.\begin{aligned}
            &+ 3 k^{2(\rho - r)} k^{r} \bigl( \lambda + rk^{r - 1} - \alpha\bigr)\bigl\| z^{k + 1} - z^{k}\bigr\|^{2} \\
            &+ 2 k^{2(\rho - r)} \Biggl\{ k^{2r} \Bigl\{ \Bigl[ 2\theta \bigl( \lambda + rk^{r - 1} - \alpha\bigr) + \bigl( \theta \alpha + \theta rk^{r - 1} - 2\bigr)\Bigr]\beta_{k} + \theta k^{r} \bigl( \beta_{k} - \beta_{k - 1}\bigr)\Bigr\} \\
            &\quad\quad + (Q_{1, k} + Q_{2, k}) \beta_{k}\Biggr\} \bigl\langle z^{k + 1} - z^{k}, V(z^{k + 1})\bigr\rangle -\frac{4}{3} \delta\theta^{2} k^{2(\rho - r)} k^{3r} \beta_{k}^{2} \| V(z^{k + 1})\|^{2} 
        \end{aligned}\right\}\label{eq:quad 2}\\
        &\hspace{-0.28em}\left.\begin{aligned}
            &-\frac{1}{3} \delta\theta^{2} k^{2(\rho - r)} k^{3r} \beta_{k}^{2} \| V(z^{k + 1})\|^{2} + \theta k^{2(\rho - r)} k^{2r} \beta_{k - 1} \eta_{k} \bigl\langle V(z^{k + 1}), V(z^{k + 1}) - V(z^{k})\bigr\rangle \\
            &- k^{2(\rho - r)}\frac{\theta^{2}}{2} \Bigl[ k^{4r} \beta_{k - 1}^{2} + (k + 1)^{2r} k^{2r} \beta_{k}\beta_{k - 1}\Bigr] \bigl\| V(z^{k + 1}) - V(z^{k})\bigr\|^{2}.
        \end{aligned}\right\}\label{eq:quad 3}
    \end{align}
    We will use Lemma \ref{lem:quad} to show that
    \begin{itemize}
        \item \eqref{eq:quad 1} is nonpositive;
        \item There exist $\lambda$ such that $0 < \lambda < \alpha$ (when $r\in (0, 1)$) or $0 < \lambda < \alpha - 1$ (when $r = 1$) and such that \eqref{eq:quad 2} is nonpositive;
        \item \eqref{eq:quad 3} is nonpositive. 
    \end{itemize}
    To show that \eqref{eq:quad 1} is nonpositive, first assume that $r\in (0, 1)$. Take
    \[
        B := P_{k}, \quad A := 3\lambda(\rho - r) k^{2(\rho - r) - 1}, \quad C := \frac{1}{3}k^{2\rho - r}\bigl(\lambda + r k^{r - 1} - \alpha\bigr).
    \]
    According to Lemma \ref{lem:coefs of <z(k+1)-z(k),V(z(k+1))> and |V(z(k+1))|}, we have $P_{k}^{2} = \mathcal{O}(k^{4\rho - 2r - 4})$ as $k\to +\infty$. It follows that for large enough $k$ we have
    \begin{align*}
        B^{2} - AC = P_{k}^{2} - \lambda(r - \rho)\bigl(\alpha - r k^{r - 1} - \lambda\bigr) k^{4\rho - 3r - 1} \leq 0,
    \end{align*}
    since $4\rho - 2r - 4 < 4\rho - 3r - 1$. If $r = 1$, the only term of \eqref{eq:quad 1} which remains is $\frac{1}{3} k (\lambda + 1 - \alpha) \bigl\| z^{k + 1} - z^{k}\bigr\|^{2} \leq 0$.
    
    Before addressing \eqref{eq:quad 2}, we deal with \eqref{eq:quad 3}, since it is simpler. We factor out $k^{2(\rho - r)}$ and take
    \[
        B := \frac{\theta}{2} k^{2r} \beta_{k - 1}\eta_{k}, \quad A := -\frac{1}{3} \delta\theta^{2} k^{3r} \beta_{k}^{2}, \quad C := - \frac{\theta^{2}}{2} \Bigl[ k^{4r} \beta_{k - 1}^{2} + (k + 1)^{2r} k^{2r} \beta_{k}\beta_{k - 1}\Bigr].
    \]
    Therefore, 
    \begin{align*}
        B^{2} - AC &= \frac{\theta^{2}}{4} k^{4r} \beta_{k - 1}^{2} \eta_{k}^{2} - \frac{1}{6}\delta \theta^{4} k^{3r} \beta_{k}^{2} \Bigl[ k^{4r} \beta_{k - 1}^{2} + (k + 1)^{2r} k^{2r} \beta_{k}\beta_{k -  1}\Bigr] \\
        &\leq \frac{\theta^{2}}{4} M_{\beta}^{2} k^{4r}q_{k}^{2} \beta_{k - 1}^{4} - \frac{1}{3} \delta\theta^{4} k^{7r} \beta_{k - 1}^{4} \\
        &= \left[ \frac{\theta^{2}}{4} M_{\beta}^{2} k^{4r} q_{k - 1}^{2} - \frac{1}{3}\delta\theta^{4} k^{7r}\right] \beta_{k - 1}^{4},
    \end{align*}
    where we used the bound \eqref{eq:order of eta(k)} we had on $|\eta_{k}|$, as well as \eqref{eq:G3}. We had established that $q_{k} = \mathcal{O}(k^{r})$ as $k\to +\infty$, which gives $k^{4r} q_{k}^{2} = \mathcal{O}(k^{6r})$ as $k\to +\infty$. Thus, the term between square brackets becomes negative for large $k$, which gives $B^{2} - AC \leq 0$.

    Now, we focus on \eqref{eq:quad 2}. Again, factor out $k^{2(\rho - r)}$ and define
    \[
        \varepsilon_{k} := \alpha - rk^{r - 1} - \lambda, \quad c_{k} := \theta\alpha + \theta rk^{r - 1} - 2. 
    \]
    According to our assumptions, for large enough $k$ we have $\varepsilon_{k} > 0$ and $c_{k} \geq 0$. With this notation, we set 
    \[
        B := k^{2r} \Bigl[ \bigl(-2\theta \varepsilon_{k} + c_{k}\bigr)\beta_{k} + \theta k^{r} \bigl( \beta_{k} - \beta_{k - 1}\bigr)\Bigr] + (Q_{1, k} + Q_{2, k})\beta_{k}, \quad A := -3k^{r} \varepsilon_{k}, \quad C := -\frac{4}{3} \delta\theta^{2} k^{3r} \beta_{k}^{2}.  
    \]
    First of all, notice that 
    \begin{align*}
        B^{2} = &\: k^{4r} \Bigl[ \bigl(-2\theta \varepsilon_{k} + c_{k}\bigr)\beta_{k} + \theta k^{r} \bigl( \beta_{k} - \beta_{k - 1}\bigr)\Bigr]^{2} + (Q_{1, k} + Q_{2, k})^{2} \beta_{k}^{2}\\
        &+ 2 k^{2r}\Bigl[ \bigl(-2\theta \varepsilon_{k} + c_{k}\bigr)\beta_{k} + \theta k^{r} \bigl( \beta_{k} - \beta_{k - 1}\bigr)\Bigr] (Q_{1, k} + Q_{2, k}) \beta_{k} \\
        = &\: k^{4r} \Bigl[ \bigl(-2\theta \varepsilon_{k} + c_{k}\bigr)\beta_{k} + \theta k^{r} \bigl( \beta_{k} - \beta_{k - 1}\bigr)\Bigr]^{2} \\
        &+ \left\{2 k^{2r}\left[ \bigl(-2\theta \varepsilon_{k} + c_{k}\bigr) + \theta k^{r} \frac{\beta_{k} - \beta_{k - 1}}{\beta_{k}}\right](Q_{1, k} + Q_{2, k}) + (Q_{1, k} + Q_{2, k})^{2}\right\} \beta_{k}^{2} \\
        = &\: k^{4r} \Bigl[ \bigl(-2\theta \varepsilon_{k} + c_{k}\bigr)\beta_{k} + \theta k^{r} \bigl( \beta_{k} - \beta_{k - 1}\bigr)\Bigr]^{2} + S_{k} \beta_{k}^{2}, 
    \end{align*}
    where $S_{k}$ is the term between curly brackets. In Lemma \ref{lem:coefs of <z(k+1)-z(k),V(z(k+1))> and |V(z(k+1))|} we had shown that $Q_{1, k} = \mathcal{O}\left( k^{r}\right)$ and $Q_{2, k} = \mathcal{O}(k^{3r - 1})$ as $k\to +\infty$, which gives $S_{k} = \mathcal{O}(k^{\max\{ 3r, 5r - 1\}})$ as $k\to +\infty$. With these observations, we proceed: 
    \begin{align}
        &\:B^{2} - AC - S_{k}\beta_{k}^{2} \nonumber\\
        \leq &\: k^{4r} \Bigl[ \bigl(-2\theta \varepsilon_{k} + c_{k}\bigr)\beta_{k} + \theta k^{r} \bigl( \beta_{k} - \beta_{k - 1}\bigr)\Bigr]^{2} - 4\delta\theta^{2} k^{4r} \varepsilon_{k} \beta_{k}^{2} \nonumber \\
        = &\: k^{4r} \Bigl[ \bigl(-2 \theta \varepsilon_{k} + c_{k}\bigr)^{2} \beta_{k}^{2} + 2\theta \bigl(-2\theta \varepsilon_{k} + c_{k}\bigr) \beta_{k} k^{r} \bigl( \beta_{k} - \beta_{k - 1}\bigr) + \theta^{2} k^{2r} \bigl(\beta_{k} - \beta_{k - 1}\bigr)^{2}\Bigr] - 4\delta\theta^{2} k^{4r} \varepsilon_{k} \beta_{k}^{2} \nonumber\\
        \leq &\: k^{4r} \Bigl[ \bigl(4\theta^{2} \varepsilon_{k}^{2} - 4\theta c_{k} \varepsilon_{k} + c_{k}^{2} \bigr) \beta_{k}^{2} + 2c_{k} \bigl(1 - 2r\theta k^{r - 1} - \delta\theta\bigr) \beta_{k}^{2} + \bigl(1 - 2r\theta k^{r - 1} - \delta\theta\bigr)^{2} \beta_{k}^{2}\Bigr] \nonumber\\
        &- 4\delta\theta^{2} k^{4r} \varepsilon_{k} \beta_{k}^{2} \nonumber \\
        = &\: k^{4r} \beta_{k}^{2} \left\{ 4\theta^{2} \varepsilon_{k}^{2} - 4\theta (\delta\theta + c_{k}) \varepsilon_{k} + \Bigl[ c_{k} + (1 - 2r\theta k^{r - 1} - \delta\theta)\Bigr]^{2}\right\} \nonumber\\
        = &\: k^{4r} \beta_{k}^{2} \: p_{k}, \label{eq:B^2-AC is negative}
    \end{align}
    where $p_{k}$ is the term inside curly brackets. We will distinguish between the cases $r\in (0, 1)$ and $r = 1$. 

    \fbox{Case $r\in (0, 1)$:} define 
    \[
        c := \theta\alpha - 2, \quad \varepsilon := \alpha - \lambda, \quad p(t) := 4\theta^{2} t^{2} - 4\theta (c + \delta\theta)t + (c + 1 - \delta\theta)^{2}. 
    \]
    Notice that $c_{k} \to c$, $\varepsilon_{k} \to \varepsilon$ and $p_{k} \to p(\varepsilon)$ as $k\to +\infty$. The discriminant of the quadratic equation $p(t)=0$ reads
    \[
        \Delta = 16 \theta^{2} \Bigl[ (c + \delta\theta)^{2} - (c + 1 - \delta\theta)^{2}\Bigr].
    \]
    Notice that $\delta$ was chosen such that $\delta < \frac{1}{\theta}$. The growth assumption \eqref{eq:discrete growth intro} tells us it can also be chosen such that $\frac{1}{\theta} < 2\delta$, thus giving $\delta\theta > 1 - \delta\theta$ and $c + \delta\theta > c + 1 - \delta\theta > 0$. Therefore, $\Delta > 0$ and $p$ has two distinct roots: setting $\tilde{\Delta} := \frac{1}{8\theta^{2}} \sqrt{\Delta}$, these two roots are
    \[
        \underline{\varepsilon} = \frac{c + \delta\theta}{2\theta} - \tilde{\Delta}, \quad \overline{\varepsilon} = \frac{c + \delta\theta}{2\theta} + \tilde{\Delta}.
    \]
    The midpoint between the two roots is given by $\varepsilon_{r, 1} := \frac{c + \delta\theta}{2\theta}$, which fulfills
    \[
        0 < \varepsilon_{r, 1} = \frac{c + \delta\theta}{2\theta} \leq \frac{c + 1}{2\theta} = \frac{\theta\alpha - 1}{2\theta} = \frac{\alpha}{2} - \frac{1}{2\theta} < \alpha.
    \]
    Choose $\varepsilon_{r, 2} > 0$ such that $\varepsilon_{r, 1} - \tilde{\Delta} < \varepsilon_{r, 2} < \varepsilon_{r, 1}$. Set $\lambda_{r, i} := \alpha - \varepsilon_{r, i}$, $ i = 1, 2$. As we pointed out previously, we have $p_{k} \to p(\varepsilon_{r, i})$ as $k\to +\infty$. Therefore, for large enough $k$, we arrive at 
    \[
        p_{k} < \frac{p(\varepsilon_{r, i})}{2} = \frac{p(\alpha - \lambda_{r, i})}{2} < 0 \quad \text{for } i = 1, 2. 
    \]
    Going back to \eqref{eq:B^2-AC is negative}, this gives 
    \[
        B^{2} - AC \leq k^{4r} \beta_{k}^{2} \:\frac{p(\alpha - \lambda_{r, i})}{2} + S_{k}\beta_{k}^{2} < 0 \quad \text{for large enough } k \text{ and } i = 1, 2, 
    \]
    since $S_{k} = \mathcal{O}(k^{\max\{ 3r, 5r - 1\}})$ as $k\to +\infty$ and $4r > 5r - 1$ when $r < 1$.

    \fbox{Case $r = 1$:} here, we have
    \[
        \varepsilon_{k} \equiv \varepsilon := \alpha - 1 - \lambda, \quad \text{and} \quad c_{k} \equiv c := \theta\alpha + \theta - 2. 
    \]
    In this case, we define
    \[
        p(t) := 4\theta^{2} t^{2} - 4\theta(c + \delta\theta)t + \bigl(c + 1 - 2\theta - \delta\theta\bigr)^{2}.
    \]
    With this definition, $p_{k}$ is nothing else than $p(\varepsilon)$ for every $k$. The discriminant of the quadratic equation $p(t)=0$ the reads 
    \[
        \Delta = 16\theta^{2} \Bigl[ (c + \delta\theta)^{2} - (c + 1 - 2\theta - \delta\theta)^{2}\Bigr],
    \]
    and again, it is positive provided $\delta < \frac{1}{\theta} - 2$ is chosen such that $\frac{1}{\theta} - 2 < 2\delta$, which can be done thanks to the growth condition \eqref{eq:discrete growth intro}. Define $\tilde{\Delta}$ and the midpoint $\varepsilon_{1, 1}$ as before. Now, 
    \[
        0 < \varepsilon_{1, 1} = \frac{c + \delta\theta}{2\theta} \leq \frac{c + 1 - 2\theta}{2\theta} = \frac{\theta\alpha - \theta - 1}{2\theta} = \frac{\alpha - 1}{2} - \frac{1}{2\theta} < \alpha - 1.
    \]
    Choose $\varepsilon_{1, 2} > 0$ such that $\varepsilon_{1, 1} - \tilde{\Delta} < \varepsilon_{1, 2} < \varepsilon_{1, 1}$ and define $\lambda_{1, i} = \alpha - 1 - \varepsilon_{1, i}$, $i = 1, 2$. Going back to \eqref{eq:B^2-AC is negative}, this gives
    \[
        B^{2} - AC \leq k^{4} \beta_{k}^{2} \: p(\alpha - 1 - \lambda_{1, i}) + S_{k}\beta_{k}^{2} < 0 \quad \text{for large enough } k \text{ and } i = 1, 2, 
    \]
    since $p(\alpha - 1 - \lambda_{1, i}) = p(\varepsilon_{1, i}) < 0$ and $Q_{2, k}\equiv 0$ when $r = 1$, giving $S_{k} = \mathcal{O}(k^{3})$ as $k\to +\infty$. 

    To summarize, we have shown that \eqref{eq:quad 1}, \eqref{eq:quad 2}, and \eqref{eq:quad 3} are nonpositive. Going back to these lines, we have now 
    \begin{align*}
        \mathcal{E}_{\lambda_{r, i}}^{k + 1} - \mathcal{E}_{\lambda_{r, i}}^{k} \leq &\: \Bigl[ \lambda_{r, i}(\rho - r)k^{2(\rho - r) - 1} + O_{k}\Bigr] \bigl\| z^{k + 1} - z_{*}\bigr\|^{2} \\
        &+ 4\lambda_{r, i} k^{2\rho - r} \Bigl[ \bigl(2r\theta k^{r - 1} - 1\bigr)\beta_{k} + \theta k^{r} \bigl( \beta_{k} - \beta_{k - 1}\bigr)\Bigr] \langle z^{k + 1} - z_{*}, V(z^{k + 1})\rangle \\
        &+ \frac{2}{3} k^{2\rho - r} \bigl( \lambda_{r, i} + r k^{r - 1} - \alpha\bigr) \bigl\| z^{k + 1} - z^{k}\bigr\|^{2} + k^{2(\rho - r)} \left[ -\frac{1}{3} \delta\theta^{2} k^{3r} + R_{k}\right] \beta_{k}^{2} \| V(z^{k + 1})\|^{2} 
    \end{align*}
    If $r \in (0, 1)$, then $2(\rho - r) - 1 > 2(\rho - r) + r - 2$. Recall that $O_{k} = k^{2(\rho - r) + r - 2}$ as $k\to +\infty$, thus the term with $\| z^{k + 1} - z_{*}\|^{2}$ eventually becomes nonpositive. If $r = 1$, the $\| z^{k + 1} - z_{*}\|^{2}$ term vanishes. In both cases the $\| V(z^{k + 1})\|^{2}$ eventually becomes nonpositive since we had $R_{k} = o(k^{3r})$ as $k\to +\infty$. The $\langle z^{k + 1} - z_{*}, V(z^{k + 1})\rangle$ term is nonpositive according to \eqref{eq:G1}, and the $\bigl\| z^{k + 1} - z^{k}\bigr\|^{2}$ term is evidently nonpositive. All in all, we have 
    \[
        \mathcal{E}_{\lambda_{r, i}}^{k + 1} - \mathcal{E}_{\lambda_{r, i}}^{k} \leq 0 
    \]
    for $k$ large enough, say $k\geq K \geq k_{0}$. This means that $\bigl(\mathcal{E}_{\lambda_{r, i}}^{k}\bigr)_{k\geq K}$ is monotonically decreasing, and thus 
    \begin{equation} \label{eq: disc energy function bounded} 
        0 \leq \mathcal{E}_{\lambda_{r, i}}^{k} \leq \mathcal{E}_{\lambda_{r, i}}^{K} \text{ for }k\geq K\text{, }r\in (0, 1]\text{ and }i=1, 2.
    \end{equation}

    In order to conclude, we separate again the cases $r\in (0, 1)$ and $r = 1$. 

    \fbox{Case $r\in (0, 1)$:} according to the definition of the discrete energy function \eqref{eq:energy function line 1}-\eqref{eq:energy function line 4}, \eqref{eq: disc energy function bounded}, the fact that $(\beta_{k})_{k\geq 1}$ is nondecreasing and \eqref{eq:G3}, for $k\geq K$ we have 
    \begin{equation}\label{eq:disc rate for ||V(z(k))||}
        \frac{\theta^{2}}{2} k^{2(\rho + r)} \frac{\beta_{k - 1}^{2}}{M_{\beta}} \| V(z^{k})\|^{2} \leq \frac{\theta^{2}}{2} k^{2(\rho + r)} \frac{\beta_{k}^{2}}{M_{\beta}}\| V(z^{k})\|^{2} \leq \frac{\theta^{2}}{2} (k + 1)^{2r} k^{2\rho} \beta_{k} \beta_{k - 1} \| V(z^{k})\|^{2} \leq \mathcal{E}_{\lambda_{r, 1}}^{K}
    \end{equation}
    and 
    \begin{equation}\label{eq:disc rate for <z(k) - z*, V(z(k))>}
        2\lambda_{r, 1} \theta k^{2\rho} \frac{\beta_{k}}{M_{\beta}} \langle z^{k + 1} - z_{*}, V(z^{k})\rangle \leq 2\lambda_{r, 1} \theta k^{2\rho} \beta_{k - 1} \langle z^{k} - z_{*}, V(z^{k})\rangle \leq \mathcal{E}_{\lambda_{r, 1}}^{K}, 
    \end{equation}
    from which we deduce 
    \begin{equation}\label{eq: disc rates for <z(k) - z*, V(z(k))>, ||V(z(k))||}
        \langle z^{k} - z_{*}, V(z^{k})\rangle \leq \frac{M_{\beta}\mathcal{E}_{\lambda_{r, 1}}^{K}}{2\lambda_{r, 1} \theta}\cdot\frac{1}{k^{2\rho}\beta_{k}} \quad \text{and} \quad \| V(z^{k})\| \leq \frac{\sqrt{2 M_{\beta}\mathcal{E}_{\lambda_{r, 1}}^{K}}}{\theta} \cdot \frac{1}{k^{\rho + r} \beta_{k}}.
    \end{equation}
    Furthemore, say that for $k\geq K$ we have $(2\rho - r)k^{r - 1} < \xi < \alpha - \lambda_{r, 1}$. Going back to \eqref{eq:energy function line 2}, this means that for $k\geq K$ it holds 
    \[
        2\lambda_{r, 1} k^{2(\rho - r)} \bigl(\alpha - \xi - \lambda_{r, 1}\bigr) \| z^{k} - z_{*}\|^{2} \leq 2\lambda_{r, 1} k^{2(\rho - r)} \bigl(\alpha - (2\rho - r)k^{r - 1}) - \lambda_{r, 1}\bigr) \| z^{k} - z_{*}\|^{2} \leq \mathcal{E}_{\lambda_{r, 1}}^{K}
    \]
    and thus 
    \[
        2\lambda_{r, 1}t^{\rho - r}\| z^{k} - z_{*}\| \leq \sqrt{\frac{2 \mathcal{E}_{\lambda_{r, 1}}^{K}}{ \bigl(\alpha - \xi - \lambda_{r, 1}\bigr)}}.
    \]
    Combining this together with \eqref{eq:energy function line 1}, \eqref{eq: disc energy function bounded}, \eqref{eq:disc rate for ||V(z(k))||} and \eqref{eq: disc rates for <z(k) - z*, V(z(k))>, ||V(z(k))||} yields
    \begin{align*}
        2k^{\rho} \bigl\| z^{k} - z^{k - 1}\bigr\| \leq &\: 2\lambda_{r, 1} k^{\rho - r} \| z^{k} - z_{*}\| + \theta k^{\rho + r} \beta_{k - 1} \| V(z^{k})\| \\
        &+ \Bigl\| 2\lambda_{r, 1} k^{\rho - r}(z^{} - z_{*}) + 2k^{\rho} \bigl(z^{k} - z^{k - 1}\bigr)+ \theta k^{\rho + r} \beta_{k - 1} V(z^{k})\Bigr]\Bigr\| \\
        \leq &\:\sqrt{\frac{2 \mathcal{E}_{\lambda_{r, 1}}^{K}}{ \bigl(\alpha - \xi - \lambda_{r, 1}\bigr)}} + \sqrt{2 M_{\beta} \mathcal{E}_{\lambda_{r, 1}}^{K}} + \sqrt{2 \mathcal{E}_{\lambda_{r, 1}}^{K}}.
    \end{align*}
    All in all, we have obtained that, as $k\to +\infty$,
    \begin{equation}\label{eq:disc to change rho by r and O by o}
        \| V(z^{k})\| = \mathcal{O}\left(\frac{1}{k^{\rho + r} \beta_{k}}\right), \quad \langle z^{k} - z_{*}, V(z^{k})\rangle = \mathcal{O}\left(\frac{1}{k^{2\rho}\beta_{k}}\right), \quad \bigl\| z^{k} - z^{k - 1}\bigr\| = \mathcal{O}\left(\frac{1}{k^{\rho}}\right). 
    \end{equation}

    \fbox{Case $r = 1$:} going back to \eqref{eq:energy function line 2} we have, for $k\geq K$, 
    \[
        \| z^{k} - z_{*}\| \leq \sqrt{\frac{\mathcal{E}_{\lambda_{1, 1}}^{K}}{2\lambda_{1, 1} \bigl(\alpha - 1 - \lambda_{1, 1}\bigr)}},
    \]
    which gives the boundedness of $\bigl( z^{k}\bigr)_{k\geq 2}$. Using the boundedness of $\bigl( \mathcal{E}_{\lambda_{r, 1}}^{k}\bigr)_{k\geq 1}$ and arguing exactly as in the case $r \in (0, 1)$ we obtain 
    \begin{equation}\label{eq:disc to change O by o}
        \| V(z^{k})\| = \mathcal{O}\left(\frac{1}{k^{2} \beta_{k}}\right), \quad \langle z^{k} - z_{*}, V(z^{k})\rangle = \mathcal{O}\left( \frac{1}{k^{2}\beta_{k}}\right), \quad \bigl\| z^{k} - z^{k - 1}\bigr\| = \mathcal{O}\left(\frac{1}{k^{r}}\right)
    \end{equation}
    as $k\to +\infty$. 
\end{proof}

Similar to the continuous case, in the following theorem we will assume that $\bigl( z^{k}\bigr)_{k\geq 2}$ is bounded when $r\in (0, 1)$. This will allow us to set $\rho = r$ and improve the rate of convergence results from $\mathcal{O}$ to $o$ in \eqref{eq:disc to change rho by r and O by o}, as well as to show weak convergence of the iterates towards a zero of $V$. Additionally, when $r = 1$, we will can also lift $\mathcal{O}$ to $o$ and again show weak convergence of the iterates towards a zero of $V$.  

\begin{theorem}
    Under the same hypotheses of Theorem \ref{th:big O rates}, consider the following convergence rates as $k\to +\infty$:
    \[
        \| V(z^{k})\| = o\left(\frac{1}{k^{2r} \beta_{k}}\right), \quad \langle z^{k} - z_{*}, V(z^{k})\rangle = o\left( \frac{1}{k^{2r} \beta_{k}}\right) \quad \text{and} \quad \| z^{k} - z^{k - 1}\| = o\left( \frac{1}{k^{r}}\right). 
    \]
    If $r = 1$ or $\Biggl( r\in (0, 1)$ and $(z^{k})_{k\geq 2}$ is bounded $\Biggr)$, then the above rates hold and $(z^{k})_{k\geq 2}$ converges weakly to a zero of $V$ as $k\to +\infty$.
\end{theorem}
\begin{proof}
    Going back to Lemma \ref{lem:discrete derivative of energy function} and \ref{lem:coefs of <z(k+1)-z(k),V(z(k+1))> and |V(z(k+1))|}, choosing $\rho \in (0, r)$ when $r \in (0, 1)$ is done just so we can guarantee that term attached to $\| z^{k + 1} - z_{*}\|^{2}$ eventually becomes nonpositive. We can plug $\rho = r$ in our calculations; the term with $\bigl\langle z^{k + 1} - z_{*}, z^{k + 1} - z^{k}\bigr\rangle$ will vanish, and the nonnegativity of \eqref{eq:quad 2} and \eqref{eq:quad 3} will not be affected since it doesn't depend on $\rho$. We now have 
    \begin{align}
        \mathcal{E}_{\lambda_{r, i}}^{k + 1} - \mathcal{E}_{\lambda_{r, i}}^{k}
        \leq &\: 2\lambda_{r, i} \:r \Big[ k^{r - 1} - (k + 1)^{r - 1}\Bigr] \| z^{k + 1} - z_{*}\|^{2} \nonumber\\
        &+ 4\lambda_{r, i} k^{r} \Bigl[ \bigl(2r\theta k^{r - 1} - 1\bigr)\beta_{k} + \theta k^{r} \bigl( \beta_{k} - \beta_{k - 1}\bigr)\Bigr] \langle z^{k + 1} - z_{*}, V(z^{k + 1})\rangle \label{eq:final dE line 1}\\
        &+ k^{r} \bigl( \lambda_{r, i} + rk^{r - 1} - \alpha\bigr)\bigl\| z^{k + 1} - z^{k}\bigr\|^{2} \label{eq:final dE line 2}\\
        &+ \left[-\frac{1}{3}\delta\theta^{2} k^{3r} + R_{k}\right]\beta_{k}^{2} \| V(z^{k + 1})\|^{2} \label{eq:final dE line 3}\\
        \leq &\: 2\lambda_{r, i}\: r \Bigl[ k^{r - 1} - (k + 1)^{r - 1}\Bigr] \| z^{k + 1} - z_{*}\|^{2}, \label{eq:final dE line 4}
    \end{align}
    since for large enough $k$ lines \eqref{eq:final dE line 1} - \eqref{eq:final dE line 3} become nonpositive (recall that $R_{k} = o(k^{3r})$ as $k\to +\infty$, so $R_{k} - \frac{1}{3}\delta\theta^{2} k^{3r}$ eventually becomes negative). If $r\in (0, 1)$, recall we assume that $(\| z^{k} - z_{*}\|)_{k\geq 0}$ is bounded, say $ \| z^{k} - z_{*}\|^{2} \leq M$ for all $k$. According to \eqref{eq:final dE line 4}, for large $k$ (say $k\geq K \geq 2$ suffices), we have
    \begin{equation}\label{eq:almost decreasing property for E}
        \mathcal{E}_{\lambda_{r, i}}^{k + 1} - \mathcal{E}_{\lambda_{r, i}}^{k} \leq 2\lambda_{r, i}\: r \Bigl[ k^{r - 1} - (k + 1)^{r - 1}\Bigr] \| z^{k + 1} - z_{*}\|^{2} \leq 2\lambda_{r, i} \: rM \Bigl[ k^{r - 1} - (k + 1)^{r - 1}\Bigr].
    \end{equation}
    Rearranging the previous terms gives 
    \[
        \mathcal{E}_{\lambda_{r, i}}^{k + 1} + 2\lambda_{r, i} rM (k + 1)^{r - 1} \leq \mathcal{E}_{\lambda_{r, i}}^{k} + 2\lambda_{r, i} rM k^{r - 1} \quad \forall k\geq K. 
    \]
    Therefore, the sequence $\bigl(\mathcal{E}_{\lambda_{r, i}}^{k} + 2\lambda_{r, i} rM k^{r - 1}\bigr)_{k\geq K}$ is nonnegative and nonincreasing, meaning it has a limit as $k\to +\infty$. Since $2\lambda_{r, i} r (1 - r) M k^{r - 1}$ also has a limit as $k\to +\infty$, we come to the following: for $\rho = r$, we have 
    \begin{align}
        0 \leq \mathcal{E}_{\lambda_{r, i}}^{k} \leq \mathcal{E}_{\lambda_{r, i}}^{K} + 2\lambda_{r, i} r (1 - r) M K^{r - 1}\text{ for }k\geq K\text{, }r\in (0, 1]\text{ and }i = 1, 2 ; \\
        \lim_{k\to +\infty} \mathcal{E}_{\lambda_{r, i}}^{k}\text{ exists for }r\in (0, 1]\text{ and }i = 1, 2. 
    \end{align}
    By going back to \eqref{eq:energy function line 1}-\eqref{eq:energy function line 4}, setting $\rho = r$ and reasoning exactly as we did before, we arrive at 
    \begin{equation}
        \| V(z^{k})\| = \mathcal{O}\left(\frac{1}{k^{2r} \beta_{k}}\right), \quad \langle z^{k} - z_{*}, V(z^{k})\rangle = \mathcal{O}\left( \frac{1}{k^{2r}\beta_{k}}\right), \quad \bigl\| z^{k} - z^{k - 1}\bigr\| = \mathcal{O}\left( \frac{1}{k^{r}}\right)
    \end{equation}
    as $k\to +\infty$, where these rates now hold for all $r\in (0, 1]$.

    For $r\in (0, 1]$, we proceed to show three summability results that we will need later. Going back to \eqref{eq:final dE line 4}, choosing $i = 1$ and using \eqref{eq:G1}, for $k\geq K$ we have
    \begin{align}
        &\:4\lambda_{r, 1} \delta\theta k^{r} \beta_{k} \langle z^{k + 1} - z_{*}, V(z^{k + 1})\rangle + \bigl(\alpha - rk^{r - 1} - \lambda_{r, 1}\bigr) k^{r} \bigl\| z^{k + 1} - z^{k}\bigr\|^{2} + \frac{1}{6} \delta\theta^{2} k^{3r} \beta_{k}^{2} \| V(z^{k + 1})\|^{2} \nonumber\\
        \leq &\: \mathcal{E}_{\lambda_{r, 1}}^{k} - \mathcal{E}_{\lambda_{r, 1}}^{k + 1} + 2\lambda_{r, 1} \Bigl[ k^{r - 1} - (k + 1)^{r - 1}\Bigr] \| z^{k + 1} - z_{*}\|^{2}, \label{eq:three integrals}
    \end{align}
    where we used the fact that $-\frac{1}{6}\delta\theta^{2} k^{3r} + R_{k}$ eventually becomes negative. Notice that the last line is summable. We have
    \begin{enumerate}[(I)]
        \item Using \eqref{eq:A3} and the Cauchy-Schwarz inequality, and using that $\| V(z^{k})\| = \mathcal{O}\left( \frac{1}{k^{2r} \beta_{k}}\right)$ as $k\to +\infty$, we obtain  
        \[
            \Bigl[(k + 1)^{r} - k^{r}\Bigr] \beta_{k} \langle z^{k + 1} - z_{*}, V(z^{k + 1})\rangle \leq rk^{r - 1} \beta_{k} \| z^{k + 1} - z_{*}\| \| V(z^{k + 1})\| = \mathcal{O}\left( \frac{1}{k^{1 + r}}\right)
        \]
        as $k\to +\infty$. This implies that the left-hand side is summable. Therefore, 
        \begin{align}
            &\:\sum_{k = 1}^{+\infty} (k + 1)^{r} \beta_{k} \langle z^{k + 1} - z_{*}, V(z^{k + 1})\rangle \nonumber\\
            \leq &\: \underbrace{\sum_{k = 1}^{+\infty} k^{r} \beta_{k} \langle z^{k + 1} - z_{*}, V(z^{k + 1})\rangle}_{<+\infty\text{ by \eqref{eq:three integrals}}} + \underbrace{\sum_{k = 1}^{+\infty} \Bigl[(k + 1)^{r} - k^{r}\Bigr] \beta_{k} \langle z^{k + 1} - z_{*}, V(z^{k + 1})\rangle}_{<+\infty} < +\infty. \label{eq:integral 1}
        \end{align}
        \item Arguing like in the previous item and using $\bigl\|z^{k} - z^{k - 1}\bigr\| = \mathcal{O}\left(\frac{1}{k^{r}}\right)$ as $k\to +\infty$, \eqref{eq:three integrals} implies
        \begin{equation}
            \sum_{k = }^{+\infty}(k + 1)^{r} \bigl\| z^{k + 1} - z^{k}\bigr\|^{2} < +\infty. \label{eq:integral 2}
        \end{equation}
        \item Using \eqref{eq:A4} and reasoning as before, \eqref{eq:three integrals} entails
        \begin{equation}
            \sum_{k = 1}^{+\infty} (k + 1)^{3r} \beta_{k}^{2} \| V(z^{k + 1})\|^{2} < +\infty. \label{eq:integral 3}
        \end{equation}
    \end{enumerate}
    Consider the $\lambda_{r, i}$ defined in the proof of the previous theorem. Recalling the definition \eqref{eq:energy function line 1} - \eqref{eq:energy function line 4} of our energy function, now with $\rho = r$, we can write 
    \begin{align*}
        &\:\mathcal{E}_{\lambda_{r, 2}}^{k} - \mathcal{E}_{\lambda_{r, 1}}^{k} \\
        = &\: \frac{1}{2} \Bigl\| 2\lambda_{r, 2}(z^{k} - z_{*}) + 2k^{r} \bigl( z^{k} - z^{k - 1}\bigr) + \theta k^{2r} \beta_{k - 1} V(z^{k})\Bigr\|^{2} \\
        &- \frac{1}{2} \Bigl\| 2\lambda_{r, 1}(z^{k} - z_{*}) + 2k^{r} \bigl( z^{k} - z^{k - 1}\bigr) + \theta k^{2r} \beta_{k - 1} V(z^{k})\Bigr\|^{2} \\
        &+ \Bigl[ 2\lambda_{r, 2} \bigl(\alpha - rk^{r - 1} - \lambda_{r, 2}\bigr) - 2\lambda_{r, 1}\bigl(\alpha - rk^{r - 1} - \lambda_{r, 1}\bigr)\Bigr] \| z^{k} - z_{*}\|^{2} \\
        &+ 2 \bigl( \lambda_{r, 2} - \lambda_{r, 1}\bigr) \theta k^{2r} \beta_{k - 1} \langle z^{k} - z_{*}, V(z^{k})\rangle \\
        = &\: \Bigl\langle 2\lambda_{r, 2} (z^{k} - z_{*}) + 2k^{r} \bigl(z^{k} - z^{k - 1}\bigr) + \theta k^{2r} \beta_{k - 1} V(z^{k}), \:\: 2 \bigl( \lambda_{r, 2} - \lambda_{r, 1}\bigr) (z^{k} - z_{*})\Bigr\rangle \\
        &- \frac{1}{2} \bigl\| 2\bigl( \lambda_{r, 2} - \lambda_{r, 1}\bigr) (z^{k} - z_{*})\bigr\|^{2} + 2\bigl( \lambda_{r, 2} - \lambda_{r, 1}\bigr) \bigl(\alpha - rk^{r - 1}\bigr) \| z^{k} - z_{*}\|^{2}\\
        &- 2\bigl(\lambda_{r, 2}^{2} - \lambda_{r, 1}^{2}\bigr) \| z^{k} - z_{*}\|^{2} + \bigl\langle 2\bigl( \lambda_{r, 2} - \lambda_{r, 1}\bigr) (z^{k} - z_{*}), \:\: \theta k^{2r} \beta_{k - 1} V(z^{k})\bigr\rangle \\
        = &\: \Bigl\langle 2\bigl( \lambda_{r, 2} - \lambda_{r, 1}\bigr) (z^{k} - z_{*}), \:\: 2k^{r} \bigl( z^{k} - z^{k - 1}\bigr) + 2\theta k^{2r} \beta_{k - 1} V(z^{k})\Bigr\rangle \\
        &+ 2\bigl(\lambda_{r, 2} - \lambda_{r, 1}\bigr) \bigl( \alpha - rk^{r - 1}\bigr) \| z^{k} - z_{*}\|^{2} \\
        &+ \underbrace{\Bigl[ 4\lambda_{r, 2} \bigl( \lambda_{r, 2} - \lambda_{r, 1}\bigr) - 2\bigl( \lambda_{r, 2} - \lambda_{r, 1}\bigr)^{2} - 2\bigl( \lambda_{r, 2}^{2} - \lambda_{r, 1}^{2}\bigr)\Bigr]}_{= 0} \| z^{k} - z_{*}\|^{2} \\
        = &\: 4\bigl( \lambda_{r, 2} - \lambda_{r, 1}\bigr) \left[ \frac{\alpha - rk^{r - 1}}{2} \:\| z^{k} - z_{*}\|^{2} + k^{r} \Bigl\langle z^{k} - z_{*}, \:\: z^{k} - z^{k - 1} + \theta k^{r} \beta_{k - 1} V(z^{k})\Bigr\rangle\right].
    \end{align*}
    Call $p_{k}$ the term between square brackets. We showed that $\lim_{k\to +\infty} \mathcal{E}_{\lambda_{r, i}}^{k}$ exists for $i = 1, 2$, and $\lambda_{r, 2} - \lambda_{r, 1} \neq 0$, thus
    \begin{equation}\label{eq:p(k) has a limit}
        \lim_{k\to +\infty} p_{k} \quad \text{exists}.
    \end{equation}
    With this at hand, we can rewrite $\mathcal{E}_{\lambda_{r, 1}}^{k}$ as 
    \begin{align}
      \mathcal{E}_{\lambda_{r, 1}}^{k} 
        = &\: \frac{1}{2} \Bigl\| 2\lambda_{r, 1} (z^{k} - z_{*}) + 2k^{r} \bigl( z^{k} - z^{k - 1}\bigr) + \theta k^{2r} \beta_{k - 1} V(z^{k})\Bigr\|^{2} + 2\lambda_{r, 1}\bigl( \alpha - rk^{r - 1} - \lambda_{r, 1}\bigr) \| z^{k} - z_{*}\|^{2} \nonumber\\
        &+ 2\lambda_{r, 1} \theta^{2r} \beta_{k - 1} \langle z^{k} - z_{*}, V(z^{k})\rangle + \frac{\theta^{2}}{2} (k + 1)^{2r} k^{2r} \beta_{k} \beta_{k - 1} \| V(z^{k})\|^{2} \nonumber\\
        = &\: \frac{1}{2} \Biggl[ 4\lambda_{r, 1}^{2} \| z^{k} - z_{*}\|^{2} + 2\lambda_{r, 1} k^{r} \Bigl\langle z^{k} - z_{*}, \:\: 2\bigl( z^{k} - z^{k - 1}\big) + \theta k^{r} \beta_{k - 1} V(z^{k})\Bigr\rangle \nonumber\\
        &\quad\quad + k^{2r} \Bigl\| 2\bigl(z^{k} - z^{k - 1}\bigr) + \theta k^{r} \beta_{k - 1} V(z^{k})\Bigr\|^{2}\Biggr] \nonumber\\
        &+ \Bigl[ 2\lambda_{r, 1} \bigl(\alpha - rk^{r - 1}\bigr) - 2\lambda_{r, 1}^{2}\Bigr] \| z^{k} - z_{*}\|^{2} + 2\lambda_{r, 1} \theta k^{2r} \langle z^{k} - z_{*}, V(z^{k})\rangle \nonumber\\
        &+ \frac{\theta^{2}}{2} (k + 1)^{2r} k^{2r} \beta_{k}\beta_{k - 1} \| V(z^{k})\|^{2} \nonumber\\
        = &\: 4\lambda_{r, 1} p_{k} + \frac{k^{2r}}{2} \Bigl\| 2\bigl(z^{k} - z^{k - 1}\bigr) + \theta k^{r} \beta_{k - 1} V(z^{k})\Bigr\|^{2} + \frac{\theta^{2}}{2} k^{4r} \beta_{k - 1}^{2} \| V(z^{k})\|^{2} \nonumber\\
        &+ \frac{\theta^{2}}{2}\left[ (k + 1)^{2r} k^{2r} \beta_{k}\beta_{k - 1} - k^{4r} \beta_{k - 1}^{2}\right] \| V(z^{k})\|^{2} \nonumber\\
        = &\:  4\lambda_{r, 1} p_{k} + k^{2r} \Bigl\| z^{k} - z^{k - 1} + \theta k^{r} \beta_{k - 1} V(z^{k})\Bigr\|^{2} + k^{2r} \bigl\| z^{k} - z^{k - 1}\bigr\|^{2}  \nonumber\\
        &+ \frac{\theta^{2}}{2}\left[ (k + 1)^{2r} k^{2r} \beta_{k}\beta_{k - 1} - k^{4r} \beta_{k - 1}^{2}\right] \| V(z^{k})\|^{2}. \label{eq:rewriting E}
    \end{align}
    Let us analyze \eqref{eq:rewriting E}. We have 
    \begin{align*}
        &\:(k + 1)^{2r} k^{2r} \beta_{k} \beta_{k - 1} - k^{4r} \beta_{k - 1}^{2} \\
        = &\: k^{2r} \beta_{k - 1} \Bigl[ (k + 1)^{2r} \beta_{k} - k^{2r} \beta_{k - 1}\Bigr] \\
        = &\: k^{2r} \beta_{k - 1} \Bigl\{ \Bigl[ (k + 1)^{2r} - k^{2r}\Bigr]\beta_{k} + k^{2r} \bigl( \beta_{k} - \beta_{k - 1}\bigr)\Bigr\} \\
        \leq &\: k^{2r} \beta_{k - 1} \Bigl[ \bigl(2r k^{2r - 1} + rk^{2r - 2}\bigr) \beta_{k} + k^{r} \bigl(1 - 2r\theta k^{r - 1} - \delta\theta\bigr) \beta_{k}\Bigr] \\
        \leq &\: \bigl(2rk^{4r - 1} + rk^{4r - 2}\bigr)M_{\beta} \beta_{k - 1}^{2} + (1 - \delta\theta) k^{3r} M_{\beta} \beta_{k - 1}^{2}, 
    \end{align*}
    where we used \eqref{eq:A1}, \eqref{eq:G2} and \eqref{eq:G3}. According to Theorem \ref{th:big O rates} and the previous inequality, we obtain that line \eqref{eq:rewriting E} converges to zero as $k\to +\infty$. Since we know that $\lim_{k\to +\infty} p_{k}$ exists, we obtain the existence of $\lim_{k\to +\infty} h_{k}$, where 
    \[
        h_{k} := k^{2r} \Bigl\| z^{k} - z^{k - 1} + \theta k^{r} \beta_{k - 1} V(z^{k})\Bigr\|^{2} + k^{2r} \bigl\| z^{k} - z^{k - 1}\bigr\|^{2}.
    \]
    Furthermore, we have 
    \[
        \sum_{k = 1}^{+\infty} \frac{h_{k}}{k^{r}} \leq 3\sum_{k = 1}^{+\infty} k^{r} \bigl\| z^{k} - z^{k - 1} \bigr\|^{2} + 2\theta \sum_{k = 1}^{+\infty} k^{3r} \beta_{k - 1}^{2}\| V(z^{k})\|^{2} < +\infty,
    \]
    according to \eqref{eq:integral 2} and \eqref{eq:dE line 3}. Since $\sum_{k = 1}^{+\infty} \frac{1}{k^{r}} = +\infty$, it must be the case that 
    \[
        \lim_{k\to +\infty} h_{k} = 0.
    \]
    In particular, this implies 
    \[
        \lim_{k\to +\infty} k^{2r} \bigl\| z^{k} - z^{k - 1}\bigr\|^{2} = 0, \text{ equivalently } \bigl\| z^{k} - z^{k - 1}\bigr\| = o\left(\frac{1}{k^{r}}\right)\text{ as }k\to +\infty.
    \]
    Using again that $\lim_{k\to +\infty} h_{k} = 0$ together with the previous statement, \eqref{eq:G3} and the triangle inequality, gives
    \[
        \lim_{k\to +\infty} k^{4r} \beta_{k - 1}^{2} \| V(z^{k})\|^{2} = 0, \text{ equivalently } \| V(z^{k})\| = o\left(\frac{1}{k^{2r} \beta_{k}}\right) \text{ as }k\to +\infty. 
    \]
    This last statement, together with the boundedness of the trajectories and the Cauchy-Schwarz inequality, gives
    \[
        \langle z^{k} - z_{*}, V(z^{k})\rangle = o\left( \frac{1}{k^{2r} \beta_{k}}\right) \text{ as }k\to +\infty.  
    \]
    This produces the stated $o$ rates. 

    We now deal with the weak convergence of the iterates. Just like in the continuous case, we use Opial's lemma (see Lemma \ref{lem:Opial}). Define, for $k\geq K$, 
    \[
        q_{k} := \frac{1}{2} \| z^{k} - z_{*}\|^{2} + \theta \sum_{j = 1}^{k} j^{r} \beta_{j - 1} \langle z^{j} - z_{*}, V(z^{j})\rangle. 
    \]
    It follows that for $k\geq K$ we have 
    \begin{align*}
        q_{k} - q_{k - 1} &= \frac{1}{2}\Bigl[ \| z^{k} - z_{*}\|^{2} - \| z^{k - 1} - z_{*}\|^{2}\Bigr] + \theta k^{r} \beta_{k - 1}\langle z^{k} - z_{*}, V(z^{k})\rangle \\
        &= \bigl\langle z^{k} - z_{*}, z^{k} - z^{k - 1} + \theta k^{r} \beta_{k - 1} V(z^{k})\bigr\rangle - \frac{1}{2} \bigl\| z^{k} - z^{k - 1}\bigr\|^{2}.
    \end{align*}
    We distinguish between the cases $r\in (0, 1)$ and $r = 1$. 
    
    \fbox{Case $r \in (0, 1)$:} we have
    \begin{align*}
        \alpha q_{k} + k^{r} \bigl( q_{k} - q_{k - 1}\bigr) = &\: \frac{\alpha}{2} \| z^{k} - z_{*}\|^{2} + \bigl\langle z^{k} - z_{*}, z^{k} - z^{k - 1} + \theta k^{r} \beta_{k - 1} V(z^{k})\bigr\rangle \\
        &+ \theta\alpha \sum_{j = 1}^{k} j^{r} \beta_{j - 1} \langle z^{j} - z_{*}, V(z^{j})\rangle - \frac{1}{2} \bigl\| z^{k} - z^{k - 1}\bigr\|^{2} \\
        = &\: \frac{r}{2} k^{r - 1} \| z^{k} - z_{*}\|^{2} + p_{k} + \theta\alpha \sum_{j = 1}^{k} j^{r} \beta_{j - 1} \langle z^{j} - z_{*}, V(z^{j})\rangle - \frac{1}{2} \bigl\| z^{k} - z^{k - 1}\bigr\|^{2}.
    \end{align*}
    Since $\lim_{k\to +\infty} k^{r - 1} = 0$, and according to \eqref{eq:integral 1}, \eqref{eq:p(k) has a limit} and Theorem \ref{th:big O rates}, the last line has a limit as $k\to +\infty$. We use Lemma \ref{lem:q discrete} to ensure the existence of $\lim_{k\to +\infty} q_{k}$. Going back to the definition of $q_{k}$ and again using \eqref{eq:integral 1}, we finally obtain that 
    \[
        \lim_{k\to +\infty} \| z^{k} - z_{*}\| \quad \text{exists}. 
    \]
    
    \fbox{Case $r = 1$:} we similarly obtain 
    \[
        (\alpha - 1)q_{k} + k^{r}\bigl(q_{k} - q_{k - 1}\bigr) = p_{k} + \theta (\alpha - 1) \sum_{j = 1}^{k} j \beta_{j - 1} \langle z^{j} - z_{*}, V(z^{j})\rangle - \frac{1}{2}\| z^{k} - z^{k - 1}\|^{2}, 
    \]
    and with the same reasoning as before come to the existence of $\lim_{k\to +\infty} \| z^{k} - z_{*}\|$. This verifies the first condition of Opial's Lemma. 

    For the second condition, assume that $\overline{z}$ is a sequential cluster point of $\bigl( z^{k}\bigr)_{k\geq 2}$. Therefore, we have $z^{k_{n}} \rightharpoonup \overline{z}$ as $n\to +\infty$ for some subsequence $\bigl( z^{k_{n}}\bigr)_{n\geq 0}$. According to Theorem \ref{th:big O rates}, we have $V(z^{k_{n}}) \to 0$ as $n\to +\infty$. Since the graph of $V$ is sequentially closed in $\mathcal{H}^{\text{weak}} \times \mathcal{H}^{\text{strong}}$, we conclude that 
    \[
        V(\overline{z}) = 0, 
    \]
    therefore checking the second condition of Opial's Lemma and thus ending the proof for this theorem. 
\end{proof}

\appendix

\section{Auxiliary results}

\begin{lemma}\label{lem:auxiliary inequalities}
    For $k\geq 1$, $r\in [0, 1]$ and $\sigma \leq 0$, the following hold: 
    \begin{align}
        (k + 1)^{r} - k^{r} &\leq rk^{r - 1}; \tag{A1}\label{eq:A3} \\
        (k + 1)^{3r} - k^{3r} &\leq 3rk^{3r - 1} + 3rk^{3r - 2} + rk^{3r - 3}; \tag{A2}\label{eq:A4} \\
        (k + 1)^{2r} - k^{2r} &\leq 2rk^{2r - 1} + rk^{2r - 2}; \tag{A3}\label{eq:A1} \\
        (k + 1)^{\sigma} - k^{\sigma} &\leq \sigma k^{\sigma - 1} + \sigma(\sigma - 1) k^{\sigma - 2}; \tag{A4}\label{eq:A5} \\
        \Bigl|(k + 1)^{\sigma} - k^{\sigma} \Bigr| &\leq |\sigma| k^{\sigma - 1}; \tag{A5}\label{eq:A6} \\
        \Bigl|2rk^{2r - 1} - \Bigl[ (k + 1)^{2r} - k^{2r}\Bigr]\Bigr| &\leq 2r |2r - 1| k^{2r - 2}. \tag{A6} \label{eq:A2}
    \end{align}
\end{lemma}
\begin{proof}
    We first show \eqref{eq:A1}. According to the mean value theorem, for some $\xi \in \Bigl( k^{2}, (k + 1)^{2}\Bigr)$ we have 
    \[
        (k + 1)^{2r} - k^{2r} = \left((k + 1)^{2}\right)^{r} - \left( k^{2}\right)^{r} = r\xi^{r - 1} \Bigl[ (k + 1)^{2} - k^{2}\Bigr] = r\xi^{r - 1} \bigl(2k + 1\bigr). 
    \]
    Notice that 
    \[
        k^{2} < \xi < (k + 1)^{2} \quad \Rightarrow \quad \frac{1}{(k + 1)^{2}} < \frac{1}{\xi} < \frac{1}{k^{2}} \quad \Rightarrow \quad \frac{1}{(k + 1)^{2(1 - r)}} \leq \frac{1}{\xi^{1 - r}} \leq \frac{1}{k^{2(1 - r)}}
    \]
    and so 
    \[
        r\xi^{r - 1} \bigl(2k + 1\bigr) \leq \frac{r \bigl(2k + 1\bigr)}{k^{2(1 - r)}} = 2rk^{2r - 1} + rk^{2r - 2}. 
    \] 
    The inequalities \eqref{eq:A3} and \eqref{eq:A4} are shown in the same way. For \eqref{eq:A2}, there exists $\mu\in (k , k + 1)$ such that $(k + 1)^{2r} - k^{2r} = 2r \mu^{2r - 1}$. Again applying the mean value theorem, there exists $\xi \in [k, \mu]$ such that
    \[
        \left\lvert 2rk^{2r - 1} - \left[ (k + 1)^{2r} - k^{2r} \right] \right\rvert = \bigl|2rk^{2r - 1} - 2r \mu^{2r - 1}\bigr| = 2r |2r - 1| \xi^{2r - 2} (\mu - k) \leq 2r |2r - 1|k^{2r - 2}, 
    \]
    since $\mu - k \leq (k + 1) - k = 1$ and $2r - 2 \leq 0$. For showing \eqref{eq:A5}, again write, for some $\xi \in (k, k + 1)$, 
    \begin{equation}\label{eq:auxiliary inequality with sigma}
        (k + 1)^{\sigma} - k^{\sigma} = \sigma \xi^{\sigma - 1}. 
    \end{equation}
    Arguing similarly as before, we have $ (k + 1)^{\sigma - 1} \leq \xi^{\sigma - 1} \leq k^{\sigma - 1}$, thus 
    \[
        \sigma k^{\sigma - 1} \leq \sigma \xi^{\sigma - 1} \leq \sigma (k + 1)^{\sigma - 1} = \sigma \Bigl[ (k + 1)^{\sigma - 1} - k^{\sigma - 1}\Bigr] + \sigma k^{\sigma - 1} \leq \sigma k^{\sigma - 1} + \sigma (\sigma - 1)k^{\sigma - 2}, 
    \]
    where the $\sigma(\sigma - 1)k^{\sigma - 2}$ is obtained following the same reasoning. Inequality \eqref{eq:A6} is obtained by taking the absolute value of both sides of \eqref{eq:auxiliary inequality with sigma}. 
\end{proof}

The following elementary result is used several times in the paper.
\begin{lemma}
	\label{lem:quad}
	Let $A, B, C \in \R$ be such that $A \neq 0$ and $B^{2} - AC \leq 0$.
	The following statements are true:
	\begin{enumerate}[\rm (i)]
		\item
		\label{quad:vec-pos}
		if $A > 0$, then it holds
		\begin{equation*}
			A \left\lVert X \right\rVert ^{2} + 2B \left\langle X , Y \right\rangle + C \left\lVert Y \right\rVert ^{2} \geq 0 \quad \forall X, Y \in \sH ;
		\end{equation*}
		
		\item
		\label{quad:vec}
		if $A < 0$, then it holds
		\begin{equation*}
			A \left\lVert X \right\rVert ^{2} + 2B \left\langle X , Y \right\rangle + C \left\lVert Y \right\rVert ^{2} \leq 0 \quad \forall X, Y \in \sH .
		\end{equation*}
	\end{enumerate}
\end{lemma}

Regarding the following lemma, the case where $r = 1$ has already appeared in the literature, but we haven't found a proof when $r \in [0, 1)$. For the sake of completeness, we provide a proof for $r\in [0, 1]$.  

\begin{lemma}\label{lem:q}
    Let $a > 0, r\in [0, 1]$ and $q \colon [t_{0}, +\infty) \to \R$ be a continuously differentiable function such that 
    \[
        \lim_{t\to +\infty} \left( q(t) + \frac{t^{r}}{a} \Dot{q}(t)\right) = \ell\in \R. 
    \]
    Then it holds $\lim_{t\to +\infty} q(t) = \ell$. 
\end{lemma}
\begin{proof}
    By considering $q(\cdot) - \ell$ if necessary, we may assume that $\ell = 0$, i.e., we need to show that $\lim_{t\to +\infty} q(t) = 0$. Define, for $t\geq t_{0}$, 
    \[
        F(t):= 
        \begin{dcases}
            \frac{1}{a} \exp\left(\frac{a}{1 - r} t^{1 - r}\right) & r \in [0, 1)\\
            t^{a} & r = 1.
        \end{dcases}
    \]
    Therefore, for every $t\geq t_{0}$ we have
    \[
        \Dot{F}(t) = 
        \begin{dcases}
            \frac{1}{t^{r}} \exp\left(\frac{a}{1 - r} t^{1 - r}\right) & r\in [0, 1) \\
            a t^{a - 1} & r = 1. 
        \end{dcases}
    \]
    In any case, $\Dot{F}(t) > 0$ for all $t\geq t_{0}$, and it holds 
    \[
        F(t) = \frac{t^{r}}{a} \Dot{F}(t), 
    \]
    which means that 
    \begin{equation}\label{eq: derivative of F(t)q(t)}
        \Dot{F}(t) \left(q(t) + \frac{t^{r}}{a} \Dot{q}(t)\right) = \Dot{F}(t) q(t) + \frac{t^{r}}{a} \Dot{F}(t) \Dot{q}(t) = \Dot{F}(t) q(t) + F(t)\Dot{q}(t) = \frac{d}{dt}\bigl(F(t)q(t)\bigr).
    \end{equation}
    By assumption we have $q(t) + \frac{t^{r}}{a}\Dot{q}(t) \to 0$ as $t\to +\infty$. Take $\varepsilon > 0$. Then, there exists $T \geq t_{0}$ such that 
    \[
        \left| q(t) + \frac{t^{r}}{a}\Dot{q}(t)\right| \leq \varepsilon \quad \forall t\geq T.
    \]
    Using \eqref{eq: derivative of F(t)q(t)}, for $t\geq T$ it holds
    \[
        \left| \frac{d}{dt}\bigl( F(t)q(t)\bigr)\right| = \left| \Dot{F}(t)\left( q(t) + \frac{t^{r}}{a} \Dot{q}(t)\right)\right| \leq \varepsilon \Dot{F}(t), 
    \]
    so integration from $T$ to $t\geq T$ yields
    \begin{align*}
        \bigl| F(t)q(t) - F(T)q(T)\bigr| &= \left| \int_{T}^{t} \frac{d}{ds} \bigl(F(s)q(s)\bigr) ds\right| \leq \int_{T}^{t} \left| \frac{d}{ds} \bigl(F(s)q(s)\bigr)\right| ds \\
        &\leq \varepsilon \int_{T}^{t} \Dot{F}(s)ds = \varepsilon \bigl(F(t) - F(T)\bigr). 
    \end{align*}
    It follows that 
    \[
        \bigl| F(t)q(t)\bigr| \leq \bigl| F(t)q(t) - F(T)q(T)\bigr| + \bigl|F(T)q(T)\bigr| \leq \varepsilon \bigl( F(t) - F(T)\bigr) + \bigl| F(T)q(T)\bigr|, 
    \]
    from which we deduce
    \[
        |q(t)| \leq \varepsilon + \frac{F(T) \bigl( |q(T)| - \varepsilon\bigr)}{F(t)} \quad \forall t\geq T. 
    \]
    Now, we use the fact that $F(t) \to +\infty$ as $t\to +\infty$ to finally obtain 
    \[
        \limsup_{t\to +\infty} |q(t)| \leq \varepsilon. 
    \]
    Since $\varepsilon > 0$ was arbitrary, the desired result is shown. 
\end{proof}

The following result is the discrete counterpart of the previous lemma. Again, the case $r = 1$ has already been previously addressed, but we found no proof when $r\in [0, 1)$. We provide a complete proof. 

\begin{lemma}\label{lem:q discrete}
    Let $a > 0$, $r\in [0, 1]$ and let $(q_{k})_{k\geq 1}$ be a sequence of real numbers such that 
    \[
        \lim_{k\to +\infty} \left[ q_{k + 1} + \frac{k^{r}}{a}\bigl(q_{k + 1} - q_{k}\bigr)\right] = \ell \in \R. 
    \]
    Then it holds $\lim_{k\to +\infty} q_{k} = \ell$. 
\end{lemma}
\begin{proof}
    By taking $q_{k} - \ell$ instead of $q_{k}$, we may assume w.l.o.g. that $\ell = 0$. First of all, we recall a small auxiliary result. For any sequence $(a_{n})_{n\geq 1}$ of nonnegative numbers, it holds 
    \begin{equation}\label{eq:aux prod sum inequality}
        \prod_{k = 1}^{n} (1 + a_{k})\geq 1 + \sum_{k = 1}^{n}a_{k}. 
    \end{equation}
    Now, define $F_{1} = 1$, and inductively, 
    \[
        F_{k + 1} = \left(1 + \frac{a}{k^{r}}\right) F_{k}. 
    \]
    We readily see that for every $n\geq 1$, the previous definition together with \eqref{eq:aux prod sum inequality} entail
    \[
        F_{n + 1} = \prod_{k = 1}^{n}\left(1 + \frac{a}{k^{r}}\right) \geq 1 + a\sum_{k = 1}^{n}\frac{1}{k^{r}} \geq 1 + a\sum_{k = 1}^{n}\frac{1}{k},
    \]
    and since the right-hand side grows to $+\infty$ as $n\to +\infty$, we have $\lim_{n\to +\infty}F_{n} = +\infty$. Furthermore, notice that $(F_{k})_{k\geq 1}$ is increasing, since $F_{k + 1} = F_{k} + \frac{a}{k^{r}} F_{k} \geq F_{k}$. Additionally, we have 
    \[
        F_{k} = \frac{k^{r}}{a} \bigl(F_{k + 1} - F_{k}\bigr), 
    \]
    which gives
    \begin{align}
        \bigl(F_{k + 1} - F_{k}\bigr) \left[ q_{k + 1} + \frac{k^{r}}{a}\bigl(q_{k + 1} - q_{k}\bigr)\right] \nonumber
        &= \bigl(F_{k + 1} - F_{k}\bigr)q_{k + 1} + \frac{k^{r}}{a}\bigl( F_{k + 1} - F_{k}\bigr) \bigl(q_{k + 1} - q_{k}\bigr) \nonumber\\
        &= \bigl( F_{k + 1} - F_{k}\bigr)q_{k + 1} + F_{k} \bigl(q_{k + 1} - q_{k}\bigr) \nonumber\\
        &= F_{k + 1}q_{k + 1} - F_{k}q_{k}. \label{eq:derivative of F(k)q(k)}
    \end{align}
    Let $\varepsilon > 0$. Then, there exists $k_{0}\geq 1$ such that for $k\geq k_{0}$ it holds
    \[
        \left| q_{k + 1} + \frac{k^{r}}{a}\bigl(q_{k + 1} - q_{k}\bigr)\right| \leq \varepsilon. 
    \]
    Multiplying both sides by $F_{k + 1} - F_{k} \geq 0$ and using \eqref{eq:derivative of F(k)q(k)} yields
    \[
        \bigl| F_{k + 1}q_{k + 1} - F_{k}q_{k}\bigr| = \bigl(F_{k + 1} - F_{k}\bigr)\left| q_{k + 1} + \frac{k^{r}}{a}\bigl(q_{k + 1} - q_{k}\bigr)\right| \leq \varepsilon \bigl( F_{k + 1} - F_{k}\bigr). 
    \]
    Summing the previous inequality from $k_{0}$ to $n - 1\geq k_{0}$ leads to 
    \begin{align*}
        \bigl| F_{n}q_{n} - F_{k_{0}}q_{k_{0}}\bigr| &= \left| \sum_{k = k_{0}}^{n - 1} \bigl( F_{k + 1}q_{k + 1} - F_{k}q_{k}\bigr)\right| \leq \sum_{k = k_{0}}^{n - 1} \bigl| F_{k + 1}q_{k + 1} - F_{k}q_{k}\bigr| \\
        &\leq \varepsilon \sum_{k = k_{0}}^{n - 1} \bigl( F_{k + 1} - F_{k}\bigr) = \varepsilon\bigl( F_{n} - F_{k_{0}}\bigr).
    \end{align*}
    It follows that for $n\geq k_{0} + 1$ we have 
    \[
        |F_{n}q_{n}| \leq \bigl| F_{n}q_{n} - F_{k_{0}}q_{k_{0}}\bigr| + |F_{k_{0}}q_{k_{0}}| \leq \varepsilon \bigl( F_{n} - F_{k_{0}}\bigr) + |F_{k_{0}}q_{k_{0}}|
    \]
    and therefore
    \[
        |q_{n}| \leq \varepsilon + \frac{F_{k_{0}} \bigl( |q_{k_{0}}| - \varepsilon\bigr)}{F_{n}}.
    \]
    Since we have already established that $F_{n} \to  +\infty$ as $n\to +\infty$, we come to 
    \[
        \limsup_{n\to +\infty} |q_{n}| \leq \varepsilon. 
    \]
    Since $\varepsilon > 0$ was arbitrary, we conclude that $\lim_{n\to +\infty} q_{n} = 0$. 
\end{proof}

The proof for the following lemma can be found in \cite{Opial}. 
\begin{lemma}[Opial's Lemma] \label{lem:Opial}
    Let $\mathcal{H}$ be a real Hilbert space, $S \subseteq \mathcal{H}$ a nonempty set, $t_{0} > 0$ and $z : \left[ t_{0}, +\infty \right) \to \mathcal{H}$ a mapping that satisfies
    \begin{enumerate}[\rm (i)]
        \item for every $z_{*} \in S$, $\lim_{t\to +\infty} \|z(t) - z_{*}\|$ exists;
        \item every weak sequential cluster point of the trajectory $z(t)$ as $t \to +\infty$ belongs to $S$. 
    \end{enumerate}
    Then, $z(t)$ converges weakly to an element of $S$ as $t \to +\infty$.
\end{lemma}

\section{Proof of Lemma \ref{lem:coefs of <z(k+1)-z(k),V(z(k+1))> and |V(z(k+1))|}}
For the sake of readability, we include the proof here. 
\begin{proof}[Proof of Lemma \ref{lem:coefs of <z(k+1)-z(k),V(z(k+1))> and |V(z(k+1))|}]
    (i) The coefficient reads
    \begin{align*}
        &\: 2\lambda^{2} \Bigl[ (k + 1)^{2(\rho - r)} - k^{2(\rho - r)}\Bigr] + \overbrace{2\lambda (2\rho - r) k^{2(\rho - r)}\Bigl[ k^{r - 1} - (k + 1)^{r - 1}\Bigr]}^{=: O_{k}} \\
        &\quad\quad + 2\lambda \Bigl[ (k + 1)^{2(\rho - r)} - k^{2(\rho - r)}\Bigr]\bigl( \alpha - (2\rho - r)(k + 1)^{r - 1} - \lambda\bigr). 
    \end{align*}
    After dropping the nonpositive term of the second line and using \eqref{eq:A5} and \eqref{eq:A6}, the previous quantity is less or equal to 
    \begin{align*}
        &\: 2\lambda^{2} \Bigl[ (k + 1)^{2(\rho - r)} - k^{2(\rho - r)}\Bigr] + |O_{k}| \\
        \leq &\: 4\lambda^{2}(\rho - r) k^{2(\rho - r) - 1} + \Bigl\{ 4\lambda^{2}(\rho - r)(2(\rho - r) - 1)k^{2(\rho - r) - 2} + 2\lambda |2\rho - r| (1 - r) k^{2(\rho - r) + r - 2} \Bigr\}. 
    \end{align*}

    (ii) Recall the term in question reads 
    \begin{align*}
        2P_{k} = &\:8\lambda(r - \rho) k^{2(\rho - r) - 1}k^{r} + 4\lambda \Bigl[ (k + 1)^{2(\rho - r)} - k^{2(\rho - r)}\Bigr](k + 1)^{r} \\
        = &\: 4\lambda k^{r} \Bigl\{ 2(r - \rho)k^{2(\rho - r) - 1} \! - \Bigl[k^{2(\rho - r)} \! - (k \! + \! 1)^{2(\rho - r)}\Bigr]\Bigr\} \! + \! 4\lambda \Bigl[ (k \! + \! 1)^{2(\rho - r)}\! - k^{2(\rho - r)}\Bigr] \! \Bigl[ (k \! + \! 1)^{r} \! - k^{r}\Bigr], 
    \end{align*}
    which means that 
    \[
        |P_{k}| \leq  2\lambda k^{r} \cdot 2(r - \rho) \bigl(1 + 2(r - \rho)\bigr) k^{2(\rho - r) - 2} + 2\lambda \cdot 2(r - \rho) k^{2(\rho - r) - 1} \cdot r k^{r - 1}, 
    \]
    and the right-hand side is of order $\mathcal{O}(k^{2\rho - r - 2})$ as $k\to +\infty$.

    (iii) We can rewrite the term as  
    \begin{align}
        \eta_{k} &= 2k^{r} \Bigl[\bigl(2r\theta k^{r - 1} - 1\bigr)\beta_{k} + \theta k^{r} \bigl( \beta_{k} - \beta_{k - 1}\bigr)\Bigr] - \theta \Bigl[(k + 1)^{2r}\beta_{k} - k^{2r}\beta_{k - 1}\Bigr] \nonumber\\
        &= 2k^{r} \bigl(2r\theta k^{r - 1} - 1\bigr)\beta_{k} + 2\theta k^{2r} \bigl(\beta_{k} - \beta_{k - 1}\bigr) - \theta \Bigl\{ \Bigl[ (k + 1)^{2r} - k^{2r}\Bigr]\beta_{k} + k^{2r} \bigl( \beta_{k} - \beta_{k - 1}\bigr)\Bigr\} \nonumber\\
        &= 2k^{r} \bigl(2r\theta k^{r - 1} - 1\bigr)\beta_{k} + \theta k^{2r} \bigl( \beta_{k} - \beta_{k - 1}\bigr) - \theta \Bigl[ (k + 1)^{2r} - k^{2r}\Bigr] \beta_{k}. \label{eq:another formula for eta(k)}
    \end{align}
    According to our assumptions, for large enough $k$ we have $|2r\theta k^{r - 1} - 1| = 1 - 2r\theta k^{r - 1} \leq 1$. Using this, together with \eqref{eq:A1} and \eqref{eq:G2}, yields
    \begin{equation*}
        |\eta_{k}| \leq 2k^{r} \beta_{k} + k^{r} \bigl(1 - 2r\theta k^{r - 1} - \delta\theta\bigr) \beta_{k} + \bigl(2r\theta k^{2r - 1} + r\theta k^{2r - 2}\bigr)\beta_{k}. 
    \end{equation*}
    Since $2r - 2 \leq 2r - 1 \leq r$, the previous inequality tells us that 
    \begin{equation}\label{eq:order of eta(k)}
        |\eta_{k}| \leq q_{k} \beta_{k}, 
    \end{equation}
    where $q_{k} = \mathcal{O}(k^{r})$ as $k\to +\infty$. We rewrite the term accompanying $\bigl\langle z^{k + 1} - z^{k}, V(z^{k + 1})\bigr\rangle$ as 
    \begin{align}
        & \: k^{2(\rho - r)} \Bigl\{ 2\eta_{k} (k + 1)^{r} + 2\theta \bigl( \lambda + rk^{r - 1} - \alpha\bigr)(k + 1)^{2r} \beta_{k} - 2\bigl( \lambda + rk^{r - 1} - \alpha\bigr)\eta_{k} + 2\lambda\theta k^{2r} \beta_{k - 1}\Bigr\} \label{eq:coef of <z(k+1)-z(k),V(z(k+1))> rewritten line 1}\\
        &+ 2\theta \Bigl[ (k + 1)^{2(\rho - r)} - k^{2(\rho - r)}\Bigr] (k + 1)^{3r} \beta_{k}. \label{eq:coef of <z(k+1)-z(k),V(z(k+1))> rewritten line 2}
    \end{align}
    We first focus on \eqref{eq:coef of <z(k+1)-z(k),V(z(k+1))> rewritten line 1}, namely, the term inside curly brackets: 
    \begin{align}
        &\:2\eta_{k} (k + 1)^{r} + 2\theta \bigl( \lambda + rk^{r - 1} - \alpha\bigr)(k + 1)^{2r} \beta_{k} - 2\bigl( \lambda + rk^{r - 1} - \alpha\bigr)\eta_{k} + 2\lambda\theta k^{2r} \beta_{k - 1} \nonumber\\
        = &\: 2\Bigl[ (k + 1)^{r} - k^{r} - \bigl( \lambda + rk^{r - 1} - \alpha\bigr)\Bigr]\eta_{k} + 2k^{r} \eta_{k} \nonumber\\
        &+ 2\theta \Bigl[ (k + 1)^{2r} - k^{2r}\Bigr] \bigl( \lambda + rk^{r - 1} - \alpha\bigr) \beta_{k} + 2\theta k^{2r} \bigl( \lambda + rk^{r - 1} - \alpha\bigr) \beta_{k} \nonumber\\
        &+ 2\lambda \theta k^{2r} \bigl(\beta_{k - 1} - \beta_{k}\bigr) + 2\lambda\theta k^{2r} \beta_{k} \nonumber\\
        = &\: 2 k^{r} \eta_{k} + 2\theta k^{2r} \bigl( \lambda + rk^{r - 1} - \alpha\bigr) \beta_{k} + 2\lambda\theta  k^{2r} \beta_{k} \nonumber\\
        &+ 2\Bigl[ (k + 1)^{r} - k^{r} - \bigl( \lambda + rk^{r - 1} - \alpha\bigr)\Bigr]\eta_{k} + 2\theta \Bigl[ (k + 1)^{2r} - k^{2r}\Bigr] \bigl( \lambda + rk^{r - 1} - \alpha\bigr) \beta_{k} \nonumber\\
        &+ 2\lambda \theta k^{2r} \bigl(\beta_{k - 1} - \beta_{k}\bigr) \nonumber\\
        = &\: 2k^{r} \Bigl\{2k^{r} \bigl(r\theta k^{r - 1} - 1\bigr)\beta_{k} + 2r\theta k^{2r - 1}\beta_{k} + \theta k^{2r} \bigl(\beta_{k} - \beta_{k - 1}\bigr) - \theta \Bigl[ (k + 1)^{2r} - k^{2r}\Bigr]\beta_{k}\Bigr\} \nonumber\\
        &+ 2\theta k^{2r} \bigl(\lambda + rk^{r - 1} - \alpha\bigr)\beta_{k} + 2\lambda\theta k^{2r} \beta_{k} \nonumber\\
        &+ 2\Bigl[ (k + 1)^{r} - k^{r} - \bigl( \lambda + rk^{r - 1} - \alpha\bigr)\Bigr]\eta_{k} + 2\theta \Bigl[ (k + 1)^{2r} - k^{2r}\Bigr] \bigl( \lambda + rk^{r - 1} - \alpha\bigr) \beta_{k} \nonumber\\
        &+ 2\lambda \theta k^{2r} \bigl(\beta_{k - 1} - \beta_{k}\bigr) \nonumber\\
        = &\: 2k^{r} \Bigl[2k^{r} \bigl(r\theta k^{r - 1} - 1\bigr)\beta_{k} + \theta k^{2r} \bigl( \beta_{k} - \beta_{k - 1}\bigr)\Bigr] + 2\theta k^{2r} \bigl( \lambda + rk^{r - 1} - \alpha\bigr)\beta_{k} + 2\lambda\theta k^{2r} \beta_{k} \nonumber\\
        &+ 2\theta k^{r} \Bigl\{ 2rk^{2r - 1} - \Bigl[(k + 1)^{2r} - k^{2r}\Bigr]\Bigr\} \beta_{k}\nonumber\\
        &+ 2\Bigl[ (k + 1)^{r} - k^{r} - \bigl( \lambda + rk^{r - 1} - \alpha\bigr)\Bigr]\eta_{k} + 2\theta \Bigl[ (k + 1)^{2r} - k^{2r}\Bigr] \bigl( \lambda + rk^{r - 1} - \alpha\bigr) \beta_{k} \nonumber\\
        &+ 2\lambda \theta k^{2r} \bigl(\beta_{k - 1} - \beta_{k}\bigr) \nonumber\\
        = &\: 2k^{2r} \Bigl[2\bigl(r\theta k^{r - 1} - 1\bigr)\beta_{k} + \theta k^{r} \bigl(\beta_{k} - \beta_{k - 1}\bigr) + \theta \bigl(\lambda + rk^{r - 1} - \alpha\bigr) \beta_{k} + \lambda\theta \beta_{k}\Bigr] \label{eq:coef of <z(k+1)-z(k),V(z(k+1))> line 1}\\
        &+ 2\theta k^{r} \Bigl\{ 2rk^{2r - 1} - \Bigl[(k + 1)^{2r} - k^{2r}\Bigr]\Bigr\} \beta_{k} + 2\Bigl[ (k + 1)^{r} - k^{r} - \bigl( \lambda + rk^{r - 1} - \alpha\bigr)\Bigr]\frac{\eta_{k}}{\beta_{k}}\cdot \beta_{k} \label{eq:coef of <z(k+1)-z(k),V(z(k+1))> line 2}\\
        &+ 2\theta \Bigl[ (k + 1)^{2r} - k^{2r}\Bigr] \bigl( \lambda + rk^{r - 1} - \alpha\bigr) \beta_{k} + 2\lambda \theta k^{2r} \frac{\beta_{k - 1} - \beta_{k}}{\beta_{k}} \cdot \beta_{k} \label{eq:coef of <z(k+1)-z(k),V(z(k+1))> line 3}
    \end{align}
    Line \eqref{eq:coef of <z(k+1)-z(k),V(z(k+1))> line 1} reads 
    \begin{align*}
        &\: 2k^{2r} \Bigl[\bigl(2r\theta k^{r - 1} - 2 + \lambda\theta + \theta r k^{r - 1} - \theta\alpha + \lambda\theta \bigr)\beta_{k} + \theta k^{r} \bigl(\beta_{k} - \beta_{k - 1}\bigr)\Bigr] \\
        = &\: 2k^{2r} \Bigl\{ \Bigl[-2\theta \bigl(\alpha - rk^{r - 1} - \lambda\bigr) + \bigl(\theta\alpha + \theta rk^{r - 1} - 2\bigr)\Bigr] \beta_{k} + \theta k^{r}\bigl( \beta_{k} - \beta_{k - 1}\bigr)\Bigr\}.
    \end{align*}
    For lines \eqref{eq:coef of <z(k+1)-z(k),V(z(k+1))> line 2} and \eqref{eq:coef of <z(k+1)-z(k),V(z(k+1))> line 3}, we factor out $2\beta_{k}$ and write these lines as $2Q_{1, k}\beta_{k}$. We use \eqref{eq:A2}, \eqref{eq:A3}, the fact that $\bigl| \lambda + rk^{r - 1} - \alpha\bigr| = \alpha - rk^{r - 1} - \lambda \leq \alpha$ for large enough $k$, \eqref{eq:order of eta(k)}, \eqref{eq:A1} and \eqref{eq:G2} to obtain 
    \begin{align*}
        |Q_{1, k}| \leq &\: \theta k^{r} \Bigl| 2rk^{2r - 1} - \Bigl[ (k + 1)^{2r} - k^{2r}\Bigr]\Bigr| + \Bigl| (k + 1)^{r} - k^{r} - \bigl(\lambda + rk^{r - 1} - \alpha\bigr)\Bigr|\left| \frac{\eta_{k}}{\beta_{k}}\right| \\ 
        &+ \theta \Bigl| (k + 1)^{2r} - k^{2r}\Bigr| \bigl| \lambda + rk^{r - 1} - \alpha\bigr| + \lambda\theta k^{2r} \left| \frac{\beta_{k - 1} - \beta_{k}}{\beta_{k}}\right| \\
        \leq &\: 2r|2r - 1| \theta k^{3r - 2} + \bigl|rk^{r - 1} + \alpha - rk^{r - 1} - \lambda\bigr| q_{k} \\
        &+ \bigl(2r\theta k^{2r - 1} + r\theta k^{2r - 2}\bigr) \bigl(\alpha - rk^{r - 1} - \lambda\bigr) + \lambda k^{r} \bigl(1 - 2r\theta k^{r - 1} - \delta\theta\bigr) \\
        \leq &\: 2r|2r - 1| \theta k^{3r - 2} + (\alpha - \lambda) q_{k} + \alpha \bigl(2r\theta k^{2r - 1} + r\theta k^{2r - 2}\bigr) + \lambda k^{r} \bigl(1 - \delta\theta\bigr). 
    \end{align*}
    The right-hand side is of order $\mathcal{O}(k^{r})$ as $k\to +\infty$. Now, line \eqref{eq:coef of <z(k+1)-z(k),V(z(k+1))> rewritten line 2}: we write it as $2 k^{2(\rho - r)} Q_{2, k}$ and we have, according to \eqref{eq:A6}, 
    \begin{align*}
        |Q_{2, k}| &\leq \frac{\theta}{k^{2(\rho - r)}}\Bigl| (k + 1)^{2(\rho - r)} - k^{2(\rho - r)}\Bigr| (k + 1)^{3r} \leq \frac{1}{k^{2(\rho - r)}} \cdot 2\theta (r - \rho) k^{2(\rho - r) - 1} (k + 1)^{3r} \\
        &\leq 2\theta(r - \rho) k^{-1} \bigl[ k^{3r} + 3k^{2r} + 3k^{r} + 1\bigr]
    \end{align*}
    and this last term is of order $\mathcal{O}\left(k^{3r - 1}\right)$ as $k\to +\infty$. 

    (iv) First, notice that 
    \begin{align*}
        (k + 1)^{2\rho} = &\: (k + 1)^{2\rho} - k^{2(\rho - r)}(k + 1)^{2r} + k^{2(\rho - r)}(k + 1)^{2r} \\
        = &\: (k + 1)^{2r} \Bigl[ (k + 1)^{2(\rho - r)} - k^{2(\rho - r)}\Bigr] + k^{2(\rho - r)}(k + 1)^{2r} \\
        \leq &\: k^{2(\rho - r)}(k + 1)^{2r}.
    \end{align*}
    It follows that the coefficient attached to $\| V(z^{k + 1})\|^{2}$ is less or equal than 
    \begin{align}
        &\:k^{2(\rho - r)} \left[ \theta(k + 1)^{2r} \eta_{k} \beta_{k} - \frac{1}{2}\eta_{k}^{2}\right] + \frac{\theta^{2}}{2} \Bigl[ (k + 2)^{2r} k^{2(\rho - r)}(k + 1)^{2r}\beta_{k + 1}\beta_{k} - (k + 1)^{2r} k^{2\rho} \beta_{k} \beta_{k - 1}\Bigr] \nonumber\\
        &+ \frac{\theta^{2}}{2} \Bigl[ (k + 1)^{2(\rho - r)} - k^{2(\rho - r)}\Bigr](k + 1)^{4r} \beta_{k}^{2} \nonumber\\
        \leq &\: k^{2(\rho - r)} \left\{ \theta (k + 1)^{2r} \eta_{k} \beta_{k} - \frac{1}{2} \eta_{k}^{2} + \frac{\theta^{2}}{2} \Bigl[ (k + 2)^{2r} (k + 1)^{2r} \beta_{k + 1} \beta_{k} - (k + 1)^{2r} k^{2r} \beta_{k} \beta_{k - 1}\Bigr]\right\} \nonumber\\ 
        \leq &\: k^{2(\rho - r)} \left\{ \theta k^{2r} \beta_{k} \eta_{k} + \frac{\theta^{2}}{2} (k + 1)^{2r} \beta_{k} \Bigl[ (k + 2)^{2r} \beta_{k + 1} - k^{2r} \beta_{k - 1}\Bigr]\right\}, \label{eq:first inequality for coef of |V(z(k+1))|}
    \end{align}
    where we used the fact that for large enough $k$, $\eta_{k} \leq 0$ (a consequence of \eqref{eq:G1}) and we dropped the nonpositive term $-\frac{1}{2}\eta_{k}^{2}$. We focus on the term inside curly brackets and address the cases $r\in (0, 1)$ and $r = 1$ separately. 
    
    \fbox{Case $r\in (0, 1)$:} since the inequality \eqref{eq:supremum < than 1/theta - delta} is strict, there exists $\tilde{\delta}$ such that $0 < \delta < \tilde{\delta} < \frac{1}{\theta}$ and such that inequalities \eqref{eq:G1}, \eqref{eq:G2} and \eqref{eq:G3} hold replacing $\delta$ by $\tilde{\delta}$. Taking into account \eqref{eq:another formula for eta(k)}, the first of summand of \eqref{eq:first inequality for coef of |V(z(k+1))|} reads
    \begin{align}
        \theta k^{2r} \beta_{k} \eta_{k} &= \theta k^{2r} \beta_{k}\Bigl\{2k^{r} \bigl(2r\theta k^{r - 1} - 1\bigr)\beta_{k} + \theta k^{2r} \bigl(\beta_{k} - \beta_{k - 1}\bigr) - \theta \Bigl[ (k + 1)^{2r} - k^{2r}\Bigr]\beta_{k}\Bigr\} \nonumber\\
        &\leq \theta k^{2r} \beta_{k}\Bigl[ 2 k^{r} \bigl( 2r\theta k^{r - 1} - 1\bigr) \beta_{k} + k^{r} \bigl(1 - 2r\theta k^{r - 1} - \tilde{\delta}\theta\bigr)\beta_{k}\Bigr] \nonumber\\
        &= \theta k^{2r} \beta_{k}\Bigl[ k^{r} \bigl(2r\theta k^{r - 1} - 1\bigr)\beta_{k} - \tilde{\delta}\theta k^{r} \beta_{k}\Bigr] \nonumber\\
        &=  \theta \bigl(- 1 - \tilde{\delta}\theta\bigr) k^{3r} \beta_{k}^{2} + 2r\theta^{2} k^{4r - 1} \beta_{k}^{2}. \label{eq:first summand |V(z(k+1))| case r<1} 
    \end{align}
    Moving on to the second summand of \eqref{eq:first inequality for coef of |V(z(k+1))|}, we have 
    \begin{align}
        &\: (k + 2)^{2r} \beta_{k + 1} - k^{2r} \beta_{k - 1} \nonumber\\
        = &\: (k + 2)^{2r} \beta_{k + 1} - (k + 1)^{2r} \beta_{k} + (k + 1)^{2r} \beta_{k} - k^{2r} \beta_{k - 1} \nonumber\\
        = &\: \Bigl[ (k + 2)^{2r} - (k + 1)^{2r}\Bigr] \beta_{k + 1} + (k + 1)^{2r} \bigl(\beta_{k + 1} - \beta_{k}\bigr) + \Bigl[ (k + 1)^{2r} - k^{2r}\Bigr] \beta_{k} + k^{2r} \bigl( \beta_{k} - \beta_{k - 1}\bigr) \nonumber\\
        \leq &\: \bigl(2r(k + 1)^{2r - 1} + r(k + 1)^{2r - 2}\bigr) \beta_{k + 1} + (k + 1)^{r} \left(\frac{1}{\theta} - 2r (k + 1)^{r - 1} - \tilde{\delta}\right) \beta_{k + 1} \nonumber\\
        &+ \bigl( 2r k^{2r - 1} + r k^{2r - 2}\bigr) \beta_{k} + k^{r} \left(\frac{1}{\theta} - 2rk^{r - 1} - \tilde{\delta}\right) \beta_{k} \nonumber\\
        \leq & \: \Bigl[ M_{\beta} \bigl( 2r (k + 1)^{2r - 1} + r(k + 1)^{2r - 2}\bigr) + 2r k^{2r - 1} + rk^{2r - 2}\Bigr] \beta_{k} \nonumber\\
        &+ (k + 1)^{r} \left(\frac{1}{\theta} - \tilde{\delta}\right) M_{\beta}\beta_{k} + k^{r} \left( \frac{1}{\theta} - \tilde{\delta}\right) \beta_{k} \nonumber\\
        \leq &\: \left\{ M_{\beta} \bigl( 2r (k + 1)^{2r - 1} + r(k + 1)^{2r - 2}\bigr) + 2r k^{2r - 1} + rk^{2r - 2} + M_{\beta} \left( \frac{1}{\theta} - \tilde{\delta}\right)\right\} \beta_{k} \nonumber\\
        &+ (1 + M_{\beta}) \left( \frac{1}{\theta} - \tilde{\delta}\right) k^{r} \beta_{k} \nonumber\\
        = &\: r_{k} \beta_{k} + (1 + M_{\beta})\left( \frac{1}{\theta} - \tilde{\delta}\right) k^{r} \beta_{k}, \label{eq:second summand |V(z(k+1))| case r<1}
    \end{align}
    where $r_{k}$ is the term between curly brackets. Notice that $r_{k} = \mathcal{O}\bigl(k^{\max\{ 2r - 1, 0\}}\bigr)$ as $k\to +\infty$. Here, we used \eqref{eq:A1}, \eqref{eq:G2}, \eqref{eq:G3}, we dropped the nonpositive terms $-2r(k + 1)^{r - 1}$ and $-2rk^{r - 1}$ and we used the fact that $(k + 1)^{r} \leq k^{r} + 1$, since $t\mapsto t^{r}$ is subadditive on $[0, +\infty)$. With \eqref{eq:first summand |V(z(k+1))| case r<1} and \eqref{eq:second summand |V(z(k+1))| case r<1} at hand, we can bound \eqref{eq:first inequality for coef of |V(z(k+1))|}:
    \begin{align*}
        &\:\theta k^{2r} \beta_{k} \eta_{k} + \frac{\theta^{2}}{2} (k + 1)^{2r} \beta_{k} \Bigl[ (k + 2)^{2r} \beta_{k + 1} - k^{2r} \beta_{k - 1}\Bigr] \\
        \leq &\: \theta \bigl(- 1 - \tilde{\delta}\theta\bigr) k^{3r} \beta_{k}^{2} + 2r\theta^{2} k^{4r - 1} \beta_{k}^{2} + \frac{\theta^{2}}{2} (k + 1)^{2r}\left[r_{k} + (1 + M_{\beta})\left(\frac{1}{\theta} - \tilde{\delta}\right)k^{r}\right] \beta_{k}^{2} \\
        \leq &\: \theta k^{3r} \left[-1 + \frac{1 + M_{\beta}}{2} - \left(1 + \frac{1 + M_{\beta}}{2}\right)\tilde{\delta}\theta\right] \beta_{k}^{2} \\
        &+ \theta^{2}\left\{ 2r k^{4r - 1} + \frac{(k + 1)^{2r}}{2} r_{k} + k^{r}\bigl(2k^{r} + 1\bigr) \frac{1 + M_{\beta}}{2} \left(\frac{1}{\theta} - \tilde{\delta}\right)\right\}\beta_{k}^{2} \\
        \leq &\: \theta k^{3r} \left[ \frac{M_{\beta} - 1}{2} - 2\tilde{\delta} \theta\right] \beta_{k}^{2} + R_{k} \beta_{k}^{2}, 
    \end{align*}
    where $R_{k}$ is the term between curly brackets multiplied by $\theta^{2}$. We used the fact that $M_{\beta} > 1$ for the bound $-\frac{1 + M_{\beta}}{2} < -1$, and again the subadditivity of $t\mapsto t^{r}$ to obtain $(k + 1)^{2r} \leq \bigl(k^{r} + 1\bigr)^{2} = k^{2r} + 2k^{r} + 1$. Since $r < 1$, we have $3r > 4r - 1$, and from here we deduce that $R_{k} = o(k^{3r})$ as $k\to +\infty$. According to Remark \ref{rem:growth condition}, $M_{\beta}$ can be taken as close to $1$ as desired, provided $k$ is large enough. In particular, it can be chosen such that $\frac{M_{\beta} - 1}{2} < 2\theta (\tilde{\delta} - \delta)$, and thus 
    \[
        \frac{M_{\beta} - 1}{2} - 2\tilde{\delta} \theta < 2\theta(\tilde{\delta} - \delta) - 2\tilde{\delta} \theta = -2\delta\theta,  
    \]
    and this is what we wanted to show. 

    \fbox{Case $r = 1$:} again, according to \eqref{eq:another formula for eta(k)}, we have 
    \begin{align*}
        \eta_{k} &= 2k (2\theta - 1)\beta_{k} + \theta k^{2} \bigl(\beta_{k} - \beta_{k - 1}\bigr) - \theta \Bigl[(k + 1)^{2} - k^{2}\Bigr] \beta_{k} \\
        &= 2k(2\theta - 1)\beta_{k} + \theta k^{2} \bigl(\beta_{k} - \beta_{k - 1}\bigr) - \theta(2k + 1)\beta_{k} \\
        &= \Bigl[2(\theta - 1)k - \theta\Bigr]\beta_{k} + \theta k^{2} \bigl(\beta_{k} - \beta_{k - 1}\bigr), 
    \end{align*}
    so using \eqref{eq:G2} yields 
    \begin{align}
        \theta k^{2} \beta_{k} \eta_{k} &\leq \theta k^{2} \beta_{k} \Bigl\{ \Bigl[ 2(\theta - 1)k - \theta\Bigr]\beta_{k} + k(1 - 2\theta - \delta\theta)\beta_{k}\Bigr\} \nonumber\\
        &= \theta k^{3} \Bigl[2(\theta - 1) + (1 - 2\theta - \delta\theta)\Bigr] \beta_{k}^{2} - \theta^{2}k^{2} \beta_{k}^{2}. \label{eq:first summand |V(z(k+1))| case r=1}
    \end{align}
    For the second summand of \eqref{eq:first inequality for coef of |V(z(k+1))|}, we write
    \begin{align}
        &\:(k + 2)^{2}\beta_{k + 1} - k^{2} \beta_{k - 1} \nonumber\\
        = &\: (k + 2)^{2} \bigl(\beta_{k + 1} - \beta_{k}\bigr) + k^{2} \bigl( \beta_{k} - \beta_{k - 1}\bigr) + \Bigl[(k + 2)^{2} - k^{2}\Bigr] \beta_{k} \nonumber\\
        = &\: \Bigl[ (k + 1)^{2} + 2k + 3\Bigr] \bigl(\beta_{k + 1} - \beta_{k}\bigr) + k^{2} \bigl(\beta_{k} - \beta_{k - 1}\bigr) + (4k + 4)\beta_{k} \nonumber\\
        \leq &\: \left(\frac{1}{\theta} - 2 - \delta\right) \Bigl[(k + 1)\beta_{k + 1} + k \beta_{k}\Bigr] + (2k + 3) \bigl(\beta_{k + 1} - \beta_{k}\bigr) + 4(k + 1)\beta_{k} \nonumber\\
        = &\: \left(\frac{1}{\theta} - 2 - \delta\right) (k + 1) \bigl(\beta_{k + 1} - \beta_{k}\bigr) + (2k + 3)\bigl(\beta_{k + 1} - \beta_{k}\bigr) + \left[\left(\frac{1}{\theta} - 2 - \delta\right)(2k + 1) + 4(k + 1)\right] \beta_{k} \nonumber\\
        \leq & \: \left(\frac{1}{\theta} - 2 - \delta\right)^{2} \beta_{k + 1} + \frac{2k + 3}{k + 1}\left(\frac{1}{\theta} - 2 - \delta\right)\beta_{k + 1} + \left[\left(\frac{1}{\theta} - 2 - \delta\right)(2k + 1) + 4(k + 1)\right] \beta_{k} \nonumber\\
        \leq &\: \left[\left(\frac{1}{\theta} - 2 - \delta\right)^{2} M_{\beta} + \frac{2k + 3}{k + 1} \left(\frac{1}{\theta} - 2 - \delta\right) M_{\beta}\right]\beta_{k} + \left[\left(\frac{1}{\theta} - 2 - \delta\right)(2k + 1) + 4(k + 1)\right] \beta_{k} \nonumber\\
        = &\: 2\left(\frac{1}{\theta} - \delta\right) k \beta_{k} + \left\{\left(\frac{1}{\theta} - 2 - \delta\right)^{2} M_{\beta} + \frac{2k + 3}{k + 1} \left(\frac{1}{\theta} - 2 - \delta\right) M_{\beta} + \frac{1}{\theta} + 2 - \delta\right\} \beta_{k} \nonumber\\
        = &\: 2\left(\frac{1}{\theta} - \delta\right) k \beta_{k} + r_{k}\beta_{k}, \label{eq:second summand |V(z(k+1))| case r=1}
    \end{align}
    where $r_{k}$ is the term between curly brackets. We have $r_{k} = \mathcal{O}\left( 1\right)$ as $k\to +\infty$. To obtain the bounds, we repeatedly used \eqref{eq:G2} and \eqref{eq:G3}. Using \eqref{eq:first summand |V(z(k+1))| case r=1} and \eqref{eq:second summand |V(z(k+1))| case r=1}, we have a bound for \eqref{eq:first inequality for coef of |V(z(k+1))|}:
    \begin{align*}
         &\:\theta k^{2r} \beta_{k} \eta_{k} + \frac{\theta^{2}}{2} (k + 1)^{2r} \beta_{k} \Bigl[ (k + 2)^{2r} \beta_{k + 1} - k^{2r} \beta_{k - 1}\Bigr] \\
         \leq &\: \theta k^{3} \Bigl[2(\theta - 1) + (1 - 2\theta - \delta\theta)\Bigr] \beta_{k}^{2} - \theta^{2}k^{2} \beta_{k}^{2} + \frac{\theta^{2}}{2}(k + 1)^{2} \beta_{k} \left[ 2\left(\frac{1}{\theta} - \delta\right)k\beta_{k} + r_{k}\beta_{k}\right] \\
         \leq &\: \theta k^{3} \Bigl[ 2(\theta - 1) + (1 - 2\theta - \delta\theta) + (1 - \delta\theta)\Bigr] \beta_{k}^{2} + \frac{\theta^{2}}{2}(2k + 1) r_{k} \beta_{k}^{2} \\
         = &\: -2\delta \theta^{2} k^{3} \beta_{k}^{2} + R_{k} \beta_{k}^{2}, 
    \end{align*}
    where $R_{k} = \frac{\theta^{2}}{2} (2k + 1) r_{k} = o(k^{3})$ as $k\to +\infty$. This concludes the proof. 
\end{proof}

\printbibliography

 
\end{document}